\documentclass[13pt]{amsbook}

\usepackage{amsfonts,amsmath,amscd,amssymb}
\usepackage{mathrsfs,color,pst-node}
\usepackage{latexsym,graphicx,verbatim,nameref,tikz-cd}
\usepackage[numbers,sort&compress]{natbib}
\usepackage[all,cmtip]{xy}
\usepackage{hyperref}
\usepackage{fancyhdr}
\usepackage{tikz}
\usepackage[toc,page]{appendix}
\usepackage{epigraph} 

\usetikzlibrary{matrix}
\makeindex

\newtheorem{theorem}{Theorem}[section]
\newtheorem{lemma}{Lemma}[section]
\newtheorem{claim}{Claim}[theorem]
\newtheorem{corollary}{Corollary}[section]
\newtheorem{proposition}{Proposition}[section]
\newtheorem{exercise}{Exercise}[chapter]

\theoremstyle{definition}
\newtheorem{definition}{Definition}[section]
\newtheorem{example}{Examples}[section]
\newtheorem{remarks}{Remarks}[section]
\newtheorem{remark}{Remark}[section]

\newcommand{\ord}{{\rm Ord}}
\newcommand{\lcm}{{\rm lcm}}
\newcommand{\tf}{{\rm F}}
\newcommand{\txau}{{\rm Aut}}

\newcommand{\txch}{{\rm Ch}}
\newcommand{\ox}{\mathcal{O}_x}

\newcommand{\oa}{\mathcal{O}_a}
\newcommand{\mbn}{\mathbb{N}}
\newcommand{\mbz}{\mathbb{Z}}
\newcommand{\mbq}{\mathbb{Q}}
\newcommand{\mbr}{\mathbb{R}}
\newcommand{\mbc}{\mathbb{C}}
\newcommand{\mbf}{\mathbb{F}}
\newcommand{\mbe}{\mathbb{E}}
\newcommand{\mbk}{\mathbb{K}}
\newcommand{\pf}{\mbf[x]}

\newcommand{\pq}{\mbq[x]}

\newcommand{\pc}{\mbc[x]}
\newcommand{\pz}{\mbz[x]}

\numberwithin{section}{chapter}
\numberwithin{equation}{chapter}
\tikzset{node distance=1.8cm, auto}

\begin{document}
\frontmatter
\title{A Concise Course in Galois Theory}


\author{Huichi Huang}
\address{Mathematical Department, College of mathematics and statistics, Chongqing University \\
Chongqing, China, 401331}
\email{huanghuichi@cqu.edu.cn}




\date{\today}



\maketitle


\setcounter{page}{4}

\epigraph{Algebra is the offer made by the devil to the mathematician. The devil says:``I will give you this powerful machine, it will answer any question you like. All you need to do is give me your soul: give up geometry and you will have this marvellous machine."}
{\textit{Michael Atiyah}}

\tableofcontents

\include{Preface}

\mainmatter
\chapter*{Preface}



Algebra, in its essence, is a language of symmetry and structure. From the ancient quest to solve polynomial equations to the modern exploration of abstract algebraic systems, the journey of algebra has been marked by profound insights and elegant theories. Among these, Galois theory stands as a crowning achievement, weaving together group theory, field theory, and the geometry of equations into a unified framework. This book aims to guide readers through this beautiful landscape.

Born from years of teaching undergraduate and graduate algebra courses at Chongqing University, this text is designed to introduce Galois theory while minimizing prerequisites. It seeks to reconnect the abstract machinery of modern algeba: groups, rings, and fields with the historical problem that inspired its creation: determining when a polynomial can be solved by radicals. By anchoring abstract concepts in concrete motivation, we hope to illuminate both the ``how" and the ``why" of algebraic structures.

The book is divided into four chapters,  each building systematically toward the heart of Galois theory:  

1. Basic Group Theory: From Lagrange’s pioneering work to Sylow’s theorems, we lay the group-theoretic foundation, emphasizing symmetric groups, solvability, and the simplicity of alternating groups.  

2. Basic Ring and Field Theory: Polynomial rings, ideals, and field extensions are developed with an eye toward their role in splitting fields and separability.  

3. Galois Theory: The fundamental theorem and its implications are rigorously proved, linking field extensions to group actions and paving the way for applications.  

4. Applications: From solving polynomials by radicals to classical ruler-and-compass problems, we demonstrate how Galois theory resolves age-old questions with modern tools.  

The book has the following  features:
  
Minimalist Foundation: Only essential concepts from group, ring, and field theory are introduced, avoiding unnecessary abstraction while maintaining rigor.  

Historical Motivation: The narrative frequently returns to the problem of solvability by radicals, grounding abstract ideas in their historical context.  

Computational Insights: Explicit calculations of Galois groups and splitting fields bridge theory and practice. 
 
Exercises and Examples: Over 150 exercises ranging from foundational proofs to exploratory problems—complement detailed examples, encouraging active learning.

This book is tailored for advanced undergraduates and beginning graduate students in mathematics. Familiarity with linear algebra and basic set theory is assumed, but no prior exposure to abstract algebra is required. Instructors will find it suitable for a one-semester course, with flexibility to emphasize theoretical depth or computational techniques.

                                                                                                                                                                                                                                                 Huichi Huang

                                                                                                                                                                                                                                           Chongqing University
  
                                                                                                                                                                                                                                                  Spring 2025

\chapter*{Notations}
Throughout the book, we use the following notations.

$\mbz$: the set of integers;

$\mbn$: the set of nonnegative integers;

$\mbz^+$: the set of positive integers;

$\mbq$: the set of rational numbers;

$\mbq^+$: the set of positive rational numbers;

$\mbr$: the set of real numbers;

$\mbc$: the set of complex numbers;

$|A|$: the cardinality of a finite set $A$;

$\emptyset$: the empty set.

$\mbf$: a field.

\chapter{Basic Group Theory}

\section{Definitions and examples}

The formal definition of group was  given by  Cayley during the third decade of the 19th century, though  concrete examples already appeared  one century ago in the work of Lagrange.~\cite{Kleiner1986} Nowadays groups are  everywhere in mathematics and physics.

\begin{definition}
\index{group}
  A \textbf{group} is a nonempty set $G$ with a binary operation $``\cdot"$ such that
  \begin{enumerate}
    \item $a\cdot b$ is in $G$ for every $a$,$b$ in $G$;
    \item $(a\cdot b)\cdot c=a\cdot(b\cdot c)$ for all $a,b,c$ in $G$;
    \item there exists an element $e$ in $G$ such that $e\cdot a=a\cdot e=a$ for every $a$ in $G$;
    \item for every $a$ in $G$, there exists $b$ in $G$ with $a\cdot b=b\cdot a=e$.
  \end{enumerate}
  For brevity, we just write $a\cdot b$ as $ab$ when the binary operation is clear. If $ab=ba$ for all $a,b$ in $G$, then $G$ is called an \textbf{abelian group}.~\index{abelian group}
\end{definition}

\begin{example}\

  \begin{enumerate}
    \item The set of integers $\mbz$ under addition is an abelian group;
    \item The set of nonzero real numbers $\mbr^*=\mbr\setminus\{0\}$ under multiplication is an abelian group;
    \item The set $\text{GL}_n(\mbr)$ of invertible $n\times n$ real matrices under matrix multiplication is a group;
    \item Let $X$ be a nonempty set. The set of bijections on $X$ is a group under composition of maps. Denote the group by $S_X$ and call it \textbf{the symmetric group} of $X$~\index{symmetric group}.
    \item For $n\geq 1$, let $\mbz/n\mbz$ be the set $\{\bar{0},\bar{1},\cdots,\overline{n-1}\}$ and define the addition $\bar{i}$ and $\bar{j}$ in $\mbz/n\mbz$ to be $\bar{k}$ where $k$ in $[0,n)$ is the reminder of $i+j$ modulo $n$. Then $\mbz/n\mbz$ is a group under the above-defined addition, which is still denoted by $+$. Another common notation for $\mbz/n\mbz$ is $\mbz_n$.
    \item For  $n\geq 1$, let $(\mbz/n\mbz)^\times$ be set of positive integers not exceeding $n$ and prime to $n$. Define the product of $i$ and $j$ in $(\mbz/n\mbz)^\times$ be to the remainder of $i\times j$ modulo $n$. Then $(\mbz/n\mbz)^\times$ under the above-defined multiplication is an abelian group.
  \end{enumerate}
\end{example}

\begin{remarks}
\begin{enumerate}
  \item A group  $G$ is not only determined by the set, but also  by the binary operation. It might happen that the same set is a group under one binary operation, but not a group under another operation. For example, $\mbz$ is a group under addition, but not a group under multiplication.
  \item There is a unique identity in a group.
  \item Given an element $a$ in a group $G$, there is a unique $b$ in $G$ such that $ab=ba=e$ and $b$ is called the \textbf{inverse} of $a$ denoted by $a^{-1}$.~\index{inverse}
\end{enumerate}

\end{remarks}

A group $G$ is called \textbf{cyclic} if there exits $a$ in $G$ such that each $b$ in $G$ can be expressed as $b=a^n$ for some $n$ in $\mbz$~\index{cyclic group}, and $a$ is called a generator of $G$. In this case, write $G$ as $\langle a\rangle$. For example $\mbz=\langle 1 \rangle$ and $\mbz/n\mbz=\langle \bar{1} \rangle$.

A subset $A$ of  a group $G$ is called a \textbf{generating set} of $G$ if every element $a$ of $G$ can be expressed as
$$a=a_1^{n_1}a_2^{n_2}\cdots a_m^{n_m}$$ for $a_i$'s in $A$ and  $n_i$'s in $\mbz$. We also say that $G$ is generated by $A$. A cyclic group is a group generated by a single element.~\index{generating set} The smallest subgroup of $G$ containing some elements is called the subgroup generated by these elements.

\begin{remark}
  Generating sets of a group may not be unique, sometimes even infinite. For example, $\mbz^2=\langle (1,0), (0,1)\rangle=\langle (1,2), (2,3)\rangle$. In fact, rows of any element in $GL_2(\mbz)$ gives a generating set of $\mbz^2$.
\end{remark}

A subset $N$ of a group $G$ is called a \textbf{subgroup} of $G$ if $N$ is a group under the same binary operation as $G$.~\index{subgroup} Denote it by $N\leq G$. Denote the subgroup $\{e\}$ of $G$ by $1$.

Every group $G$ has two trivial subgroups $1$ and $G$. The set of even integers $2\mbz$ is a subgroup of $\mbz$, and $\mbz$ is a subgroup of $\mbr$.

For a subgroup $N$ of a group $G$ and $x$ in $G$, the set $xN=\{xy|y\in N\}$~($Nx=\{yx|y\in N\}$ is called a left~(right) \textbf{coset} of $N$ in $G$~\index{coset}. For example, the set of odd integers $1+2\mbz$ is both a left coset and a right coset of  $2\mbz$  in  $\mbz$ since $\mbz$ is abelian.  One can check that any two cosets either coincide or are disjoint.

\begin{theorem}[Lagrange's theorem]\

\label{thm: index}
If $H$ is a subgroup of a finite group $G$, then $|H|$ divides $|G|$.
\end{theorem}

To prove Lagrange's theorem, we need a lemma.
\begin{lemma}
\label{lem:coset}
Suppose that $H$ is a subgroup of $G$. Any two left cosets $xH$ and $yH$ either coincide or are disjoint.
\end{lemma}
\begin{proof}
Suppose $z\in xH\cap yH$. Then there exist $h_1,h_2\in H$ such that $z=xh_1=yh_2$. Hence for any $h\in H$, we have that $xh=(xh_1)(h_1^{-1}h)=(yh_2)(h_1^{-1}h)=y(h_2h_1^{-1}h)$ which gives that $xH\subseteqq yH$. Similarly $yH\subseteqq xH$. Therefore $xH=yH$.
\end{proof}

\begin{proof}~[Proof of Theorem~\ref{thm: index}\

Note that $G=\bigcup_{x\in G} xH$ and $|xH|=|H|$ for every $x$ in $G$. By Lemma~\ref{lem:coset}, $|G|$ is a multiple of $|H|$.
\end{proof}

Denote by $[G:H]$ the number of left cosets of $H$ in $G$ and call it the \textbf{index} of $H$ in $G$.~\index{index} From Theorem~\ref{thm: index}, one has $$[G:H]=\dfrac{|G|}{|H|}.$$

Every element $a$ in a group $G$ can generate a subgroup $\langle a \rangle=\{a^n\,|\, n\in\mbz\}$ of $G$. The cardinality of $\langle a\rangle$ is called the \textbf{order} of $a$. Denote it by $\ord_a$.~\index{order} The order of $a$ is the smallest positive integer $n$ such that $a^n=e$. If such $n$ does not exist for $a$, then we say that $a$ has infinite order.

By Lagrange's theorem, the order of every element of a finite group $G$ is a factor of $|G|$.

\section{The symmetric group $S_n$}

When $X=\{1,2, \cdots, n\}$, the group $S_X$ is called \textbf{the symmetric group of degree n}, denoted by $S_n$.~\index{symmetric group}

In $S_n$, the element $\sigma$ is called an \textbf{m-cycle} if $\sigma$ permutes $\{i_1,i_2,\cdots,i_m\}\subseteq \{1,2,\cdots, n\}$ as follows: $\sigma(i_k)=i_{k+1}$ for $1\leq k\leq m-1$ and $\sigma(i_m)=i_1$ and $\sigma(j)=j$ for $j\notin \{i_1,i_2,\cdots,i_m\}$. Denote it by $(i_1\, i_2\,\cdots\, i_m)$.~\index{cycle} A 2-cycle is called a \textbf{transposition}.~\index{transposition}

For example, in $S_5$, the 3-cycle $\sigma=(1\, 2\, 3)$ satisfies that $\sigma(1)=2,\sigma(2)=3,\sigma(3)=1$ and $\sigma(4)=4, \sigma(5)=5$.

A product of cycles reads from right to left as compositions of maps. For example with this rule, $(1\,2)(2\,3\,4)=(1\,2\,3\,4)$.

\begin{lemma}
  The order of an m-cycle is m.
\end{lemma}
\begin{proof}
Consider the m-cycle $\sigma=(i_1\, i_2\,\cdots\, i_m)$. One can see that $\sigma(i_1)=i_2, \sigma^2(i_1)=i_3,\cdots, \sigma^{m-1}(i_1)=i_m,\sigma^m(i_1)=i_1$. Similarly $\sigma^m(i_k)=i_k$ for all $1\leq k\leq m$. Also $\sigma^m(j)=j$ for all $j$ not in $\{i_1,i_2,\cdots,i_m\}$. So $\sigma^m=e$. Moreover for any $k<m$, $\sigma^k\neq e$ since $\sigma^k(i_1)=i_{k+1}\neq i_1$. Hence $\ord_\sigma=m$.
\end{proof}

Two cycles $(i_1\, i_2\,\cdots\, i_m)$ and $(j_1\, j_2\,\cdots\, j_k)$ are called  disjoint if $\{i_1,i_2,\cdots,i_m\}$ and $\{j_1,j_2,\cdots,j_k\}$ are disjoint sets.

\begin{proposition}
Every element of $S_n$ is a product of disjoint cycles.
\end{proposition}
\begin{proof}
Suppose $\sigma\neq e$. Take $i_1$ in $\{1,2,\cdots, n\}$ such that $\sigma(i_1)=i_2\neq i_1$. There exists an integer $m\geq 2$ such that $\sigma^m(i_1)=i_1$ since the order of $\sigma$ is finite. Assume that $m$ is the smallest positive integer such that $\sigma^m(i_1)=i_1$. We get an m-cycle $(i_1\,i_2\,\cdots\,i_m)$ where $i_k=\sigma^{k-1}(i_1)$ for $2\leq k\leq m$.
Repeat this process for $\{1,2,\cdots, n\}\setminus\{i_1,i_2,\cdots,i_m\}$. One may get another cycle. Then $\sigma$ is the product of all cycles gotten from the above process, and these cycles are disjoint.
\end{proof}

\begin{lemma}~\label{lem:join}
 For distinct $i_1,i_2,\cdots,i_m,k,j_1,j_2,\cdots,j_l$, 
 $$(i_1\, i_2\, \cdots\,i_m\,k)(k\,j_1\,j_2\,\cdots\,j_l)=(i_1\, i_2\, \cdots\,i_m\,k\,j_1\,j_2\,\cdots\,j_l).$$
\end{lemma}
Verification of this lemma is straightforward.
\begin{theorem}
Every element of $S_n$ is a product of transpositions.
\end{theorem}
\begin{proof}
It's enough to prove that every m-cycle is a product of transpositions. By Lemma~\ref{lem:join}, we have

\begin{align*}
  (i_1\,i_2\,\cdots\,i_m)&=(i_1\,i_2)(i_2\,i_3\,\cdots\,i_m)=(i_1\,i_2)(i_2\,i_3)(i_3\,\cdots\,i_m)  \\
  &=(i_1\,i_2)(i_2\,i_3)\cdots(i_{m-1}\,i_m).
\end{align*}
\end{proof}

An expression of an element $\sigma$ of $S_n$  as a product of transpositions is called a decomposition of $\sigma$ to transpositions.
An element of $S_n$ may have different decomposition to transpositions. For example in $S_4$, $(1\,2\,3)=(1\,2)(2\,3)=(1\,2\,3\,4)(3\,4)=(4\,1\,2\,3)(3\,4)=(4\,1)(1\,2)(2\,3)(3\,4)$. However, the parity of all decompositions is the same. That's, if a decomposition is  a product of even(odd) numbers of transpositions, then  so is any other decomposition.

We prove this in the rest of the section.

Suppose  $G$ and $H$ are groups. A group \textbf{homomorphism} $f:G\to H$ is a map such that $f(xy)=f(x)f(y)$ for all $x,y\in G$.~\index{homomorphism of groups}A bijective group homomorphism is called an \textbf{isomorphism}. 

Denote the standard basis of $\mathbb{R}^n=\mathbb{R}^{n\times1}$ by $\{e_i\}_{i=1}^n$. For $\sigma$ in $S_n$, the $n\times n$ permutation matrix $E_\sigma$ is given by $E_\sigma=(e_{\sigma(1)}\, e_{\sigma(2)}\,\cdots\,e_{\sigma(n)})$. The set $P_n$ of $n\times n$ permutation matrices is a group under matrix multiplications. Then define a map $\Phi$ from $S_n$ to $P_n$ by $\Phi(\sigma)=E_{\sigma^{-1}}$ for every $\sigma$ in $S_n$. One can check that  $\Phi$ is a group isomorphism.

The determinant map $\det$ is a group homomorphism from $P_n$ to the multiplicative group $\{\pm 1\}$. Hence $\det\cdot \Phi$ is a group homomorphism from $S_n$ to $\{\pm 1\}$, and $sgn(\sigma)=\det(\Phi(\sigma))$. 

Suppose $\sigma=\tau_1\tau_2\cdots\tau_k=\tau_1^\prime\tau_2^\prime\cdots\tau_l^\prime$ where $\tau_i$'s and $\tau_j^\prime$'s are transpositions. 
Then \begin{align*}
      \det \Phi(\sigma)&=\det \Phi(\tau_1)\det \Phi(\tau_2)\cdots\det \Phi(\tau_k)=(-1)^k \\
      &=\det \Phi(\tau_1^\prime)\det  \Phi(\tau_2^\prime)\cdots\det  \Phi(\tau_l^\prime)=(-1)^l.
     \end{align*}
Hence $k-l$ is even. So the sign of a permutation is well-defined.

\begin{definition}
The \textbf{alternating group} of degree n, $A_n$ consists of elements in $S_n$ which can be decomposed into a product of even numbers of transpositions.~\index{alternating group of degree n}
\end{definition}


\section{Constructions of groups}
In this section, we give some constructions of new groups from known ones.
\subsection{Quotient group}

A subgroup $N$ of a group $G$ is called \textbf{normal} if $xN=Nx$ for every $x$ in $G$. Denote it by $N\unlhd G$.~\index{normal subgroup} Define $G/N$ to be the set of all left cosets of $N$ in $G$.

\begin{theorem}
\label{quotient group}
  If $N\unlhd G$, then $G/N$ is a group under the operation given by $xN\cdot yN=xyN$ for all $x,y$ in $G$. We call $G/N$ the \textbf{quotient group} $G$ modulo $N$.~\index{quotient group}
\end{theorem}
\begin{proof}
First we check that $xN\cdot yN=xyN$ is well-defined. That is, if $x_1N=x_2N$ and $y_1N=y_2N$, then $x_1y_1N=x_2y_2N$. This follows from that $(x_1y_1)^{-1}x_2y_2=y_1^{-1}x_1^{-1}x_2y_2=(y_1^{-1}y_2)(y_2^{-1}(x_1^{-1}x_2)y_2)$ is in $N$ by the normality of $N$.

Moreover
\begin{itemize}
\item $G/N$ is closed under the binary operation;
\item the binary operation on $G/N$ is associative;
\item the unit of $G/N$ is $N$;
\item the inverse of $xN$ is $x^{-1}N$.
\end{itemize}
\end{proof}

We call the map $\pi:G\to G/N$ given by $\pi(x)=xN$ for all $x$ in $G$ the~\textbf{quotient map}.~\index{quotient map}

In fact $\mbz/n\mbz$ is the group of $\mbz$ modulo $n\mbz$.

Suppose $f:G\to H$ is a group homomorphism. The set $\{x\in G\,|\, f(x)=e\}$ is called the \textbf{ kernel} of $f$. Denote it by $\ker{f}$.~\index{kernel of group homomorphism}
The set $\{f(x)\,|\, x\in G\}$ is called the \textbf{image} of $f$. Denote it by $\text{im}f$.~\index{image of a group homomorphism}
It follows that $\ker f$ is a normal subgroup of $G$ and $\text{im}f$ is a subgroup of $H$.

A group isomorphism from $G$ to itself is called an \textbf{automorphism} of $G$. If there is an isomorphism  between groups $G$ and $H$, then we say that $G$ is isomorphic to $H$. Denote it by $G\cong H$.~\index{isomorphism of groups}

\begin{remark}
  In fact, for a group $G$, a subgroup $N$ is normal iff there exists a group homomorphism $f:G\to H$ such that $N=\ker{f}$.
\end{remark}

\begin{example}\

  \begin{enumerate}
  \item $G/G\cong 1$ and $G/1\cong G$.
    \item The quotient group $\mbr/\mbz$;
    \item the quotient group $\mbz/n\mbz$ for a positive integer $n$.
  \end{enumerate}
  \end{example}
\subsection{Direct sum and direct product of groups}\

Suppose that $\{G_i\}_{i\in I}$ is a set of groups. The \textbf{direct product} of $G_i$'s, denoted by $\displaystyle\Pi_{i\in I} G_i$, is the set $\{(x_i)_{i\in I}\,|\, x_i\in G_i\, \text{for each}\, i\in I\}$ with the binary operation $(x_i)(y_i)=(x_iy_i)$ for $(x_i),(y_i)$ in $\displaystyle\Pi_{i\in I} G_i$.~\index{direct product of groups} The \textbf{direct sum} of $G_i$'s, denoted by $\displaystyle\bigoplus_{i\in I} G_i$, consists of elements $(x_i)_{i\in I}$ in $\displaystyle\Pi_{i\in I} G_i$ such that $x_i=e$ for all but finitely many $i$ in $I$, and the binary operation on  $\displaystyle\bigoplus_{i\in I} G_i$ is the same as that on $\displaystyle\Pi_{i\in I} G_i$~\index{direct sum of groups}.

Direct product and direct sum of groups give rise to many interesting  groups. For example, $\displaystyle\bigoplus_{n=1}^\infty \mbz/n\mbz$ is an infinite group in which every element has finite order and every positive integer is the order of some element.

\section{Isomorphism theorems}

In this section, we state and prove isomorphism theorems for groups.
\begin{theorem}[the first isomorphism theorem]\

Suppose $f:G\to H$ is a group homomorphism. Then $\ker f$ is a normal subgroup of $G$, the image of $f$, $\text{Im} f$ is a subgroup of $H$ and
$$G/\ker f\cong \text{Im} f.$$
\end{theorem}
\begin{proof}
Verifications of  $\ker f\unlhd G$ and $\text{Im} f\leq H$ are left as exercises.
 
Define a map $\tilde{f}: G/\ker f\to \text{Im} f$ by
$$\tilde{f}(x\ker f)=f(x)$$ for every $x$ in $G$.

It's routine to check that
\begin{itemize}
  \item $\tilde{f}$ is a well-defined group homomorphism;
  \item $\tilde{f}$ is surjective;
  \item $\tilde{f}$ is injective.
\end{itemize}
\end{proof}

\begin{theorem}[the second isomorphism theorem]\

\label{2ndisothm}

Suppose $H$ is a subgroup of a group $G$ and $K$ is a normal subgroup of $G$. Then
\begin{enumerate}
\item $HK=KH$, and $HK$ is a subgroup of $G$.
\item $H\cap K$ is a normal subgroup of $H$.
\item $HK/K\cong H/(H\cap K)$.
\end{enumerate}
\end{theorem}
\begin{proof}\

(1)For $h$ in $H$ and $k$ in $K$, we have $hk=hkh^{-1}h$ is in $KH$. Hence $HK\subseteq KH$. Also $kh=hh^{-1}kh$ is in $HK$. So $KH\subseteq HK$. Therefore $HK=KH$.

The set $HK$ contains the identity of $G$. For $h$ in $H$ and $k$ in $K$, the element $(hk)^{-1}=k^{-1}h^{-1}=h^{-1}hk^{-1}h^{-1}$ is in $HK$ since the normality of $K$  implies that $hk^{-1}h^{-1}$ is in $K$. Moreover for $h_1, h_2$ in $H$ and $k_1, k_2$ in $K$, we have that $h_1k_1h_2k_2=h_1h_2 (h_2^{-1}k_1 h_2)k_2$ is in $HK$ from the normality of $K$. So $HK$ is a subgroup of $G$.

(2)The set $H\cap K$ is a subgroup of $H$. We only prove that $H\cap K$ is normal in $H$. For any $h$ in $H$ and any $k$ in $H\cap K$, $hkh^{-1}$ is in $H\cap K$ since $K$ is normal in $G$.

(3)Define $f: HK/K\to H/(H\cap K)$ by $f(hkK)=hH\cap K$ for all $h$ in $H$ and $k$ in $K$.

The map $f$ is well-defined since if $h_1k_1K=h_2k_2K$, then $h_1K=h_2K$. So $h_1h_2^{-1}$ is in $K$. Thus $h_1h_2^{-1}$ is in $H\cap K$. This gives that $h_1H\cap K=h_2H\cap K$.

Note that $h_1k_1Kh_2k_2K=h_1Kh_2K=h_1h_2K$. The last equality follows from that $K$ is normal. Hence
\begin{align*}
  f(h_1k_1Kh_2k_2K)=f(h_1h_2K)=h_1h_2H\cap K \\
=(h_1H\cap K)( h_2H\cap K)=f(h_1k_1K)f(h_2k_2K).
\end{align*}

So $f$ is a homomorphism.

Clearly $f$ is surjective. If $h$ is in $H\cap K$, then $hk K=K$. Hence $f$ is injective.

\end{proof}

For $G=\mbz$, consider $H=m\mbz$ and $K=n\mbz$. Then $H+K=\gcd(m,n)\mbz$ and $H\cap K=\lcm(m,n)\mbz$. From Theorem~\ref{2ndisothm}, we have $$\gcd(m,n)\mbz/n\mbz\cong m\mbz/\lcm(m,n)\mbz.$$
\begin{theorem}[the third isomorphism theorem]\

If $S\subseteq K$ are normal subgroups of a group $G$, then
$$G/S \big/ K/S\cong G/K.$$
 \end{theorem}
\begin{proof}
Define a map $f$ from $G/S$ to $G/K$ by $f(xS)=xK$ for all $x$ in $G$.

The map $f$ is well-defined since $S\subseteq K$.

Since $S$ and $K$ are normal subgroups of $G$, we get that $f(xSyS)=f(xyS)=xyK=xKyK=f(xS)f(yS)$ for all $x,y$ in $G$. So $f$ is a homomorphism. Also $f$ is surjective.

Note that $\ker{f}=K/S$. Applying the first isomorphism theorem, we complete the proof.

\end{proof}

\section{Group actions}

In this section, we introduce group actions on sets, and use it as a tool to prove the Sylow's theorem. The Sylow's theorem is a fundamental theorem in finite group theory~\cite{Sylow1872}.
\begin{definition}
We say that a group $G$ acts on a set $X$ if there is a map $\alpha: G\times X \to X$ mapping $(g,x)$ to $\alpha_g(x)$ for all $g$ in $G$ and $x$ in $X$ such that
\begin{enumerate}
\item $\alpha_e(x)=x$ for all $x$ in $X$;
\item $\alpha_{gh}(x)=\alpha_g(\alpha_h(x))$ for all $g,h$ in $G$ and $x$ in $X$.
\end{enumerate}
The map $\alpha$ is called an  \textbf{action} of $G$ on $X$. Denote a group action by $G\overset{\alpha}{\curvearrowright}X$.~\index{group action}

A group $G$ can act on a set $X$ such that every group element is the identity map on $X$. Also a group can act on itself by left translations.

For every $x$ in $X$, the \textbf{orbit} of $x$, denoted by $\ox$, is the set $\{\alpha_g(x)\,|g\in G\}$.~\index{orbit} The \textbf{stabilizer} of $x$, denoted by $G_x$, is the set $\{g\in G\,|\, \alpha_g(x)=x\}$.~\index{stabilizer}
\end{definition}

We say that $a$ and $b$ in $G$ are \textbf{conjugate} if there is $g$ in $G$ such that $b=gag^{-1}$.~\index{conjugate} The set $C_a=\{gag^{-1}|\,g\in G\}$ is called the \textbf{conjugacy class} of $a$.~\index{conjugacy class}

Below we list some facts of orbits whose proofs are left as exercises.

\begin{proposition}
\label{prop:orbit}
Consider a group action $G\overset{\alpha}{\curvearrowright}X$. The  following hold:
\begin{enumerate}
  \item Any two orbits either are disjoint or coincide. Hence $X$ is the disjoint union of orbits.
  \item For every $x$ in $G$, the stabilizer $G_x$ is a subgroup of $G$.
  \item For every $x$ in $X$,  $G/G_x\cong \ox$ as sets, so $|\ox|$ divides $|G|$ when $G$ is finite.
 \end{enumerate}
\end{proposition}

\begin{remarks}\

\begin{itemize}
  \item Every group action of $G$ on $X$ corresponds to a group homomorphism $\Phi: G \to S_X$ given by $\Phi(g)=\alpha_g$ for all $g$ in $G$.
  \item Proposition~\ref{prop:orbit}(3) is called the~\textbf{orbit-stabilizer theorem}.~\index{orbit-stabilizer theorem}
\end{itemize}

\end{remarks}

Below is a group action of particular interest.

The group action $\alpha$ of $G$ on $G$ given by $\alpha_g(a)=gag^{-1}$ for $g,a$ in $G$ is called the \textbf{inner automorphism} action of $G$~\footnote{The map $\alpha_g:G\to G$ given by $\alpha_g(a)=gag^{-1}$ for all $a$ in $G$ is an isomorphism called an inner automorphism.}.~\index{inner automorphism} The conjugacy class $C_a$ of $a$ is $\oa$, the orbit of $a$ under the inner automorphism action of $G$. The stabilizer of $a$ is  called the~\textbf{centralizer} or \textbf{commutator} of $a$, denoted by $C_G(a)$. That is
$$C_G(a)=\{g\in G\,|\, gag^{-1}=a\}=\{g\in G\,|\, ga=ag\}.$$~\index{centralizer, commutator}

The \textbf{center} of $G$, $Z(G)$, is the intersection of all $C_G(a)$'s. In another word $Z(G)=\{g\in G\,|\,ga=ag\, \text{for \,all\,} a\,in\, G\}$.~\index{center}

When $G$ is a finite group, $|C_a||C_G(a)|=|G|$ by Proposition~\ref{prop:orbit}(3).

Suppose $G$ is a finite group, $p$ is a prime factor of $|G|$ and $|G|=p^km$ with $p$ and $m$ being coprime. A subgroup $H$ of $G$ with $|H|=p^k$  is called a \textbf{Sylow $p$-subgroup} of $G$ and  a subgroup $K$ of $G$ such that  $|K|=p^l$ with $l\leq k$ is called a \textbf{$p$-subgroup} of $G$.~\index{Sylow $p$-subgroup}

The following proof of part of Sylow's theorem is due to H. Wielandt.~\cite{Wielandt1959}

\begin{theorem}~[Sylow's theorem, part]\

A finite group $G$ has a Sylow $p$-subgroup.
\end{theorem}
\begin{proof}
Let $X$ be the sets of subsets of $G$ with $p^k$ elements. The group $G$ acts on $X$ by $\alpha_g(A)=gA$ for any $g$ in $G$ and any $A$ in $X$.
Note that $|X|=\binom{p^km}{p^k}$. Hence $p\nmid \binom{p^km}{p^k}$. The set $X$ is the disjoint union of orbits, so there is $B$ in $X$ such that $p\nmid |\mathcal{O}_B|$.

Next we prove that  $|G_B|=p^k$.

Note that $|\mathcal{O}_B|$ divides $|G|$ and $p\nmid |\mathcal{O}_B|$, so $|\mathcal{O}_B|$ divides $m$.  From $G/G_B\cong \mathcal{O}_B$, one has that $p^k$ divides $|G_B|$. In particular $p^k\leq |G_B|$.

Moreover fix $b$ in $B$. One can define a map $\varphi: G_B \to B$ by $\varphi(g)=gh$ for all $g$ in $G_B$. The map $\varphi$ is a well-defined injection. Hence $|G_B|\leq |B|=p^k$.
\end{proof}

\begin{corollary}~(Cauchy)\

If a prime $p$ divides $|G|$, then $G$ has an element of order $p$.
\end{corollary}
\begin{proof}
Let $H$ be a Sylow $p$-subgroup of $G$. Then $\ord_h=p^l$ with $p^l$ dividing $|G|$ for each $h\neq e$ in $H$.  Therefore $\ord_{h^{p^{l-1}}}=p$.
\end{proof}
\section{The simplicity of $A_n$}

Note that the alternating group $A_4$ is simple since $$\{e, (1\,2)(3\,4), (1\,3)(2\,4), (1\,4)(2\,3)\}\unlhd A_4.$$ However  $A_n$ is simple  when $n\geq 5$.
This fact is the main theorem in the section and plays a crucial role for proving Abel-Ruffini's theorem, that is, the polynomial $x^5-80x+5$ is unsolvable by radicals.
\begin{theorem}
~\label{thm:Ansimple}
 $A_n$ is simple for all $n\geq 5$.
\end{theorem}
Proof of Theorem~\ref{thm:Ansimple} consists of two steps:

Step 1.  $A_5$ is simple;

Step 2. For $n\geq 5$, by assuming that $A_n$ is simple, we prove that $A_{n+1}$ is simple.

We need some preliminaries.
\begin{lemma}
\label{lm:conjugate}
If $\gamma=(i_1\, i_2\,\cdots\, i_m)$ is an m-cycle in $S_n$, then a conjugate of $\gamma$, say, $\sigma\gamma\sigma^{-1}$ is  an m-cycle of the form $(\sigma(i_1)\, \sigma(i_2)\,\cdots\, \sigma(i_m))$.
\end{lemma}
\begin{proof}
Firstly $\sigma\gamma\sigma^{-1}(\sigma(i_k))=\sigma(i_{k+1})$ for all $1\leq k\leq m$.~\footnote{Here $i_{m+1}$ is understood as $i_1$.}

If $j\notin\{\sigma(i_1), \sigma(i_2),\cdots, \sigma(i_m)\}$, then there is $i\notin\{i_k\}_{k=1}^m$ such that $j=\sigma(i)$. It follows that $\sigma\gamma\sigma^{-1}(j)=\sigma\gamma\sigma^{-1}(\sigma(i))=\sigma\gamma(i)=\sigma(i)=j$.

Hence $\sigma\gamma\sigma^{-1}=(\sigma(i_1)\, \sigma(i_2)\,\cdots\, \sigma(i_m))$.
\end{proof}

\begin{lemma}
\label{lm:An3cycles}
The alternating group $A_n$ is generated by 3-cycles.
\end{lemma}
\begin{proof}
 It suffices to prove that the product of two transpositions is a product of 3-cycles.

 There are two cases.

Case 1: $(a\,b)(a\,c)=(a\,c\,b)$.

Case 2: $(a\,b)(c\,d)=(a\,b)(b\,c)(b\,c)(c\,d)=(a\,b\,c)(b\,c\,d)$.
\end{proof}

\begin{lemma}
If a normal subgroup $H$ of $A_n$ contains a 3-cycle, then $H$ contains all 3-cycles. Hence $H=A_n$.
\end{lemma}
\begin{proof}
Assume that $\sigma=(a\,b\,c)$ is in $H$.

Suppose that $\gamma=(d\,e\,f)$ with $\{d,e,f\}\cap\{a,b,c\}=\emptyset$. This only happens  when $n\geq 6$. Let $\tau=(a\,d\,b\,e)(c\,f)$. Then $\tau$ is in $A_n$ and $\gamma=\tau\sigma\tau^{-1}$ is in $H$.

Suppose that $\gamma=(a\,e\,f)$  with $\{e,f\}\cap\{b,c\}=\emptyset$. This only happens when $n\geq 5$. Let $\tau=(b\,e)(c\,f)$. Then $\tau$ is in $A_n$ and $\gamma=\tau\sigma\tau^{-1}$ is in $H$.

Suppose that $\gamma=(a\,b\,f)$ with $c\neq f$. This only happens when $n\geq 4$. Let $\tau=(a\,b)(c\,f)$. Then  $\tau$ is in $A_n$ and $\gamma=\tau\sigma^2\tau^{-1}$ is in $H$. Also $(a\,f\,b)=\gamma^2$ is in $H$.

Suppose that $\gamma=(a\,c\,b)$. Then $\gamma=\sigma^2$ is in $H$.

The above discussions exhaust all possibilities for 3-cycles distinct from $\sigma$.

\end{proof}

Now we prove the 1st step of Theorem~\ref{thm:Ansimple}.

\begin{proposition}
$A_5$ is simple.
\end{proposition}
\begin{proof}
Suppose $H\neq1$ is a normal subgroup of $A_5$. We want to prove that $H=A_5$. By Lemma~\ref{lm:An3cycles}, this is true if $H$ contains all 3-cycles. By Lemma~\ref{lm:conjugate}, since $H$ is normal, it's enough to prove that $H$ contains a 3-cycle.

Take $\sigma\neq e$ in $H$.

Without loss of generality, there are 3 cases for $\sigma$.

Case1: $\sigma=(1\,2\,3)$. If this happens, then we are done.

Case2: $\sigma=(1\,2)(3\,4)$. Take $\tau=(1\,2)(4\,5)$. Then $\tau\sigma\tau^{-1}\sigma$ is in $H$ since $\tau$ is in $A_5$. Note that
$$\tau\sigma\tau^{-1}\sigma=(\tau(1)\,\tau(2))(\tau(3)\,\tau(4))(1\,2)(3\,4)=(1\,2)(3\,5)(1\,2)(3\,4)=(3\,5)(3\,4)=(5\,3\,4).$$

So $H$ also contains a 3-cycle.

Case 3:  $\sigma=(1\,2\,3\,4\,5)$. Take $\tau=(2\,3\,4)$. Then $\tau\sigma\tau^{-1}\sigma^{-1}$ is in $H$ since $\tau$ is $A_5$. Moreover
$$\tau\sigma\tau^{-1}\sigma^{-1}=(\tau(1)\,\tau(2)\,\tau(3)\,\tau(4)\,\tau(5))\sigma^{-1}=(1\,3\,4\,2\,5)(1\,5\,4\,3\,2)=(2\,3\,5).$$

In this case, $H$ contains a 3-cycle.
\end{proof}
Now it's ready to prove Theorem~\ref{thm:Ansimple}.
\begin{proof}~[Proof of Theorem~\ref{thm:Ansimple}]\

Assume that $A_{n-1}$ is simple. We are going to prove that $A_n$ is simple.

Denote $A_n$ by $G$. Then $G$ acts on $X=\{1,2,\cdots, n\}$.

Suppose $H$ is a  normal subgroup of $G$ such that $1\subsetneqq H$.

First we prove that there exists $\sigma\neq\tau$ in $H$ such that $\sigma(i)=\tau(i)$ for some $i$ in $X$.

Take an nonidentity $\sigma$ in $H$. Since $\sigma$ is a product of disjoint cycles, there are two possible forms for $\sigma$.

1. There is an m-cycle $(i_1\,i_2\,i_3\,\cdots\, i_m)$ with $m\geq 3$ in the decomposition of $\sigma$.

In this case, let $\gamma=(i_3\,i_4\,i_5)$ with $i_4\notin\{i_1,i_2,i_3\}$. Then $\tau=\gamma\sigma\gamma^{-1}$ is in $H$ since $\gamma$ is in $A_n$ and
$$\tau=(\gamma(i_1)\,\gamma(i_2)\,\gamma(i_3)\,\cdots\, \gamma(i_m))\cdots= (i_1\,i_2\,i_4\,\cdots\, \gamma(i_m))\cdots.$$
Note that $\tau\neq\sigma$ since $\tau(i_2)=i_4$ and $\sigma(i_2)=i_3$. However $\tau(i_1)=i_2=\sigma(i_1)$.

2. $\sigma$ is a product of disjoint transpositions, i.e., $\sigma=(i_1\,i_2)(i_3\,i_4)\cdots$.

In this case, let $\gamma=(i_1\,i_2)(i_4\,i_5)$ with $i_5\notin\{i_1,i_2,i_3,i_4\}$. Then $\tau=\gamma\sigma\gamma^{-1}$ is in $H$ since $\gamma$ is in $A_5$ and
$$\tau=(\gamma(i_1)\,\gamma(i_2))(\gamma(i_3)\,\gamma(i_4))\cdots=(i_1\,i_2)(i_3\,i_5)\cdots.$$

It follows from $\tau(i_3)=i_5$ and $\sigma(i_3)=i_4$ that $\tau\neq\sigma$. But $\tau(i_1)=i_2=\sigma(i_1)$.

Hence there is an nonidentity $\sigma$ in $H$ such that $\sigma(i)=i$ for some $i$ in $X$.

The stabilizer $G_i$ of $i$ is isomorphic to $A_{n-1}$, hence $G_i$ is simple by assumption. Moreover  $G_i\cap H$ is a normal subgroup of $G_i$ and $G_i\cap H\neq1$. Therefore $G_i\cap H=G_i$ which means $G_i\subset H$. Thus $H$ contains a 3-cycle, consequently $H$ contains all 3-cycles.  So $H=G$ by Lemma~\ref{lm:An3cycles}.
\end{proof}

\section{Solvable groups}

In this section, we study the concept of solvable groups, which is, the group theoretic description of solvable polynomials.
\begin{definition}
  A group $G$ is called \textbf{solvable} if there exists a series of subgroups $\{G_i\}_{i=0}^n$ of $G$ such that
  $$1=G_n\lhd G_{n-1}\lhd G_2\cdots G_1\lhd G_0=G,$$ and $G_i$ is a normal subgroup of $G_{i-1}$ with $G_{i-1}/G_i$ being abelian for every $1\leq i\leq n$.~\index{solvable group}We call  $\{G_i\}_{i=0}^n$ a \textbf{solvable sequence} of $G$.~\index{solvable sequence}
\end{definition}
Immediately we get the following from the definition.
\begin{proposition}
A subgroup of a solvable group is also solvable, therefore, if a group contains an unsolvable subgroup, then the group itself is unsolvable.
\end{proposition}

\begin{example}~[Examples and nonexamples of solvable groups]\
\begin{enumerate}
  \item Every abelian group is solvable.
  \item For $n\leq 4$, $S_n$ is solvable. Let $H=\{e, (1\,2)(3\,4), (1\,3)(2\,4),(1\,4)(2\,3)\}$. Then $H$ is a normal subgroup of $S_4$ since $\sigma(a\,b)(c\,d)\sigma^{-1}=(\sigma(a)\,\sigma(b))(\sigma(c)\,\sigma(d))$ for every $\sigma$ in $S_4$. Moreover $1\lhd H\lhd A_4\lhd S_4$ is a solvable sequence of $S_4$.
  \item Every nonabelian simple group is unsolvable. So when $n\geq 5$, $A_n$ and $S_n$ are unsolvable.
\end{enumerate}

\end{example}
\begin{definition}
The \textbf{commutator subgroup} of a group $G$ is the subgroup generated by elements $[a,b]=ab(ba)^{-1}$ for $a,b$ in $G$. Denote the commutator subgroup of $G$ by $[G,G]$.~\index{commutator subgroup}
\end{definition}

The commutator subgroup $[G,G]$ is the smallest normal subgroup of $G$ such that the quotient is abelian. More precisely, the following hold.

\begin{theorem}
\label{thm: commutator}

The commutator subgroup $[G,G]$ is a normal subgroup of $G$ such that $G/[G,G]$ is abelian. Moreover if $H$ is a normal subgroup of $G$ with $G/H$ abelian, then $[G,G]$ is a subgroup of $H$.
\end{theorem}
\begin{proof}
For all $a,b,c$ in $G$, one has $c[a,b]c^{-1}=[cac^{-1}, cbc^{-1}]$. Thus $[G,G]$ is normal.

Moreover if $H$ is a normal subgroup of $G$ with $G/H$ abelian, then for all $a,b$ in $G$, it's true $abH=baH$, which means, $[a,b]$ is in $H$. Hence $H$ contains $[G,G]$.
\end{proof}

Denote $[G,G]$ by $G^{(1)}$ and  the commutator subgroup of $G^{(n)}$ by $G^{(n+1)}$ for $n\geq 1$.
\begin{theorem}
$G$ is solvable iff $G^{(n)}=1$ for some $n\geq 1$.
\end{theorem}
\begin{proof}
Suppose that $G^{(n)}=1$ for some $n\geq 1$.

Then $$1=G^{(n)}\lhd G^{(n-1)}\lhd\cdots \lhd  G^{(1)}\lhd G=G^{(0)}.$$ Note that $G^{(i-1)}/G^{(i)}$ is abelian for all $1\leq i\leq n$, thus $G$ is solvable.

Conversely assume that $G$ is solvable.

Then  there exists a sequence of subgroups $\{G_i\}_{i=1}^n$ of $G$ such that
  $$1=G_n\lhd G_{n-1}\lhd G_{n-2}\lhd\cdots G_{1}\lhd G_0=G,$$ and  $G_{i-1}/G_i$ is abelian for every $1\leq i\leq n$.

Since $G_0/G_1$ is abelian, by Theorem~\ref{thm: commutator}, one has that $G^{(1)}\leq G_1$.

So $G^{(2)}=[G^{(1)}, G^{(1)}]\leq [G_1, G_1]\lhd G_2$. Inductively $G^{(n)}\leq G_n=1$.

\end{proof}

The least n such that $G^{(n)}=1$ for a solvable group $G$ is called the \textbf{derived length} of $G$.~\index{derived length}

The section ends with some further properties of solvable groups.
\begin{proposition}\
\label{prop:solvable}
  \begin{enumerate}
    \item If $G$ is solvable and $N$ is a normal subgroup of $G$, then $G/N$ is solvable.
    \item Suppose $N$ is a normal subgroup of $G$. If $N$ is solvable and $G/N$ is solvable, then $G$ is solvable.
  \end{enumerate}
\end{proposition}
\begin{proof}
  (1) Suppose that  $\{G_i\}_{i=0}^n$ is a sequence of subgroups of $G$ such that
  $$1=G_n\lhd G_{n-1}\lhd G_{n-2}\lhd\cdots G_{1}\lhd G_0=G,$$ and  $G_{i-1}/G_i$ is abelian for every $1\leq i\leq n$.

 Let $\pi: G\to G/N$ be the quotient map. Then $\{\pi(G_iN)\}_{i=0}^n$ is a solvable sequence of $G/N$.

(2) Let $\{\tilde{G}_i\}_{i=0}^m$ be a solvable sequence of $G/N$ and $\{N_j\}_{j=0}^k$ be  a solvable sequence of $N$. Then one has a solvable sequence given by
$$1=N_k\lhd\cdots\lhd N_0=N= \pi^{-1}(\tilde{G}_m)\lhd \pi^{-1}(\tilde{G}_{m-1})\lhd\cdots\lhd \pi^{-1}(\tilde{G}_1)\lhd \pi^{-1}(\tilde{G}_0)=G,$$ where $\pi^{-1}(\tilde{G}_j)$ is the primage of $\tilde{G}_j$ under the quotient map $\pi:G\to G/N$ for $0\leq j\leq m$.
\end{proof}
\section*{Exercises}

\begin{exercise}
Prove that every subgroup of $\mbz$ is $n\mbz$ for some $n\geq 0$.
\end{exercise}

\begin{exercise}
Define a binary operation $*$ on $\mathbb{R}\setminus\{-1\}$ as $a*b=ab+a+b$ for $a,b\in \mathbb{R}\setminus\{-1\}$. Prove that $(\mathbb{R}\setminus\{-1\}, *)$ is a group.
\end{exercise}

\begin{exercise}
Assume that  for every $\alpha\in\Lambda$, $G_\alpha$ is a subset of $X$ and is a group. Prove that the intersection $\bigcap_{\alpha\in\Lambda}G_\alpha$ is also a subgroup. How about the union $\bigcup_{\alpha\in\Lambda}G_\alpha$?
\end{exercise}

\begin{exercise}
  Give a group homomorphism from the multiplicative group $\mbr^+$ of positive numbers to the additive group $\mbr$ of real numbers.
\end{exercise}

\begin{exercise}
Suppose $a$ and $b$ are in a group $G$, and $\ord_{a}=m$ and $\ord_{b}=n$. What can we say about $\ord_{ab}$ if $ab=ba$. What happens if $ab\neq ba$?
\end{exercise}

\begin{exercise}
  Prove that  $\ord_{ab}=  \ord_{ba}$ for $a,b$ in a group $G$.
\end{exercise}

\begin{exercise}
  Suppose $a$ is in a group $G$ with $\ord_a=n$. Find $\ord_{a^m}$ for $1\leq m\leq n$.
\end{exercise}

\begin{exercise}
\label{Fermat's little theorem}
  Prove \textbf{Fermat's little theorem}: if a prime $p$ does not divides $a$, then $a^{p-1}\equiv 1\mod p$.~\index{Fermat's little theorem}
\end{exercise}

\begin{exercise}
Prove that a subgroup  $K$ of $G$ is normal iff there exists a group homomorphism $f:G\to H$ such that $K=\ker{f}$.
\end{exercise}

\begin{exercise}
Prove that $Z(G)$ is a normal subgroup of $G$.
\end{exercise}

\begin{exercise}
A subgroup of a cyclic group is cyclic.
\end{exercise}

\begin{exercise}
  Prove the \textbf{Chinese reminder theorem}:
if $m$ and $n$ are coprime positive integers, then  $\mbz/m\mbz\times \mbz/n\mbz\cong\mbz/mn\mbz$.~\index{Chinese reminder theorem}
\end{exercise}

\begin{exercise}
Write down all group homomorphisms from $\mbz$ to $\mbz$.
\end{exercise}

\begin{exercise}
Write down all group homomorphisms from the unit circle $\mathbb{S}$ to itself.
\end{exercise}

\begin{exercise}
List all subgroups of $S_3$ and find normal subgroups among them.
\end{exercise}

\begin{exercise}
  Prove that for every group $G$, there is a set $X$ such that $G$ is isomorphic to a subgroup of the symmetric group $S_X$. In particular, every finite group is isomorphic to a subgroup of $S_n$ for some $n\geq 1$.(Hint: Take $X=G$.)
\end{exercise}

\begin{exercise}
Prove that $\displaystyle\bigoplus_{i\in I} G_i$ is a normal subgroup of  $\displaystyle\Pi_{i\in I} G_i$.
\end{exercise}

\begin{exercise}
~\label{ex:S5}
Prove that a 5-cycle and a 2-cycle generate $S_5$.
\end{exercise}

\begin{exercise}
Prove that  for $n\geq 5$,  $A_n$ is generated by 5-cycles.
\end{exercise}

\begin{exercise}
  Prove that $\mbq$ is not finitely generated.
\end{exercise}

\begin{exercise}
Suppose $\sigma=(i_1\, i_2\,\cdots\, i_m)$ is an m-cycle in $S_n$. Prove that $\sigma^k$ is a product of $\gcd(m,k)$ many $\frac{m}{\gcd(m,k)}$-cycles.
\end{exercise}

\begin{exercise}
Prove that every subgroup $\tilde{H}$ of the quotient group $G/N$ is given by a subgroup $H$ of $G$ such that $N\leq H$ and $\tilde{H}=\pi(H)$, where $\pi:G\to G/N$ is the quotient map. More precisely, there is a 1-1 correspondence between the set of subgroups of $G/N$ and the set of subgroups of $G$ larger than $N$.
\end{exercise}

\begin{exercise}
Suppose $N$ is a normal subgroup of $G$. Whether or not $G\cong N\times G/N$?
\end{exercise}

\begin{exercise}
  Suppose $M$ and $N$ are normal subgroups of $G$ and $G=MN$. Prove that $G/(M\cap N)\cong G/M\times G/N$. Apply this to show that
  $\mbz/mn\mbz\cong\mbz/m\mbz\times\mbz/n\mbz$ when $m$ and $n$ are coprime.
\end{exercise}

\begin{exercise}
 Prove that $\text{SL}_n(\mbr)$ is a normal subgroup of $\text{GL}_n(\mbr)$, and $\text{GL}_n(\mbr)/\text{SL}_n(\mbr)$ is isomorphic to $\mbr^\times$.
\end{exercise}

\begin{exercise}
  Consider the action of $\text{GL}_n(\mbr)$ on $\mbr^n$ given by $\alpha_A(x)=Ax$ for $A$ in $\text{GL}_n(\mbr)$ and $x$ in $\mbr^n$. Find $\ox$ for $x$ in $\mbr^n$.
\end{exercise}

\begin{exercise}
Suppose that $p$ is prime and $m$ is prime to $p$. Prove that $p$ is prime to $\binom{p^km}{p^k}$.
\end{exercise}

\begin{exercise}
For $n\geq 5$, suppose a normal subgroup $H$ of $S_n$ satisfies that $H\cap A_n=1$. What can we say about $H$?
\end{exercise}

\begin{exercise}
Prove that a nonabelian simple group is unsolvable.
\end{exercise}

\begin{exercise}
Prove that  for $n\geq 3$,  $A_n$ is the commutator subgroup of $S_n$.
\end{exercise}

\begin{exercise}
Fill in details in the proof of Proposition~\ref{prop:solvable}.
\end{exercise}

\begin{exercise}
~\label{ex: Solvable}
Prove that a finite group $G$ is  solvable iff there exists a sequence of subgroups $\{G_i\}_{i=1}^n$ of $G$ such that
  $$1=G_n\lhd G_{n-1}\lhd G_{n-2}\lhd\cdots G_{1}\lhd G_0=G,$$ and  $G_{i-1}/G_i$ is cyclic for every $1\leq i\leq n$.
\end{exercise}

\chapter{Basic Ring Theory}

Ring theory, together with group theory and field theory are pillars of algebra. The abstract definition of ring were given only as late as the first decade of the 20th century and the abstract ring theory formed in the hands of E. Noether and E. Artin during the same period, though some special rings such as polynomial rings, rings of algebraic integers were studied quite thoroughly because of their importance in number theory and algebraic geometry.

In this chapter, we focus on commutative ring theory, in particular, theory of polynomial rings.
\section{Definition and examples}

\begin{definition}
  A \textbf{ring} is a nonempty set $R$ together with two binary operations: addition $``+"$ and multiplication $``\cdot"$~\footnote{We write $a\cdot b$ as $ab$ for $a,b$ in $R$.} such that
  \begin{enumerate}
  \item $R$ is closed under these two operations.
    \item $(R,+)$ is an abelian group.
    \item $(ab)c=a(bc)$ for all $a,b,c$ in $R$.
    \item $(a+b)c=ac+bc$ and $a(b+c)=ab+ac$ for all $a,b,c$ in $R$.
  \end{enumerate}
  If $ab=ba$ for all $a,b$ in $R$, then $R$ is called a \textbf{commutative ring}.~\index{commutative ring}
  ~\index{ring}

If $R$ has an identity  for multiplication, that is, an element $1$ in $R$ such that $1a=a1=a$ for all $a$ in $R$, then $R$ is called a \textbf{unital ring}.


If a ring $R$ is unital and $(R\setminus\{0\},\cdot)$ is an abelian group, then $R$ is called a \textbf{field}.~\index{field}
\end{definition}
If a subset of a ring $R$ is a ring, then it is called a \textbf{subring} of $R$.~\index{subring} If a subset of a field $\mbf$ is a field, then it is called a \textbf{subfield} of $\mbf$.~\index{subfield}


Below are some examples of rings.

\begin{example}\

  \begin{enumerate}
  \item $\mbz$, $\mbq$, $\mbr$ and $\mbc$ are unital commutative rings under the addition and the multiplication of numbers. Moreover $\mbq$, $\mbr$ and $\mbc$ are fields.
   \item For $n\geq1$, $\mbz/n\mbz$ is a unital commutative ring under addition and multiplication  modulo $n$. If $p$ is prime, then $\mbz/p\mbz$ is a field denoted by $\mbf_p$.
  \item $\mbz[i]=\{a+bi\,|\, a,b\in\mbz\}$ is a unital commutative ring under the addition and the multiplication of numbers. This ring is called the ring of \textbf{Gauss integers} and is of particular interest in number theory.~\index{Gauss integer}
  \item Under matrix addition and multiplication, the set of $n\times n$ real matrices $\text{M}_n(\mbr)$ is a noncommutative unital ring.
  \item The set of polynomials over a unital commutative ring $R$, $R[x]=\{\displaystyle\sum_{i=0}^n a_ix^i=a_nx^n+a_{n-1}x^{n-1}+\cdots+a_1x+a_0\,|\,a_i\in R \,\,\text{for\,all}\,\,0\leq i\leq n\}$ is a commutative ring under the addition and multiplication given by
        $$\sum_{i=0}^n a_ix^i+\sum_{i=0}^n b_ix^i=\sum_{i=0}^n (a_i+b_i)x^i$$ and $$(\sum_{i=0}^n a_ix^i)(\sum_{j=0}^m b_jx^j)=\sum_{k=0}^{m+n}(\sum_{i=0}^k a_ib_{k-i})x^k.$$ Special cases include $\mbz[x]$ and $\pf$ for a field $\mbf$.

  \end{enumerate}
\end{example}

An element $a$ in a ring $R$ is called a \textbf{unit} if there exists $b$ in $R$ such that $ab=ba=1$. The element $b$ is unique, called the (multiplicative) inverse of $a$ and denote it by $a^{-1}$. Denote the set of units of $R$ by $U(R)$. For example $U(\mbz)=\{\pm 1\}$ and $U(\text{M}_n(\mbr))=\text{GL}_n(\mbr)$.~\index{ring unit}

A nonzero element $a$ in a commutative ring $R$ is called a \textbf{zero divisor} if there exists nonzero $b$ in $R$ such that $ab=0$.~\index{zero divisor}

A unital commutative ring $R$ without zero divisors is called an \textbf{integral domain}. For example fields and $\mbz$ are integral domains, and $\mbz/6\mbz$ is not an integral domain.~\index{integral domain}

\begin{theorem}
  Let $R$ be a unital commutative ring. Then $R$ is an integral domain iff $R[x]$ is an integral domain.
\end{theorem}
\begin{proof}
Suppose $R$ is an integral domain. For any nonzero $f=a_nx^n+a_{n-1}x^{n-1}+\cdots+a_1x+a_0$ and $g=b_mx^m+b_{m-1}x^{m-1}+\cdots+b_1x+b_0$ in $R[x]$, it holds that $fg\neq 0$ since $a_nb_m\neq 0$.

If $R[x]$ is an integral domain, then $R$ is a unital subring of $R[x]$, hence also an integral domain.
\end{proof}

A nonempty subset $I$ of a ring $R$ is called an \textbf{ideal} of $R$ if $I$ is an additive subgroup of $R$  and for every $a$ in $I$ and $b$ in $R$, $ab$ and $ba$ are in $I$.~\index{ideal} An ideal $I$ of $R$ is called proper if $I$ is a proper subset of $R$.~\index{proper ideal} Every ring $R$ has two trivial ideals: $\{0\}$, $R$.

Given elements $a_1,a_2,\cdots, a_n$ in  a unital commutative ring $R$, the set $$\{\displaystyle\sum_{i=1}^{n}a_ib_i\,|\,b_i\in R\, \,\text{for\,all}\,1\leq i\leq n\}$$ is an ideal of $R$ called the ideal generated by $a_1,a_2,\cdots, a_n$. Denote it by $( a_1,\cdots, a_n)$. For example, the subset $I=\{\sum_{i=1}^n a_ix^i\,|\, a_i\in \mbf\, \text{for\,all\,} 1\leq i\leq n\}$ of $\pf$ is the ideal generated by the polynomial $x$. Later we will prove that every ideal of $\pf$ is generated by a single polynomial.

Suppose $I$ and $J$ are ideals of a commutative ring $R$, then the set $IJ=\{\sum_{i=1}^n a_ib_i\,|\,a_i\in I\,,b_i\in J\}$ is an ideal of $R$.

\begin{proposition}

If an ideal $I$ of a ring $R$ contains a unit, then $I=R$.
\end{proposition}
\begin{proof}
Take a unit $a$ in $I$.  Every $b$ in $R$ can be expressed as $b=b(a^{-1}a)=(ba^{-1})a$, hence is in $I$.
\end{proof}

Suppose $I$ is an ideal of a  ring $R$. Define the set  $R/I=\{a+I\,|\,a\in R\}$, the addition on $R/I$  by $(a+I)+(b+I)=a+b+I$ and the multiplication on $R/I$ by $(a+I)(b+I)=ab+I$ for $a,b$ in $R$.

\begin{theorem}
$R/I$ is a ring.
\end{theorem}
\begin{proof}
If $a_1+I=a_2+I$ and $b_1+I=b_2+I$, then $a_1+b_1+I=a_2+b_2+I$ since $(a_1+b_1)-(a_2+b_2)$ is in $I$. So the addition is well-defined.

Moreover under  the addition, $R/I$ is an abelian group:

\begin{itemize}
  \item  $(a+I)+(b+I)=a+b+I=b+a+I=(b+I)+(a+I)$.
  \item $((a+I)+(b+I))+c+I=(a+b)+c+I=a+(b+c)+I=a+I+((b+I)+(c+I))$.
  \item $(a+I)+I=a+0+I=I+(a+I)$.
  \item $(a+I)+(-a+I)=0+I=I$.
\end{itemize}
 If $a+I=b+I$, then $(a+I)(c+I)=ac+I=bc+I=(b+I)(c+I)$ since $ac-bc=(a-b)c$ is in $I$. So the multiplication is well-defined.

Also we have the following:
\begin{itemize}
  \item $((a+I)(b+I))(c+I)=(ab)c+I=a(bc)+I=(a+I)((b+I)(c+I))$.
  \item $((a+I)+(b+I))(c+I)=(a+b)c+I=(ac+I)+(bc+I)=(a+I)(c+I)+(b+I)(c+I)$.
  \item $(c+I)((a+I)+(b+I))=c(a+b)+I=(ca+I)+(cb+I)=(c+I)(a+I)+(c+I)(b+I)$.
\end{itemize}
So $R/I$ is a ring under the defined addition and multiplication.
\end{proof}
$R/I$ is called the \textbf{quotient ring} of $R$ over $I$.~\index{quotient ring}

Ring homomorphisms are defined similarly as group homomorphisms.

\begin{definition}
  A \textbf{ring homomorphism} is a map $\varphi: R\to S$ such that $\varphi(a+b)=\varphi(a)+\varphi(b)$ and $\varphi(ab)=\varphi(a)\varphi(b)$ for all $a,b$ in $R$.~\index{ring homomorphism}

  A bijective ring homomorphism is called a \textbf{ring isomorphism}.~\index{ring isomorphism}

  The \textbf{kernel} of a ring homomorphism $\varphi: R\to S$, denoted by $\ker\varphi$ is the set$\{a\in R\,|\, \varphi(a)=0\}$.~\index{kernel of ring homomorphism}

\end{definition}

Like group theory, there are isomorphism theorems for rings.

\begin{theorem}~[1st ring isomorphism theorem]\

Suppose $f:R\to S$ is a ring homomorphism. Then $\ker f$ is an ideal of $R$,  $\text{Im}f$ is a subring of $S$, and
  $$R/\ker f\cong \text{Im}f.$$
\end{theorem}

The proof of the 1st ring isomorphism theorem is similar to proof of the 1st group isomorphism theorems. We left it as an exercise.

 Consider $f:\mbz[x]\to\mbz$ given by $f(p)=p(0)$ for all $p$ in $\mbz[x]$. Then $\ker f=( x)$ and $\text{Im}f=\mbz$. The 1st ring isomorphism theorem says that $\mbz[x]/( x)\cong\mbz$.

\begin{theorem}~[2nd ring isomorphism theorem]\

Suppose $A$ is a subring of $R$ and $I$ is an ideal of $R$. Then $A+I$ is a subring of $R$, $A\cap I$ is an ideal of $A$ and
  $$(A+I)/I\cong A/(A\cap I).$$
\end{theorem}

\begin{proof}

For $a_1,a_2$ in $A$ and $b_1,b_2$ in $I$, one has that $(a_1+b_1)+(a_2+b_2)=(a_1+a_2)+(b_1+b_2)$ is in $A+I$ and $(a_1+b_1)(a_2+b_2)=a_1a_2+(b_1a_2+b_2a_1+b_1b_2)$ is in $A+I$. Hence  $A+I$ is a subring of $R$.

Moreover for any $a$ in $A$ and any $b$ in $A\cap I$, it holds that $ab$ is in $A\cap I$. So $A\cap I$ is an ideal of $A$.

  Define $f: A+I\to A/(A\cap I)$ by $f(a+b)=a+A\cap I$ for all $a$ in $A$ and $b$ in $I$.
  \begin{itemize}
    \item If $a_1+b_1=a_2+b_2$ for $a_1,a_2$ in $A$ and $b_1,b_2$ in $I$, then $a_1-a_2=b_2-b_1$ is in $A\cap I$. So $f$ is well-defined.
    \item For $a_1,a_2$ in $A$ and $b_1,b_2$ in $I$, $f((a_1+b_1)(a_2+b_2))=f(a_1a_2+a_1b_2+b_1(a_2+b_2))=a_1a_2+A\cap I=f(a_1+b_1)f(a_2+b_2)$ since $a_1b_2+b_1(a_2+b_2)$ is in $I$, and $f((a_1+b_1)+(a_2+b_2))=f((a_1+a_2)+(b_1+b_2))=a_1+a_2+A\cap I=f(a_1+b_1)+f(a_2+b_2)$. So $f$ is a ring homomorphism.
    \item $f$ is surjective since $f(a+0)=a+A\cap I$ for all $a$ in $A$.
    \item If $f(a+b)=0$ for $a$ in $A$ and $b$ in $I$, then $a$ is in $A\cap I$, and $a+b$ is in $I$. So $\ker f=I$.
  \end{itemize}
  By the 1st ring isomorphism, we complete the proof.
\end{proof}

Let $R=\mbz[x]$, $A=(x)$ and $I=(2)$. Then $A+I=( x,2)$ and $A\cap I=(2x)$. The 2nd ring isomorphism theorem says that $( x,2)/( 2)\cong ( x)/( 2x)$.

\section{The fraction field of an integral domain}

One can define a field, called the fraction field, out of an integral domain.

Suppose $R$ is an integral domain. Define a relation $``\sim"$ on $R\times R$ as follows
$(a,b)\sim(c,d)$ if $ad=bc$.

One can check that $``\sim"$ is an equivalence relation, that is, the relation $``\sim"$ satisfies that
\begin{enumerate}
  \item $(a,b)\sim(a,b)$.
  \item $(a,b)\sim(c,d)$ implies that $(c,d)\sim(a,b)$.
  \item  $(a,b)\sim(c,d)$ and $(c,d)\sim(e,f)$ implies that $(a,b)\sim(e,f)$.
\end{enumerate}

for $a,b$ in $R$ with $b\neq 0$,denote the equivalence class of $(a,b)$ in $R\times R$ by $ab^{-1}$, i.e.,
$$ab^{-1}=\{(c,d)\in R\times R\,|\,(c,d)\sim(a,b)\}.$$

Define $\text{Frac}(R)$ to be the set of equivalence classes in $R\times R$. Define the multiplication and the addition on $\text{Frac}(R)$ as $(ab^{-1})(cd^{-1})=ac(bd)^{-1}$ and $ab^{-1}+cd^{-1}=(ad+bc)(bd)^{-1}$.

\begin{theorem}
For an integral domain $R$, the set $\text{Frac}(R)$ is a field.
\end{theorem}
\begin{proof}
  First we check the addition and the multiplication on $\text{Frac}(R)$ are well-defined.

  Suppose that $(a_1,b_1)\sim(a_2,b_2)$ and $(c_1,d_1)\sim(c_2,d_2)$. Then
  $$(a_1c_1)(b_2d_2)=(a_1b_2)(c_1d_2)=(b_1a_2)(c_2d_1)=(a_2c_2)(b_1d_1). $$

  That is, $(a_1c_1)(b_1d_1)^{-1}=(a_2c_2)(b_2d_2)^{-1}$. So the multiplication is well-defined.

  Moreover
  \begin{align*}
   &(a_1d_1+b_1c_1)(b_2d_2)-(a_2d_2+b_2c_2)(b_1d_1)   \\
   &=(a_1b_2d_1d_2+b_1b_2c_1d_2)-(a_2b_1d_2d_1+b_2b_1c_2d_1) \\
    &=(a_1b_2-a_2b_1)d_1d_2+b_1b_2(c_1d_2-c_2d_1)=0.
  \end{align*}
 That is, $(a_1d_1+b_1c_1)(b_1d_1)^{-1}=(a_2d_2+b_2c_2)(b_2d_2)^{-1}$. So the addition is well-defined.

 Also the following hold:

 \begin{itemize}
   \item $ab^{-1}+cd^{-1}=(ad+bc)(bd)^{-1}=(cb+da)(db)^{-1}=cd^{-1}+ab^{-1}$.
   \item $(ab^{-1}+cd^{-1})+ef^{-1}=(ad+bc)(bd)^{-1}+ef^{-1}=(adf+bcf+bde)(bdf)^{-1}=ab^{-1}+(cf+de)(df)^{-1}=ab^{-1}+(cd^{-1}+ef^{-1})$.
   \item $ab^{-1}+0d^{-1}=(ad)(bd)^{-1}=ab^{-1}$, i.e., $0d^{-1}=0$ in  $\text{Frac}(R)$.
   \item $ab^{-1}+(-a)b^{-1}=0(b^2)^{-1}=0$.
 \end{itemize}
So $\text{Frac}(R)$ is an abelian group under addition.

Notice that $ab^{-1}=0d^{-1}$ iff $ad=0$ iff $a=0$ since $d\neq 0$ and $R$ is an integral domain.

Consider nonzero elements in $\text{Frac}(R)$. We have the following:
\begin{itemize}
   \item If $ab^{-1}\neq 0$ and $cd^{-1}\neq 0$, then $(ac)(bd)^{-1}\neq 0$ since $R$ is an integral domain.
   \item $(ab^{-1})(cd^{-1})=(ac)(bd)^{-1}=(ca)(db)^{-1}=(cd^{-1})(ab^{-1})$.
   \item $((ab^{-1})(cd^{-1}))ef^{-1}=((ac)(bd)^{-1})(ef^{-1})=(ace)(bdf)^{-1}=(ab^{-1})((ce)(df)^{-1})=ab^{-1}((cd^{-1})(ef^{-1}))$.
   \item $(ab^{-1})(dd^{-1})=ad(bd)^{-1}=ab^{-1}$, i.e., $dd^{-1}=1$ in  $\text{Frac}(R)$.
   \item $(ab^{-1})(ba^{-1})=ab(ab)^{-1}=1$.
 \end{itemize}

These verify that $\text{Frac}(R)\setminus\{0\}$ is an abelian group under multiplication.
\end{proof}

 $\text{Frac}(R)$ is called the \textbf{fraction field} of $R$.~\index{fraction field of an integral domain}

 We left the proof the following theorem as an exercise.
\begin{theorem}
~\label{thm:fraction field}
An integral domain $R$ is isomorphic to the subring $\{a1^{-1}\,|\,a\in R\}$ of $\text{Frac}(R)$. Moreover if $R$ is a subring of a field $\mbf$, then $\text{Frac}(R)$ is isomorphic to a subfield of $\mbf$.
\end{theorem}
\begin{proof}
  Define $f:R\to \text{Frac}(R)$ by $f(a)=a1^{-1}$ for all $a$ in $R$. Then
  \begin{itemize}
    \item $f$ is well-defined.
    \item $f(ab)=(ab)1^{-1}=(a1^{-1})(b1^{-1})=f(a)f(b)$ for all $a,b$ in $R$. That is, $f$ is a ring homomorphism.
    \item If $f(a)=a1^{-1}=0$, i.e., $a1^{-1}\sim 01^{-1}$, then $a=0$. Hence $f$ is injective.
  \end{itemize}
By the 1st ring isomorphism theorem, $R$ is isomorphic to $\text{Im}f$, which is a subring of $\text{Frac}(R)$.

This ends the proof of part 1 of the theorem.

Proof of part 2 is left as an exercise.
\end{proof}

\begin{corollary}
\label{cor:frac}
For every $c$ in $\text{Frac}(R)$, there is $b$ in $R$ such that $cb$ is in $R$. Here $R$ is taken as a subring of $\text{Frac}(R)$.
\end{corollary}
\begin{proof}
Denote $c$ by $ab^{-1}$ for $a,b$ in $R$. Then $ab^{-1}b1^{-1}=(ab)b^{-1}=a1^{-1}$. That is, $cb$ is in $R$.
\end{proof}

\begin{example}\
  \begin{enumerate}
    \item $\text{Frac}(\mbz)\cong\mbq$.
    \item $\text{Frac}(\mbz[\sqrt{5}])\cong\mbq[\sqrt{5}]=\{a+b\sqrt{5}\,|\,a,b\in\mbq\}$.
    \item $\text{Frac}(\pf)\cong\{\frac{f}{g}\,|\, f,g\in \pf\,,\,g\neq 0\}$.
  \end{enumerate}
\end{example}

\section{Prime ideals and maximal ideals}

\begin{definition}
  A  proper ideal $I$ of a ring $R$ is called \textbf{prime} if $ab\in I$ implies that $a\in I$ or $b\in I$. ~\label{prime ideal}

  A proper ideal $I$ of a ring $R$ is called \textbf{maximal} if $R$ is the only ideal is strictly larger than $I$.~\label{maximal ideal}
\end{definition}

\begin{example}\

  \begin{enumerate}
    \item A field $\mbf$ has a unique prime(maximal) ideal $\{0\}$.
    \item For a prime $p$, $p\mbz$ is a prime(maximal) ideal of $\mbz$.
    \item The ideal generated by $x$ is a prime(maximal) ideal of $\pf$.
  \end{enumerate}
\end{example}

\begin{theorem}\

~\label{thm: max}
 Every proper ideal $I$ of a unital commutative ring $R$ is contained in a maximal ideal of $R$.
\end{theorem}
\begin{proof}
The proof relies on Zorn's lemma, which says that if every completely ordered subset of a partially ordered set has an upper bound, then the partially ordered set has a maximal element.

Suppose $I$ is a proper ideal of $R$.

Now we consider the set $\mathcal{A}$ of proper ideals of $R$ containing $I$ which is ordered  by inclusion, i.e., in $\mathcal{A}$, we define $J\prec J'$ by $J\subseteq J'$. For a completely ordered subset $\mathcal{B}$ of $\mathcal{A}$, let $J_0=\displaystyle\bigcup_{J\in\mathcal{B}}J$.

For any $a,b$ in $J_0$, their exist $J_1,J_2$ in $\mathcal{B}$ such that $a\in J_1$ and $b\in J_2$. Since $\mathcal{B}$ is completely ordered, we have $J_1\prec J_2$ or  $J_2\prec J_1$. Both $a$ and $b$ locate in only one of $J_1$ and $J_2$, so does $a+b$. For any $c$ in $R$, $ac$ is in $J_1$. Therefore $J_0$ is an ideal of $R$. Note that $1\notin J_0$, so $J_0\subsetneqq R$. Thus $J_0$ is in $\mathcal{A}$ and $\displaystyle J\prec J_0$, that is, $J_0$ is an upper bound of $\mathcal{B}$. So there is a maximal element in $\mathcal{A}$, which is a maximal ideal of $R$ containing $I$.
\end{proof}

\begin{corollary}\
\label{cor:max}
Every nonunit in a unital commutative ring $R$ is contained in a maximal ideal of $R$.
\end{corollary}
\begin{proof}
Apply Theorem~\ref{thm: max} to the ideal generated by the nonunit.
\end{proof}

\begin{theorem}\

Suppose $R$ is a unital commutative ring.
\begin{enumerate}
\item $I$ is a prime ideal of $R$ iff $R/I$ is an integral domain.
\item $I$ is a maximal ideal of $R$ iff $R/I$ is a field.
\end{enumerate}
\end{theorem}

\begin{proof}\

(1) $R/I$ is an integral domain iff for $a,b$ in $I$, $(a+I)(b+I)=ab+I=0$ implies that $a+I=0$ or $b+I=0$. Equivalently, $ab$ is in $I$ implies $a$ or $b$ is in $I$.

(2) Suppose $R/I$ is a field. Assume that $I$ is a proper subset of an ideal $J$. So $J/I$ is a nonzero ideal of $R/I$. Since $R/I$ is a field, $J/I=R/I$. So $R=J$. Therefore $I$ is a maximal ideal of $R$.

Conversely assume that $I$ is a maximal ideal of $R$. For any $a\notin I$, the ideal generated by $I$ and $a$, $I+( a)=\{b+ac\,|\, b\in I,\,c\in R\}$ is an ideal such that $I\subsetneqq I+( a)$. Hence $I+( a)=R$. So there exist $c$ in $I$ and $b$ in $R$ such that $c+ab=1$. That is, $a+I$ is a unit of $R/I$. Hence $R/I$ is a field.

\end{proof}

\begin{corollary}
\label{cor:maxprime}
A maximal ideal of a unital commutative ring is a prime ideal.
\end{corollary}

\begin{remark}
  A prime ideal is not necessarily maximal. For example, $( x)$ is a prime ideal of $\mbz[x]$ since $\mbz[x]/( x)\cong\mbz$, but $( x)$ is not a maximal ideal since $ ( x)\subsetneqq ( x,2)\subsetneqq \mbz[x]$.
\end{remark}

\section{Polynomials rings}

In this section, we study the polynomial ring over a field $\mbf$. In particular, we introduce Eulidean algorithm for polynomials.

The Euclidean algorithm for polynomials, as its name indicates, is an analogue of Eulidean algorithm for integers.

The \textbf{degree} of a polynomial $f(x)=a_nx^n+a_{n-1}x^{n-1}+\cdots+a_1x+a_0$ is the largest nonnegative integer $n$ such that $a_n\neq 0$. Denote it by $\deg f$.~\index{degree} For example, in $\mbr[x]$, $\deg(x^3-1)=3$ and $\deg(x^2-x+5)=2$.
\begin{theorem}

~\label{thm:Euclidean}
Given $f$ and nonzero $g$ in $\pf$,  there exist unique $h$ and $r$ in $\pf$ such that
$f=gh+r$ with  $\deg(r)<\deg(g)$.
\end{theorem}
\begin{proof}
We can assume that $\deg f\geq\deg g$ otherwise $h=0$ and $f=r$.

Denote $f$ by $f=a_nx^n+a_{n-1}x^{n-1}+\cdots+a_1x+a_0$ and $g$ by $b_mx^m+b_{m-1}x^{m-1}+\cdots+b_1x+b_0$ with $\deg f=n\geq m=\deg g$.

Let $h_1=a_nb_m^{-1}x^{n-m}$ and $r_1=f-gh_1$. Then $\deg(r_1)\leq n-1<n=\deg(f)$. If $\deg r_1<\deg g$, then we get that $h=h_1$ and $r=r_1$.
If $\deg r_1\geq\deg g$, then repeat the above process for $r_1$ and $g$ to get $r_2$. After finite steps, say $k$ steps, either $r_k=0$ or $\deg r_k<\deg g$. Then we can get the required $h$ and $r$ by repeatedly plugging the identities from bottom to top. More precisely $h=h_1+h_2+\cdots+h_{k-1}$ and $r=r_k$.

To summarize, what we do are the following:
\begin{align*}
  &f=gh_1+r_1 \\
  &r_1=gh_2+r_2 \\
  &\cdots\, \cdots \, \cdots \\
  &r_{k-1}=gh_{k-1}+r_k
\end{align*}
and $$f=g(h_1+h_2+\cdots+h_{k-1})+r_k.$$ We prove the existence of $h$ and $r$.

For the uniqueness, if $f=gh+r=g\tilde{h}+\tilde{r}$ with $\deg r<\deg g$ and $\deg\tilde{r}<\deg g$, then $g(h-\tilde{h})=\tilde{r}-r$. It holds that $h=\tilde{h}$ and $r=\tilde{r}$, otherwise $\deg(g(h-\tilde{h}))\geq\deg g>\deg(\tilde{r}-r)$ which leads to a contradiction.

\end{proof}

The polynomials  $r$  is called the~\textbf{reminder} of $f$ divided by $g$.~\index{reminder}

If $f=gh$ in $\pf$, then we say that $g$ divides $f$ or $g$ is a \textbf{factor} of $f$ denoted by $g| f$.~\index{factor}

A \textbf{common divisor} of $f$ and $g$ is a polynomial $h$ such that $h$ divides both $f$ and $g$.~\index{common divisor}

If a common divisor $h$ of $f$ and $g$ satisfies that any other common divisor of $f$ and $g$ is a factor of $h$, then it is called a \textbf{greatest common divisor} of $f$ and $g$.~\index{greatest common divisor}

Up to scalar multiples, greatest common divisors of $f$ and $g$ are unique. So we call it the greatest common divisor of $f$ and $g$ and denote it by $\gcd(f,g)$.

If $\gcd(f,g)=1$, then we say that $f$ and $g$ are \textbf{coprime} or $f$ is prime to $g$.~\index{coprime}

\begin{lemma}\
~\label{lm:Euclidean}
Suppose $f$ and $g$ are in $\pf$. Then  $\gcd(f,g)=fp+gq$ for some $p,q$ in $\pf$. In another word, the ideal $( f,g)$ generated by $f$ and $g$ is $( \gcd(f,g))$, the ideal generated by $\gcd(f,g)$.
\end{lemma}
\begin{proof}

Without loss of generality, we assume that $\deg g\leq \deg f$.

By Theorem~\ref{thm:Euclidean}, $f=gh_0+r_0$ with $\deg r_0<\deg g$. If $r_0\neq 0$, then $g=h_1r_0+r_1$ with $\deg r_1<\deg r_0$. Again if $ r_1\neq 0$, then $r_0=h_2r_1+r_2$. Do this inductively whenever $r_i\neq 0$, after finite steps,  we get that $r_{k-2}=h_{k-1}r_{k-1}+r_k$ and $r_{k-1}=h_kr_k$.

Note that $r_0$ is in $( f,g)$. Inductively each $r_i$ is in $( f,g)$ for $0\leq i\leq k$. Moreover if $r$ divides $f$ and $g$, then $r$ divides $r_k$ since $r_k$ is in $( f,g)$. Hence $r_k=\gcd(f,g)$ and $r_k$ is of the form  $h=fp+gq$ for some $p,q$ in $\pf$.
\end{proof}

\begin{remark}
In the proof of Theorem~\ref{thm:Euclidean} and Lemma~\ref{lm:Euclidean}, algorithms for finding the reminder $r$ of $f$ divided $g$ and $\gcd(f,g)$ are  called the \textbf{Euclidean algorithm} for polynomials.~\index{Euclidean algorithm}
 \end{remark}

 \section{Principal ideal domains}

A nonzero nonunit $f$  in a unital commutative ring $R$  is called~\textbf{irreducible} if $f=gh$ implies that $g$ or $h$ is a unit in $R$. Otherwise call $f$ \textbf{reducible}.~\index{irreducible element}\index{reducible}
\begin{definition}
  A \textbf{principal  ideal domain} is an integral domain $R$ in which every ideal is a~\textbf{principal ideal}, that is,  an ideal generated by a single element in $R$.~\index{principal  ideal domain}~\index{principal  ideal}
\end{definition}

$\mbz$ is a principal  ideal domain.

$\mbz[x]$ is not a principal  ideal domain since the proper ideal $( x,2)$ is not principal.
\begin{theorem}
\label{thm:polypid}
$\pf$ is a principal ideal domain.
\end{theorem}
\begin{proof}
Suppose $I$ is a nonzero ideal of  $\pf$. Assume that $f$ is a nonzero element in $I$ with the smallest degree.

Take any $g$ in $I$. By Theorem~\ref{thm:Euclidean}, there exist $h$ and $r$ in  $\pf$ such that
$g=fh+r$ with $\deg(r)<\deg (f)$. Then $r=0$ otherwise $r$ is a nonzero element in $I$ with the degree less than the degree of $f$ which is a contradiction. Hence $I=( f)$.
\end{proof}

\begin{theorem}
 In a principal ideal domain, a nonzero ideal is prime iff it is maximal.
\end{theorem}

\begin{proof}
  The necessity follows from Corollary~\ref{cor:maxprime}.

  We prove the sufficiency.

  Assume that $I$ is a prime ideal of a principal ideal domain $R$.

  Suppose that an ideal $J=(b)\supsetneqq I=(a)$.  Then there exists $c$ in $R$ such that $a=bc$. Since $I$ is prime and $b$ is not in $I$, the element $c$ is in $I$. There is $d$ in $R$ such that $c=ad$. Hence $a(bd-1)=0$. The element $a$ is nonzero since $I=(a)$ is a nonzero ideal of $R$. Thus $bd=1$, which means, $b$ is a unit. Therefore $J=R$ which implies that $I$ is a maximal ideal.
\end{proof}


\begin{proposition}
\label{prop: primeirreducible}
If a principal ideal $(a)$ in an integral  domain $R$ is prime, then $a$ is irreducible. If $a$ is an irreducible element in  a principal ideal domain $R$, then  the ideal $(a)$ is prime.
\end{proposition}

\begin{proof}
Suppose $a=bc$ for $b,c$ in $R$. Then either $b$ or $c$ is in $(a)$. Assume that $b$ is in $(a)$. Then $b=ad$ for $d$ in $R$. It follows that
$a=bc=adc$, which implies that $dc=1$. So $c$ is a unit and  $a$ is an irreducible in $R$.

Assume that $a$ is an irreducible element in  a principal ideal domain $R$. Suppose $(a)$ is a proper subset of an ideal $J=(b)$. Then $a=bc$ for some $c$ in $R$. Since $a$ is irreducible, either $b$ or $c$ is a unit. But $c$ cannot be a unit otherwise $(a)=(b)$. Thus $b$ is a unit and $J=R$. This shows that $(a)$ is a maximal ideal, hence a prime ideal.
\end{proof}

As a consequence we get the following which says that irreducibles in a principal ideal domain $R$ behave like primes in the ring  $\mbz$ of integers.

\begin{corollary}
\label{cor:divide}
For an irreducible $p$ in a principal ideal domain $R$, if $p| ab$, then either $p|a$ or $p|b$.
\end{corollary}

\begin{definition}
A sequence $\{I_n\}_{n=1}^\infty$ of ideals of a ring $R$ is called \textbf{ascending} if $I_n\subseteq I_{n+1}$ for all $1\leq i\leq n$.  We say that an ascending sequence $\{I_n\}_{n=1}^\infty$ of ideals satisfies the \textbf{ascending chain condition}  if there exists $K$ such that $I_k=I_{k+1}$ for all $k\geq K$.~\index{ascending sequence of ideals}~\index{ascending chain condition}
\end{definition}

\begin{lemma}
~\label{lm:ascendingchain}
Every ascending sequence of ideals in a principal ideal domain satisfies the ascending chain condition.
\end{lemma}
\begin{proof}
Suppose $\{I_n=(a_n)\}_{n=1}^\infty$ is an ascending sequence of ideals in a principal ideal domain $R$. Define $\displaystyle J=\bigcup_{n=1}^\infty I_n$. Then $J$ is an ideal and $J=(b)$ for some $b$ in $R$.

It follows that $b$ is in $I_K$ for some $K\geq 1$. Hence $(b)\subseteq I_K\subseteq I_{K+1}\subseteq\cdots\subseteq (b)$. Consequently we obtain that $I_k=I_{k+1}$ for all $k\geq K$.
\end{proof}

\begin{lemma}
~\label{lm:properideal}
Let $R$ be an integral domain.
\begin{enumerate}
\item If $a$, $b$ and $c$ are nonzero nonunits in $R$ and $a=bc$, then $(a)\subsetneqq (b)$.
\item  $(f)=(g)$ iff $f=gh$ for some unit $h$.
\end{enumerate}
\end{lemma}
\begin{proof}
 (1) Suppose $(a)=(b)$. We get that $b=ad$ for some $d$ in $R$. Then $a=bc=adc$ which implies that $c$ is a unit. This is a contradiction.

 (2)It follows from (1).
\end{proof}

\begin{theorem}
\label{thm:piddec}
Up to a multiplication of units, every nonzero nonunit in a principal ideal domain can be written uniquely as a product of finitely many irreducibles.
\end{theorem}
\begin{proof}
Suppose $a$ is a nonzero nonunit in a principal ideal domain $R$.

Then by Corollary~\ref{cor:max}, the element $a$ is in a prime(maximal) ideal of $R$. By Proposition~\ref{prop: primeirreducible}, there is an irreducible $p_1$ in $R$ such that $a=p_1a_1$ for some $a_1$ in $R$. Inductively we get $a_i=p_{i+1}a_{i+1}$ for all $i\geq 1$ when $a_i$ is a nonunit. However this process terminates after finite steps otherwise by Lemma~\ref{lm:properideal}, we can get an ascending sequence of ideals
$$(a)\subsetneqq( a_1)\subsetneqq ( a_2)\subsetneqq\cdots\subsetneqq ( a_n) \subsetneqq ( a_{n+1})\subsetneqq\cdots$$
which does not satisfies the ascending chain condition, which contradicts to Lemma~\ref{lm:ascendingchain}.

Hence $a$ can be written as a product of finitely many irreducibles.

If $a=p_1p_2\cdots p_k=q_1q_2\cdots q_l$ for $p_i's$ and $q_j's$ being irreducibles, then by Corollary~\ref{cor:divide}, each $p_i$ divides some $q_j$, which means that $p_i=b_i q_j$ for some unit $b_i$ in $R$. Hence $k=l$ and the decomposition of $a$ into a product of irreducibles is unique up to a multiplication of units.
\end{proof}

\section{Eisenstein's criterion and Gauss's lemma}\

In this section, we introduce Eisentein criterion and Gauss's lemma, which briefly speaking, are tools for finding irreducible polynomials in $\pf$.

Let $R$ be a unital commutative ring. An irreducible in $R[x]$ is called an \textbf{irreducible polynomial}.~\index{irreducible polynomial}

Theorem~\ref{thm:polypid} tells us that $\pf$ is a principal ideal domain.  Apply what we get for principal ideal domains, the following hold for $\pf$:

\begin{theorem}\

  \begin{enumerate}
    \item For an irreducible $p$ in $\pf$, if $p| fg$, then either $p|f$ or $p|g$.
    \item Up to scalar multiplications, every polynomial of degree greater than 0 in $\pf$ can be uniquely written as a product of irreducible polynomials.
    \item A nonzero ideal $(f)$ of $\pf$ is prime iff $f$ is irreducible.
  \end{enumerate}
\end{theorem}

Eisenstein's criterion together with Gauss's lemma are sufficient conditions  for a polynomial in $\pf$, in particular, in $\mbq[x]$ to be irreducible.~\cite{Eisentein1850}
\begin{theorem}[Eisentein criterion]
\index{Eisentein Criterion}
Let $R$  be an integral domain and $P$ is a prime ideal of $R$. If $f(x)=a_nx^n+a_{n-1}x^{n-1}+\cdots+a_1x+a_0$ in $R[x]$ satisfies that $a_n$ is not in $P$, $a_0, a_1, \cdots, a_{n-1}$ are in $P$ and $a_0$ is not in $P^2$, then $f$ is irreducible in $R[x]$.
\end{theorem}
\begin{proof}
Suppose $f(x)=g(x)h(x)$ for $g,h$ in $R[x]$ with $\deg(g)=k<n$ and $\deg(h)=l<n$. Assume that $g(x)=c_kx^k+c_{k-1}x^{k-1}+\cdots+c_1x+c_0$ and
$h(x)=d_lx^l+d_{l-1}x^{k-1}+\cdots+d_1x+d_0$ with $c_i$'s and $d_j$'s in $R$. We are going to prove that $c_0$ and $d_0$ are in $P$, which contradicts that $a_0=c_0d_0$ is not in $P^2$.

Note that $R[x]/P[x]\cong(R/P)[x]$. See Exercise~\ref{ex:poly}.

Hence $f+P[x]=a_nx^n+P[x]=(a_n+P)x^n$ by assumption. Moreover $g+P[x]=(c_k+P)x^k+(c_{k-1}+P)x^{k-1}+\cdots+(c_1+P)x+c_0+P$, $h+P[x]=(d_l+P)x^l+(d_{k-1}+P)x^{l-1}+\cdots+(d_1+P)x+d_0+P$ and $f+P[x]=(g+P[x])(h+P[x])$.

Hence
\begin{align*}
& c_0d_0+P=0 \\
& c_1d_0+c_0d_1+P=0 \\
&\cdots\,\,\,\cdots\,\,\,\cdots \\
& c_0d_l+c_1d_{l-1}+\cdots+c_ld_0+P=0
\end{align*}

If one of $c_0+P$ and $d_0+P$, say, $c_0+P\neq 0$, then $d_0+P=0$ and inductively $d_1+P=d_2+P=\cdots=d_{l-1}+P=d_l+P=0$. This contradicts that $c_kd_l+P=a_n+P\neq 0$. So $c_0+P$ and $d_0+P$ are zero, which means,  $c_0d_0$ is in $P^2$.
\end{proof}


\begin{theorem}[Gauss's lemma]

~\label{thm: Gauss}
\index{Gauss's lemma}
Suppose $R$ is a principal ideal domain with the fractional field $\mbf$. Then $f$ in $R[x]$ is irreducible in $R[x]$ iff $f$ is irreducible in $\pf$.
\end{theorem}
\begin{proof}
Suppose $f$ is an irreducible polynomial in $R[x]$ and $f=gh$ for $g(x)=a_nx^n+a_{n-1}x^{n-1}+\cdots+a_1x+a_0$ and $h(x)=b_mx^m+b_{m-1}x^{m-1}+\cdots+b_1x+b_0$ in $\pf$.

By Corollary~\ref{cor:frac}, there is $a,b$ in $R$ such that $g_1=ag$ and $h_1=bh$ are in $R[x]$. Denote $ab$ by $c$, and  $c=p_1p_2\cdots p_k$ for $p_j$'s being irreducibles in $R$ by Theorem~\ref{thm:piddec}. We have
\begin{align*}
p_1p_2\cdots p_kf=g_1h_1=&(\tilde{a}_nx^n+\tilde{a}_{n-1}x^{n-1}+\cdots+\tilde{a}_1x+\tilde{a}_0)  \\
&(\tilde{b}_mx^m+\tilde{b}_{m-1}x^{m-1}+\cdots+\tilde{b}_1x+\tilde{b}_0),
\end{align*}

where $\tilde{a}_i=a_ia$ and $\tilde{b}_j=b_jb$ for all $1\leq i\leq n$ and $1\leq j\leq m$. Then $p_1$ divides either all $\tilde{a}_i$'s or $\tilde{b}_j$'s by induction.

So $p_1$ can be canceled out and we get $p_2\cdots p_k=h_2g_2$ with $g_2,h_2$ in $R[x]$,  $\deg g_2=\deg g$ and $\deg h_2=\deg h$.

Repeat this process, we obtain that $f=g_kh_k$ with $g_k,h_k$ in $R[x]$, $\deg g_k=\deg g$ and $\deg h_k=\deg h$. Since $f$ is irreducible in $R[x]$, we  have either $\deg g_k=0$ or $\deg h_k=0$. Hence either $\deg g=0$ or $\deg h=0$. Therefore $f$ is irreducible in $\pf$.
\end{proof}

\begin{remark}\

(1) As a consequence of Eisenstein's criterion and Gauss's lemma, for every $n\geq 1$, there is an irreducible polynomial of degree $n$  in $\mbq[x]$, say, $x^n-2$.

(2) Gauss's lemma holds for polynomial ring over a unique factorization domain, which is more general than principal ideal domain. Interested readers may consult~\cite{DummitFoote2004} and \cite{Lang2002} for details.
\end{remark}

\section*{Exercises}
\begin{exercise}
Let $C_0(\mbr)=\{f:\mbr\to \mbr|\, f\,\text{is \,continuous\, and}\,\displaystyle\lim_{x\to\infty}f(x)=0\}$. The addition and multiplication on $C_0(\mbr)$ are defined as pointwise addition and multiplication of functions. Prove that $C_0(\mbr)$ is a nonunital commutative ring.
\end{exercise}

\begin{exercise}
  Suppose $I$ and $J$ are ideals of a commutative ring $R$. Prove that $IJ$ is an ideal of $R$, and $IJ\subseteq I\cap J$. Is it true that $IJ=I\cap J$?
\end{exercise}

\begin{exercise}
Prove that the kernel of a ring homomorphism is an ideal.
\end{exercise}

\begin{exercise}
Prove that there are only two automorphisms of the ring $\mbc$ which fix real numbers.
\end{exercise}

\begin{exercise}
Prove the 1st ring isomorphism theorems.
\end{exercise}

\begin{exercise}
Let $\mbf$ be a field. What's $U(\pf)$?  Justify your answer.
\end{exercise}

\begin{exercise}
Suppose $n\geq 2$. Write down $U(\mbz/n\mbz)$. What's the necessary and sufficient condition for $\mbz/n\mbz$ being an integral domain?  Justify your answer.
\end{exercise}

\begin{exercise}
Prove the 2nd part of Theorem~\ref{thm:fraction field}.
\end{exercise}

\begin{exercise}
Find an ideal in the ring $\mbz\times\mbz$ which is prime, but not maximal.
\end{exercise}

\begin{exercise}
Prove that every prime ideal of $\mbz$ is of the form $p\mbz$ for a prime $p$.
\end{exercise}

\begin{exercise}
Write down all prime ideals of $\mbz[\frac{1+\sqrt{5}}{2}]=\{a+b\frac{1+\sqrt{5}}{2}\,|\, a,b\in\mbz\}$.
\end{exercise}


\begin{exercise}
  Suppose $a$ is irreducible and $b$ is a unit in a unital commutative ring $R$. Prove that $ab$ is also irreducible.
\end{exercise}

\begin{exercise}
$f$ is irreducible in $\pf$ iff $f$  cannot be written as a product of two polynomials whose degrees are less than $f$.
\end{exercise}

\begin{exercise}
Suppose $f$ and $g$ are in $\pf$.  Prove that  $f$ divides $g$ and $g$ divides $f$ iff $f=\alpha g$ for some nonzero $\alpha$ in $\mbf$.
\end{exercise}

\begin{exercise}
Find the gcd of $x^3-2x^2+1$ and $x^2-x-3$ in $\mbq[x]$ and express it as a linear combination of them.
\end{exercise}

\begin{exercise}
Find the gcd of $x^4+1$ and $x^4+x^3+3x^2+2x+2$ in $\mbf_5[x]$ and express it as a linear combination of them.
\end{exercise}

\begin{exercise}
Prove that two distinct monic irreducible polynomials are coprime.
\end{exercise}

\begin{exercise}
~\label{ex:gcd}
  Suppose $p$ is an irreducible polynomial in $\pf$. Prove that $\gcd(p,f)=1$ or $p$ for any $f$ in $\pf$.
\end{exercise}

\begin{exercise}
Prove that a polynomial $p(x)$ in $\pf$ with $\deg(p)\leq 3$ is irreducible if and only if $p(x)$ has no root in $\mbf$. Is it true for $p$ with $\deg(p)>3$?
\end{exercise}

\begin{exercise}
Prove that for any prime $p$, the polynomial $x^p+p-1$ is irreducible in $\mbq[x]$.
\end{exercise}

\begin{exercise}
\label{ex:poly}
  Suppose that $P$ is a prime ideal of  a principal ideal domain $R$. Prove that $P[x]$ is a prime ideal of $R[x]$ and
  $$R[x]/P[x]\cong(R/P)[x].$$
\end{exercise}
\chapter{Basic Field Theory}

Though fields are  special cases of rings, the definition of field appeared earlier than that of ring. In this chapter, we deal with field extension theory, in particular, finite extensions of the rational field $\mathbb{Q}$, serving for the fundamental theorem of Galois theory.
\section{Finite fields}

Recall that a field is a unital commutative ring in which every nonzero is a unit.

Below are some examples of finite fields.

\begin{example}

For a prime $p$, $\mbz/p\mbz$ is a field. Denote it by $\mbf_p$.

Given an irreducible polynomial $f$ of degree $n$ in $\mbf_p[x]$. Then $\mbf_p[x]/( f)$ is a field since $( f)$ is a maximal ideal of $\mbf_p[x]$. Moreover $|\mbf_p[x]/( f)|=p^n$ since
$$\mbf_p[x]/( f)=\{a_{n-1}x^{n-1}+a_{n-2}x^{n-2}+\cdots+a_1x+a_0+( f)\,|\, a_i\in \mbf_p\,\text{for\,all\,}1\leq i\leq n-1\}.$$
For instance $x^2+x+1$ is an irreducible polynomial in $\mbf_2[x]$, then the field $\mbf_2[x]/( x^2+x+1)$ has 4 elements: $( x^2+x+1)$, $1+( x^2+x+1)$, $x+( x^2+x+1)$ and $x+1+( x^2+x+1)$. So $|\mbf_2[x]/( x^2+x+1)|=4$.
\end{example}

\begin{definition}
  The \textbf{characteristic} of a field $\mbf$ is the smallest positive integer $n$ such that $n=0$. Denote it by $\txch\mbf$. If every positive integer is nonzero, then we say $\mbf$ is of characteristic 0.
  ~\index{characteristic of a field}
\end{definition}

It is an exercise to check that for any field $\mbf$, either $\txch\mbf=0$ or $\txch\mbf$ is prime.

If $\txch\mbf=p$, then $pa=0$ for every $a$ in $\mbf$.

A finite field is of finite characteristic. Later we will see examples of infinite fields which are of finite characteristic.

One can see that $\mbq$, $\mbr$, $\mbc$ are all of characteristic 0, and $\txch\mbf_p=p$.

\begin{theorem}
  \label{thm:finitefield}
If $\mbf$ is a finite field and $\txch\mbf=p$, then $|\mbf|=p^n$ for some positive integer $n$.
\end{theorem}
\begin{proof}
  Suppose $|\mbf|=p^nm$ for some positive integer $m$ prime to $p$. If $m>1$, then $m$ has a prime factor $q$ distinct from $p$. Apply Cauchy's theorem to the additive group $\mbf$. There is an element $a$ in $\mbf$ of order $q$. So $qa=0$.

Since $q$ is prime to $p$, we have $p=lq+r$ with $0<r<q$. It follows that $0=pa=lqa+ra=ra$, which contradicts that the order of $a$ is $q$. This completes the proof.
\end{proof}

\section{Vector spaces}

In linear algebra, one deals with vector spaces over $\mbr$ or $\mbc$. Actually vector spaces could be defined over any field.
\begin{definition}
  A \textbf{vector space} $V$ over a field $\mbf$ is a nonempty set on which the operations of addition  and scalar multiplication are defined. By this we mean that, with each pair of elements $x$ and $y$ in $V$, we can associate a unique element $x+y$ that is also in $V$, and with each element $x$ in $V$ and each $\alpha$ in $\mbf$, we can associate a unique element $\alpha x$ in $V$. Moreover the following hold:
  \begin{enumerate}
    \item $(V,+)$ is an abelian group.
    \item $(\alpha\beta)x=\alpha(\beta x)$ for each $x$ in $V$ and all $\alpha,\beta$ in $\mbf$.
    \item $\alpha(x+y)=\alpha x+\alpha y$ and $(\alpha+\beta)x=\alpha x+\beta x$ for all $x,y$ in $V$ and all $\alpha,\beta$ in $\mbf$.
    \item $1x=x$ for each $x$ in $V$.
     \end{enumerate}
~\index{vector space}
An elements in a vector space is called a \textbf{vector}.~\index{vector} The field $\mbf$ is called the\textbf{ scalar field} of $V$.~\index{scalar field}

A \textbf{subspace} of a vector space  is a subset which itself is a vector space.~\index{subspace}
\end{definition}

For vectors $v_1,v_2,\cdots, v_n$ in a vector space $V$ and $\alpha_1,\alpha_2,\cdots, \alpha_n$ in $\mbf$, the expression $\alpha_1v_1+\alpha_2v_2+\cdots+\alpha_nv_n$ is called a \textbf{linear combination} of $v_1,v_2,\cdots, v_n$.~\index{linear combination}
The set of all linear combinations of $v_1,v_2,\cdots, v_n$ is called the \textbf{linear span} of $v_1,v_2,\cdots, v_n$, denoted by $\text{Span}(v_1,v_2,\cdots, v_n)$.~\index{linear span}

Vectors $v_1,v_2,\cdots, v_n$ in a vector space $V$ are called \textbf{linearly independent} if $\alpha_1v_1+\alpha_2v_2+\cdots+\alpha_nv_n=0$ implies that $\alpha_1=\alpha_2=\cdots=\alpha_n=0$. Otherwise $v_1,v_2,\cdots, v_n$ are called \textbf{linearly dependent}.~\index{linearly dependent} 

A subset $A$ of a vector space is called  linearly independent if any vectors  $v_1,v_2,\cdots, v_n$ in $A$ are linearly independent.  If a vector space $V$ has a  maximal linearly independent finite subset, then $V$ is called \textbf{finite dimensional}, otherwise \textbf{infinite dimensional}.

A set of vectors $\{v_1,v_2,\cdots, v_n\}$ in a finite dimensional vector space $V$ is called a \textbf{basis} of $V$ if they are linearly independent and $V=\text{Span}(v_1,v_2,\cdots, v_n)$.~\index{basis}

The proof of the following theorem is the same as the analogous result in linear algebra. We leave it as an exercise.

\begin{theorem}
\label{thm: dim}
 Any two bases of a vector space have the same cardinality.
\end{theorem}

For a finite dimensional vector space $V$ over a field $\mbf$, the cardinality of a basis of $V$ is called the \textbf{dimension} of $V$, denoted by $\dim_{\mbf} V$.~\index{dimension} Denote by $\dim_{\mbf} V=\infty$ when $V$ is infinite dimensional.  For example, $\dim_\mbf\{0\}=0$,
$\dim_{\mbr} \mbc=2$, $\dim_{\mbc} \mbc=1$, and $\dim_{\mbq} \mbr=\infty$.

If $\mbf$ is a subfield of $\mbe$, then $\mbe$ is a vector space over $\mbf$. Denote $\dim_{\mbf} \mbe$ by $[\mbe:\mbf]$.

\begin{theorem}
\label{thm: fielddim}
If $\mbf\subseteq\mbe\subseteq\mbk$ are fields, then $$[\mbk:\mbf]=[\mbk:\mbe][\mbe:\mbf].$$
\end{theorem}
\begin{proof}
 We may assume that $[\mbk:\mbe]$ and $[\mbe:\mbf]$ are finite since $[\mbk:\mbf]\geq [\mbk:\mbe]$ and $[\mbk:\mbf]\geq [\mbe:\mbf]$.

Let $\{v_i\}_{i=1}^n$ be a basis of $\mbe$ over $\mbf$ and $\{w_j\}_{j=1}^m$ be a basis of $\mbk$ over $\mbe$.

Suppose $\displaystyle\sum\alpha_{ij}v_iw_j=0$ for $\alpha_{ij}$'s in $\mbf$.
Then $\displaystyle\sum_{j=1}^m(\sum_{i=1}^{n}\alpha_{ij}v_i)w_j=0$. Since $\{w_j\}_{j=1}^m$ is linearly independent over $\mbe$, we have $\displaystyle\sum_{i=1}^{n}\alpha_{ij}v_i=0$ for every $1\leq j\leq m$. Furthermore since $\{v_i\}_{i=1}^m$ is linearly independent over $\mbf$, we have $\alpha_{ij}=0$ for all $1\leq i\leq n$ and $1\leq j\leq m$. So $\{v_iw_j\}_{1\leq i\leq n\,,1\leq j\leq m}$ is linearly independent over $\mbf$.

For every $v$ in $\mbk$, one has $v=\displaystyle\sum_{j=1}^{m}\beta_jw_j$ for $\beta_j's$ in $\mbe$. Moreover each $\beta_j=\displaystyle\sum_{i=1}^n \alpha_{ji}v_i$ for $\alpha_{ji}$'s in $\mbf$. Hence $v=\displaystyle\sum_{i,j}\alpha_{ji}v_iw_j$, i.e., the set $\{v_iw_j\}_{1\leq i\leq n\,,1\leq j\leq m}$ spans $\mbk$.

Therefore $\{v_iw_j\}_{1\leq i\leq n\,,1\leq j\leq m}$ is a basis of $\mbk$ over $\mbf$. Consequently
$$[\mbk:\mbf]=mn=[\mbk:\mbe][\mbe:\mbf].$$
\end{proof}

\section{Fields extensions}

\subsection{Fields extensions containing roots of polynomials}

If $\mbf$ is isomorphic to a subfield of $\mbe$, then $\mbe$ is called an \textbf{extension} of $\mbf$.~\index{extension} If $[\mbe:\mbf]<\infty$, then $\mbe$ is called a \textbf{finite extension} of $\mbf$, otherwise, an \textbf{infinite extension} of $\mbf$.~\index{finite extension}\index{infinite extension}

Assume that $\mbf$ is a subfield of $\mbe$ and $\alpha_1,\alpha_2,\cdots,\alpha_n$ are in $\mbe$. Denote by
$\mbf(\alpha_1,\alpha_2,\cdots,\alpha_n)$ the smallest subfield of $\mbe$ containing $\mbf$ and $\alpha_1,\alpha_2,\cdots,\alpha_n$. Call $\mbf(\alpha_1,\alpha_2,\cdots,\alpha_n)$ the field generated by $\alpha_1,\alpha_2,\cdots,\alpha_n$ over $\mbf$.

For example, the field $\mbq(\sqrt{2})$ is the subfield of $\mbr$ generated by $\sqrt{2}$ over $\mbq$. One can see that $\mbq(\sqrt{2})=\{a+b\sqrt{2}\,|\,a,b\in\mbq\}$ and $[\mbq(\sqrt{2}):\mbq]=2$.

\begin{theorem}\

\label{thm:root}
Suppose $p$ is an irreducible polynomial in $\pf$.  Then there exists an extension $\mbe$ of $\mbf$ such that $\mbe$ contains a root of $p$.
\end{theorem}
\begin{proof}
  Since $(p)$ is a maximal ideal,  $\pf/(p)$ is a field. Denote it by $\mbe$.

  Define $\varphi:\mbf\to\mbe$ by $\varphi(a)=a+(p)$ for all $a$ in $\mbf$. Then $\varphi$ is a nonzero homomorphism, hence injective. So $\mbf$ is isomorphic to $\text{Im}\varphi=\varphi(\mbf)=\{a+(p)\,|\,a\in\mbf\}$, which is, a subfield of $\mbe$.

  Now we can identify $p=a_nx^n+a_{n-1}x^{n-1}+\cdots+a_1x+a_0$ in $\pf$ with $\tilde{p}=(a_n+(p))x^n+(a_{n-1}+(p))x^{n-1}+\cdots+(a_1+(p))x+(a_0+(p))$ in $\varphi(\mbf)[x]$.

  Moreover
  \begin{align*}
  &\tilde{p}(x+(p))=(a_n+(p))(x+(p))^n+(a_{n-1}+(p))(x+(p))^{n-1}+\cdots\\
  &+(a_1+(p))(x+(p))+(a_0+(p))=a_nx^n+a_{n-1}x^{n-1}+\cdots+a_1x+a_0+(p)\\
  &=p+(p)=0.
  \end{align*}
So $\tilde{p}$ has a root, $x+(p)$ in $\mbe$.
\end{proof}

\begin{lemma}
\label{lm:algpoly}
Suppose $\mbf$ is a subfield of $\mbe$ and $\alpha$ in $\mbe$  is a root of an irreducible polynomial $p$ in $\pf$. Then
$\mbf(\alpha)=\{f(\alpha)\,|\,f\in \pf\}$.
\end{lemma}
\begin{proof}
Denote $\{f(\alpha)\,|\,f\in \pf\}$ by $\mbk$. Clearly $\mbk$ is a unital commutative subring of $\mbf(\alpha)$. Moreover if $f(\alpha)\neq 0$, then $f$ is prime to $p$ since $p$ is irreducible. There exist $g,h$ in $\pf$ such that $fg+ph=1$. Hence $f(\alpha)g(\alpha)=1$. This shows that $\mbk$ is a field. So $\mbf(\alpha)\subseteq \mbk$.

Therefore $\mbf(\alpha)=\mbk$.
\end{proof}

\begin{theorem}
\label{thm:algirr}
Suppose $\mbf$ is a subfield of $\mbe$ and $\alpha$ in $\mbe$  is a root of an irreducible polynomial $p$ in $\pf$. Then $\mbf(\alpha)\cong \pf/(p)$.
\end{theorem}
\begin{proof}
 By Lemma~\ref{lm:algpoly}, $\mbf(\alpha)=\mbk=\{f(\alpha)\,|\,f\in \pf\}$.

  Define $\varphi:\mbk\to \pf/(p)$ by $\varphi(f(\alpha))=f+(p)$ for all $f$ in $\pf$.
  \begin{itemize}
    \item If $f_1(\alpha)=f_2(\alpha)$, then $(f_1-f_2)(\alpha)=0$. Either $f_1-f_2$ is prime to $p$ or $p$ divides $f_1-f_2$. Since $p(\alpha)=0$, it's impossible that $f_1-f_2$ is prime to $p$. Thus $p$ divides $f_1-f_2$, which means, $f_1+(p)=f_2+(p)$. So $\varphi$ is well-defined.
    \item $\varphi(f_1(\alpha)f_2(\alpha))=\varphi(f_1f_2(\alpha))=f_1f_2+(p)=(f_1+(p))(f_2+(p))=\varphi(f_1(\alpha))\varphi(f_2(\alpha))$, and
    $\varphi(f_1(\alpha)+f_2(\alpha))=\varphi((f_1+f_2)(\alpha))=f_1+f_2+(p)=(f_1+(p))+(f_2+(p))=\varphi(f_1(\alpha))+\varphi(f_2(\alpha))$. Hence $\varphi$ is a ring homomorphism.
    \item If $f+(p)=0$, then $f$ is in $(p)$, i.e., $p$ divides $f$. So $f(\alpha)=0$. Hence $\varphi$ is injective.
    \item Obviously $\varphi$ is surjective.
  \end{itemize}
  Therefore $\mbf(\alpha)\cong \pf/(p)$.
\end{proof}

We say a polynomial $f$ in $\pf$ \textbf{split}s over a field $\mbe\supseteq\mbf$ if all roots  of $f$ are in $\mbe$.\index{split}

A polynomial $f$ in $\pf$ is called \textbf{monic} if the leading coefficient of the polynomial is 1, i.e.,  $f(x)=x^n+a_{n-1}x^{n-1}+\cdots+a_1x+a_0$.~\index{monic polynomial}
\begin{theorem}
\label{thm:split1}
  For every nonzero monic polynomial $f$ in $\pf$, there exists an extension $\mbe$ of $\mbf$ over which $f$ splits.
\end{theorem}
\begin{proof}
  We do induction on $\deg f$.

  If $\deg f=1$, then $f$ splits over $\mbf$.

  Suppose the statement holds for any polynomial of degree less than $n$. Hence every polynomial in $\pf$ of degree less than $n$ splits over an extension field of $\mbf$.

  Assume that $f$ is an irreducible monic polynomial with $\deg f=n$.

  By Theorem~\ref{thm:root}, there is an extension $\mbk$ of $\mbf$ which contains a root $\alpha$ of $f$, which means, $f(x)=(x-\alpha)g(x)$ with $\alpha$ in $\mbk$ and $g$ in $\mbk[x]$. By assumption, there is an extension $\mbe$ of $\mbk$ over which $g$ splits. Consequently $f$ splits over $\mbe$.

  This completes the proof.
\end{proof}
\subsection{Algebraic extensions}

Suppose $\mbf$ is a subfield of $\mbe$. An element $\alpha$ in $\mbe$ is called \textbf{algebraic} over $\mbf$ if it is a root of a nonzero polynomial in $\pf$. For example, $\sqrt{2}$ is algebraic over $\mbq$ since it is a root of $x^2-2$.
\begin{theorem}
\label{thm:minimalpolynomial}
If $\alpha$ in $\mbe$ is algebraic over $\mbf$, then there exists a unique monic irreducible polynomial $p$ in $\pf$ such that $p(\alpha)=0$.
\end{theorem}
\begin{proof}
We first prove the existence of $p$.

Since $\alpha$ is algebraic over $\mbf$, there is a monic polynomial $f$ in $\pf$ such that $f(\alpha)=0$. Write $f$ as a product of  irreducible monic polynomials: $f=p_1p_2\cdots p_m$. Then $p_i(\alpha)=0$ for some $1\leq i\leq m$.

Suppose $p$ and $q$ are distinct monic irreducible polynomials in $\pf$ with a common root $\alpha$. Since $p$ and $q$ are coprime, there exist $g,h$ in $\mbf$ such that $pg+qh=1$. Then $0=p(\alpha)g(\alpha)+q(\alpha)h(\alpha)=1$, which is a contradiction. This proves the uniqueness of $p$.

\end{proof}

Call $p$ the \textbf{minimal polynomial} of $\alpha$ over $\mbf$ and denote it by $m_{\alpha,\mbf}(x)$.~\index{minimal polynomial}

\begin{corollary}
Suppose $\alpha$ in $\mbe$ is algebraic over $\mbf$. If $f$ in $\pf$ satisfies that $f(\alpha)=0$, then $m_{\alpha,\mbf}(x)$ divides $f$. That is, among all nonzero polynomials in $\pf$ with $\alpha$ being a root,  $m_{\alpha,\mbf}(x)$ has the smallest degree.
\end{corollary}
\begin{proof}
  Since $m_{\alpha,\mbf}(x)$ is irreducible, either $m_{\alpha,\mbf}(x)$ divides $f$ or $m_{\alpha,\mbf}$ is prime to $f$. As the argument in the proof of Theorem~\ref{thm:minimalpolynomial} shows, it's impossible that $m_{\alpha,\mbf}$ is prime to $f$.
\end{proof}

The degree of $m_{\alpha,\mbf}(x)$ is called the \textbf{algebraic degree} of $\alpha$ over $\mbf$.~\index{algebraic degree}

\begin{corollary}
Suppose $\alpha$ in $\mbe$ is algebraic over $\mbf$. Then $\mbf(\alpha)\cong \pf/(m_{\alpha,\mbf})$.
\end{corollary}

\begin{proof}
It follows from Theorem~\ref{thm:algirr} and Theorem~\ref{thm:minimalpolynomial}.
\end{proof}

\begin{theorem}
  \label{thm:algebraicfinite}
  $\alpha$ is algebraic over $\mbf$ iff $[\mbf(\alpha):\mbf]<\infty$. Moreover $[\mbf(\alpha):\mbf]$ equals the algebraic degree of $\alpha$ when $\alpha$ is algebraic over $\mbf$.
\end{theorem}
\begin{proof}
Suppose $\alpha$ is algebraic over $\mbf$ with  the minimal polynomial $m_{\alpha,\mbf}$ and $\deg m_{\alpha,\mbf}=n$.

By Lemma~\ref{lm:algpoly}, $\mbf(\alpha)=\mbk=\{f(\alpha)\,|\,f\in \pf\}$. So $\mbf(\alpha)=\text{Span} (1,\alpha,\cdots,\alpha^{n-1})$. It follows that $[\mbf(\alpha):\mbf]\leq n<\infty$. Moreover $\{1,\alpha,\cdots,\alpha^{n-1}\}$ is linearly independent over $\mbf$ otherwise one can find a nonzero polynomial $f$ in $\pf$ with degree smaller than $n$ such that $f(\alpha)=0$, which contradicts that $\deg m_{\alpha,\mbf}=n$. Hence  $[\mbf(\alpha):\mbf]=n$.

Suppose that $[\mbf(\alpha):\mbf]=n$. Then $\{1,\alpha,\cdots,\alpha^{n-1},\alpha^n\}$ is linearly independent over $\mbf$. Hence there is a nonzero polynomial $f$ of degree no greater than $n$ in $\pf$ such that $f(\alpha)=0$, which means, $\alpha$ is algebraic over $\mbf$.

\end{proof}

\begin{theorem}
  \label{thm:algebraicfield}
  If $\alpha,\beta$ in $\mbe$ are algebraic over $\mbf$, then $\alpha\pm\beta$, $\alpha\beta$ and $\alpha^{-1}$ are also algebraic over $\mbf$. That is, algebraic elements  over $\mbf$ form a subfield of $\mbe$.
\end{theorem}
\begin{proof}
  Since $\mbf(\alpha+\beta)$ is a subfield of $\mbf(\alpha,\beta)$ and $[\mbf(\alpha,\beta):\mbf]=[\mbf(\alpha,\beta):\mbf(\alpha)][\mbf(\alpha):\mbf]<\infty$, by Theorem~\ref{thm:algebraicfinite}, $\alpha+\beta$ is algebraic over $\mbf$. Similarly $\alpha-\beta$ and $\alpha\beta$ are also algebraic over $\mbf$.

  Moreover $\mbf(\alpha^{-1})$ is a subfield of $\mbf(\alpha)$,  so $[\mbf(\alpha):\mbf]<\infty$ implies that $[\mbf(\alpha^{-1}):\mbf]<\infty$.  By Theorem~\ref{thm:algebraicfinite}, $\alpha^{-1}$ is algebraic over $\mbf$.
\end{proof}
Let $\mbf$ be a subfield of $\mbe$. We say that $\mbe$ is an \textbf{algebraic extension} of $\mbf$ if every element in $\mbe$ is algebraic over $\mbf$.~\index{algebraic extension}

Every finite extension of $\mbf$ is an algebraic extension over $\mbf$.

The following theorem is a characterization of finite extensions via algebraic extensions.

\begin{theorem}
\label{thm:alf}
  $\mbe$ is a finite extension of $\mbf$ iff in $\mbe$, there exist algebraic $\alpha_1,\alpha_2,\cdots,\alpha_n$  over $\mbf$ such that $\mbe=\mbf(\alpha_1,\alpha_2,\cdots,\alpha_n)$.
\end{theorem}
\begin{proof}
Suppose $\mbe$ is a finite extension of $\mbf$. Let $\{\alpha_i\}_{i=1}^n$ be  a basis of $\mbe$ over $\mbf$. Since $[\mbf(\alpha_i):\mbf]\leq[\mbe:\mbf]<\infty$ for all $1\leq i\leq n$, every $\alpha_i$ is algebraic over $\mbf$. Moreover $\mbe=\mbf(\alpha_1,\alpha_2,\cdots,\alpha_n)$.

Conversely if there exist algebraic $\{\alpha_i\}_{i=1}^n$  in $\mbe$ such that $\alpha_1,\alpha_2,\cdots,\alpha_n$ are algebraic over $\mbf$ and
$\mbe=\mbf(\alpha_1,\alpha_2,\cdots,\alpha_n)$. Hence for every $1\leq i\leq n$, $\alpha_{i+1}$ is algebraic over $\mbf(\alpha_1,\cdots,\alpha_i)$, that is, $[\mbf(\alpha_1,\alpha_2,\cdots,\alpha_i,\alpha_{i+1}): \mbf(\alpha_1,\alpha_2,\cdots,\alpha_i)]<\infty$. Therefore
\begin{align*}
 &[\mbe:\mbf]=[\mbf(\alpha_1,\alpha_2,\cdots,\alpha_n):\mbf]=[\mbf(\alpha_1,\alpha_2,\cdots,\alpha_{n-1},\alpha_n): \mbf(\alpha_1,\alpha_2,\cdots,\alpha_{n-1})] \\
&[\mbf(\alpha_1,\alpha_2,\cdots,\alpha_{n-1}): \mbf(\alpha_1,\alpha_2,\cdots,\alpha_{n-2})]\cdots[\mbf(\alpha_1,\alpha_2):[\mbf(\alpha_1)]<\infty.
\end{align*}
 So $\mbe$ is a finite extension of $\mbf$.
\end{proof}

The following theorem says that algebraic extensions are transitive.
\begin{theorem}
  \label{thm:alal}
  Suppose $\mbk$ is an algebraic extension of $\mbe$ and $\mbe$ is an algebraic extension of $\mbf$. Then $\mbk$ is an algebraic extension of $\mbf$.
\end{theorem}
\begin{proof}
Take any $a$ in $\mbk$.

Since $a$ is algebraic over $\mbe$, there exists a nonzero polynomial $f=a_0+a_1x+\cdots+a_nx^n$ in $\mbe[x]$ such that $f(a)=0$.

Hence $[\mbf(a,a_0,a_1,\cdots,a_n):\mbf(a_0,a_1,\cdots,a_n)]<\infty$. Note that $a_0,a_1,\cdots,a_n$ are all algebraic over $\mbf$, thus
$[\mbf(a_0,a_1,\cdots,a_n):\mbf]<\infty$. Consequently
\begin{align*}
&[\mbf(a):\mbf]\leq [\mbf(a,a_0,a_1,\cdots,a_n):\mbf] \\
&=[\mbf(a,a_0,a_1,\cdots,a_n):\mbf(a_0,a_1,\cdots,a_n)][\mbf(a_0,a_1,\cdots,a_n):\mbf]<\infty.
\end{align*}
This shows that $a$ is algebraic over $\mbf$. Since $a$ is arbitrarily chosen in $\mbk$, this means that $\mbk$ is an algebraic extension of $\mbf$.

\end{proof}

\section{Splitting fields}
\subsection{Existence and uniqueness of splitting fields}
\begin{definition}
  An extension field $\mbe$ of a field $\mbf$ is called a \textbf{splitting field} of $f$ in $\pf$ if
  \begin{enumerate}
    \item $f(x)$ factors completely over $\mbe$, i.e., $\mbe$ contains every root of $f$.
    \item $f(x)$ does not factor completely over any proper subfield of $\mbe$ containing $\mbf$.
     \end{enumerate}
     \index{splitting field}
\end{definition}

\begin{lemma}
\label{lem:unispf}
  Suppose $\varphi:\mbf_1\to\mbf_2$ is a field isomorphism and $f_1(x)=a_nx^n+a_{n-1}x^{n-1}+\cdots+a_1x+a_0$ is in $\mbf_1[x]$. Define $f_2$ in $\mbf_2[x]$ as $f_2(x)=\varphi(a_n)x^n+\varphi(a_{n-1})x^{n-1}+\cdots+\varphi(a_1)x+\varphi(a_0)$. If $\mbe_1\supseteq \mbf_1$ is a splitting field of $f_1$ and $\mbe_2\supseteq\mbf_2$ is a splitting field of $f_2$, then  there is an isomorphism $\sigma: \mbe_1\to \mbe_2$ such that $\varphi(a)=\sigma(a)$ for all $a$ in $\mbf_1$.
\end{lemma}

\begin{proof}
  we do induction on  $\deg f_1$.

If $\deg f_1=1$, then $\mbf_1$ is a splitting field of $f_1$ and $\mbf_2$ is a splitting field of $f_2$. The required isomorphism  $\sigma: \mbf_1\to \mbf_2$ is $\varphi$.

Assume that the statement holds when $f_1$ in $\mbf_1[x]$ has degree less than $n$.

Suppose $\deg f_1=n$, and $\mbe_1\supseteq \mbf_1$ and $\mbe_2\supseteq\mbf_2$ are splitting fields of $f_1$ and $f_2$ respectively.

Note that $f_1$ splits over $\mbe_1$,  so $\mbe_1$ contains  a root  $\alpha$ of $f_1$. Since $\alpha$ is algebraic over $\mbf_1$, there is an irreducible polynomial $p_1(x)=a_kx^k+\cdots+a_1x+a_0$ in $\mbf_1[x]$ which is a factor of $f_1$ and $p_1(\alpha)=0$. So $p_2(x)=\varphi(a_k)x^k+\cdots+\varphi(a_1)x+\varphi(a_0)$ is an irreducible polynomial in $\mbf_2[x]$ which is a factor of $f_2$ and splits over $\mbe_2$. Also $\beta=\varphi(\alpha)$ is a root of $p_2$.  By Theorem~\ref{thm:algirr}, we have $\mbf_1 (\alpha)\supseteq\mbf_1$, $\mbf_2(\beta)\supseteq\mbf_2$, and $$\mbf_1 (\alpha)\cong \mbf_1[x]/(p_1)\cong\mbf_2[x]/(p_2)\cong \mbf_2(\beta).$$ Moreover there is an isomorphism $\psi:\mbf_1 (\alpha)\to \mbf_2(\beta)$ such that $\psi(a)=\varphi(a)$ for all $a$ in $\mbf_1$.

Moreover $f_1=(x-\alpha)g_1$ and $f_2=(x-\beta)g_2$ with $g_1$ in $\mbf_1 (\alpha)[x]$ and $g_2$ in $\mbf_2(\beta)[x]$. Note that $g_1$ can be identified with $g_2$ via $\varphi$, and $\mbe_1\supseteq\mbf_1 (\alpha)$ is a splitting field of $g_1$ and $\mbe_2\supseteq\mbf_2(\beta)$ is a splitting field of $g_2$. Since $\deg g_1=\deg g_2=n-1$, by assumption there is an isomorphism $\sigma: \mbe_1\to \mbe_2$ such that $\sigma(b)=\psi(b)$ for all $b$ in $\mbf_1 (\alpha)$. Consequently $\sigma(a)=\psi(a)=\varphi(a)$ for all $a$ in $\mbf_1$.
\end{proof}

Next we prove the existence and uniqueness of splitting fields.
\begin{theorem}
  \label{thm:split}
 Up to isomorphism, there is a unique splitting field for every $f$ in $\pf$ with $\deg f\geq 1$.
\end{theorem}
\begin{proof}
By Theorem~\ref{thm:split1}, there is an extension $\mbk$ of $\mbf$ over which $f$ splits. Hence $f(x)=(x-\alpha_1)(x-\alpha_2)\cdots(x-\alpha_n)$ for $\alpha_1,\alpha_2,\cdots,\alpha_n$ in $\mbk$.~\footnote{We may assume that $f$ is monic.} Then $\mbf(\alpha_1,\cdots,\alpha_n)$ is a splitting field of $f$. This proves the existence of splitting fields.

The uniqueness of splitting fields follows from Lemma~\ref{lem:unispf}.
\end{proof}
\subsection{Splitting fields of some polynomials}

We calculate splitting fields of some polynomials, which is a preparation for calculations of Galois groups.

(1) The field $\mbq(\sqrt{2})$ is the splitting field of  $f(x)=x^2-2$ in $\pq$ and $[\mbq(\sqrt{2}):\mbq]=2$.  ~\footnote{$\mbr$ is not a splitting field although $f$ splits over $\mbr$.}

The fundamental theorem of algebra says that every polynomial in $\pc$ splits over $\mbc$. Hence the splitting field $\mbe$ of $x^2-2$ is a subfield of $\mbc$. Since $\sqrt{2}$ is a root of $x^2-2$, we have $\mbq(\sqrt{2})\subseteq\mbe$. Since $x^2-2$ splits over $\mbq(\sqrt{2})$, we have $\mbe=\mbq(\sqrt{2})$.

Furthermore $\mbq(\sqrt{2})=\{a+b\sqrt{2}\,|\,a,b\in\mbq\}$ and $\{1,\sqrt{2}\}$ is linearly independent over $\mbq$. Hence  $\{1,\sqrt{2}\}$ is a basis of $\mbq(\sqrt{2})$ over $\mbq$ and $[\mbq(\sqrt{2}):\mbq]=2$.

(2) The field $\mbq(\sqrt{2},\sqrt{3})$ is the splitting field of  $f(x)=(x^2-2)(x^2-3)$ in $\pq$ and $[\mbq(\sqrt{2},\sqrt{3}):\mbq]=4$.

A similar argument as before shows that $\mbq(\sqrt{2},\sqrt{3})$ is the splitting field of  $f(x)=(x^2-2)(x^2-3)$.

Next we prove that  $[\mbq(\sqrt{2},\sqrt{3}):\mbq(\sqrt{2})]=2$ which will imply that $[\mbq(\sqrt{2},\sqrt{3}):\mbq]=[\mbq(\sqrt{2},\sqrt{3}):\mbq(\sqrt{2})][\mbq(\sqrt{2}):\mbq]=4$.


It's enough to prove that $x^2-3$ is irreducible in $\mbq(\sqrt{2})[x]$, i.e., $a+b\sqrt{2}$ for all $a,b$ in $\mbq$ cannot be a root of $x^2-3$.

Suppose $(a+b\sqrt{2})^2=3$ for some $a,b$ in $\mbq$.

Then $a^2+2b^2-3=-2ab\sqrt{2}$. Since the left hand side is rational, we have $ab=0$.

When $a=0$, we get that $2b^2-3=0$ which is impossible since $b$ is rational.

When $b=0$, we get that $a^2-3=0$ which is also impossible since $a$ is rational.

Above all this leads to a contradiction.

We have the following diagram for the intermediate fields between $\mbq$ and $\mbq(\sqrt{2},\sqrt{3})$.

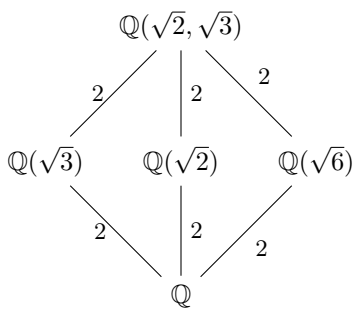
\begin{center}
\begin{tikzpicture}
  \node (K) {$\mbq(\sqrt{2},\sqrt{3})$};
   \node (K2) [below of=K] {$\mbq(\sqrt{2})$};
    \node (K3) [left of=K2] {$\mbq(\sqrt{3})$};
  \node (K6) [right of=K2] {$\mbq(\sqrt{6})$};
  \node (Q) [below of=K2] {$\mbq$};
  \draw[-] (K) to node {\small{2}} (K2);
  \draw[-] (K) to node [left]{\small{2}} (K3);
  \draw[-] (K) to node  {\small{2}} (K6);
  \draw[-] (K2) to node {\small{2}} (Q);
   \draw[-] (K3) to node[left] {\small{2}} (Q);
    \draw[-] (K6) to node {\small{2}} (Q);
\end{tikzpicture}
\end{center}
 Here the number  between two fields stands for the field's  dimension over the field below. For instance $[\mbq(\sqrt{2},\sqrt{3}):\mbq(\sqrt{2})]=2$.

(3) The splitting field of $x^3-2$ is $\mbq(\sqrt[3]{2},\sqrt{3}i)$ and $[\mbq(\sqrt[3]{2},\sqrt{3}i):\mbq]=6$.

 Let $\mbe\subseteq\mbc$ be the splitting field of $x^3-2$. Note that $x^3-2$ has 3 roots in $\mbc$:  $\theta=\sqrt[3]{2}(-\dfrac{1}{2}+\dfrac{\sqrt{3}}{2}i)$,  $\theta^2=\sqrt[3]{2}(-\dfrac{1}{2}-\dfrac{\sqrt{3}}{2}i)$ and $\theta^3=\sqrt[3]{2}$. Hence $\sqrt[3]{2}$ and $\sqrt{3}i=\dfrac{\theta-\theta^2}{\theta^3}$ are in $\mbe$ and $\mbq(\sqrt[3]{2},\sqrt{3}i)\subseteq\mbe$. Moreover $x^3-2$ splits over $\mbq(\sqrt[3]{2},\sqrt{3}i)$. So $\mbq(\sqrt[3]{2},\sqrt{3}i)=\mbe$.

 Note that $[\mbq(\sqrt[3]{2}):\mbq]=3$ since $x^3-2$ is an irreducible polynomial in $\pq$, and $x^2+3$ has no roots in $\mbq(\sqrt[3]{2})$ since $\mbq(\sqrt[3]{2})$ is contained in $\mbr$. So  $x^2+3$ is an irreducible polynomial in $\mbq(\sqrt[3]{2})[x]$ which implies that $[\mbq(\sqrt[3]{2},\sqrt{3}i):\mbq(\sqrt[3]{2})]=2$.

 Finally $[\mbq(\sqrt[3]{2},\sqrt{3}i):\mbq]=[\mbq(\sqrt[3]{2},\sqrt{3}i):\mbq(\sqrt[3]{2})][\mbq(\sqrt[3]{2}):\mbq]=6$.

 The following is a diagram of intermediate fields between $\mbq$ and $\mbq(\sqrt[3]{2},\sqrt{3}i)$.

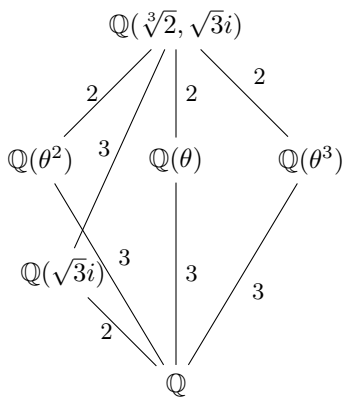
 \begin{center}
\begin{tikzpicture}
  \node (K) {$\mbq(\sqrt[3]{2},\sqrt{3}i)$};
   \node (K1) [below of=K] {$\mbq(\theta)$};
    \node (K2) [left of=K1] {$\mbq(\theta^2)$};
  \node (K3) [right of=K1] {$\mbq(\theta^3)$};
  \node (K4) [node distance=1.5cm, left of=K2]  [node distance=1.5cm,  below of=K1]{$\mbq(\sqrt{3}i)$};
  \node (Q) [node distance=3cm, below of=K1] {$\mbq$};
  \draw[-] (K) to node {\small{2}} (K1);
  \draw[-] (K) to node [left]{\small{2}} (K2);
  \draw[-] (K) to node  {\small{2}} (K3);
  \draw[-] (K) to node [left]{\small{3}} (K4);
   \draw[-] (K1) to node {\small{3}} (Q);
   \draw[-] (K2) to node {\small{3}} (Q);
   \draw[-] (K3) to node {\small{3}} (Q);
    \draw[-] (K4) to node [left] {\small{2}} (Q);
  \end{tikzpicture}
\end{center}
(4) The splitting field of $x^4+4$ is $\mbq(i)$.

 It follows from that the roots of $x^4+4$ in $\mbc$ are $\pm 1\pm i$.

(5) Let $p$ be prime. The splitting field of $x^p-1$ in $\pq$ is $\mbq(\zeta_p)$ where $\zeta_p=e^{\frac{2\pi i}{p}}$ and $[\mbq(\zeta_p):\mbq]=p-1$ since every root of $x^p-1$ is $\zeta_p^i$ for some $1\leq i\leq p$ and $\zeta_p$ is a root of $\Phi_p(x)=\dfrac{x^p-1}{x-1}$ which is an irreducible polynomial in $\pq$ of degree $p-1$.

(6) With the same setting as the previous example, the splitting field of $x^p-2$ in $\pq$ is $\mbq(\zeta_p,\sqrt[p]{2})$ and
$[\mbq(\zeta_p,\sqrt[p]{2}):\mbq]=p(p-1)$.

It's routine to prove that $\mbq(\zeta_p,\sqrt[p]{2})$ is the splitting field of $x^p-2$.

Since $\Phi_p$  is in $\pq\subseteq \mbq(\sqrt[p]{2})[x]$ and $\Phi_p(\zeta_p)=0$, we have
$$[\mbq(\zeta_p,\sqrt[p]{2}):\mbq(\sqrt[p]{2})]=[\mbq(\sqrt[p]{2})(\zeta_p):\mbq(\sqrt[p]{2})]\leq\deg \Phi_p=p-1.$$
Also $[\mbq(\sqrt[p]{2}):\mbq]=\deg (x^p-2)=p$ since $\sqrt[p]{2}$ is a root of $x^p-2$ which is an irreducible polynomial in $\pq$.

So $[\mbq(\zeta_p,\sqrt[p]{2}):\mbq]=[\mbq(\zeta_p,\sqrt[p]{2}):\mbq(\sqrt[p]{2})][\mbq(\sqrt[p]{2}):\mbq]\leq p(p-1)$.

On the other hand, both $p=[\mbq(\sqrt[p]{2}):\mbq]$ and $p-1=[\mbq(\zeta_p):\mbq]$ divide $[\mbq(\zeta_p,\sqrt[p]{2}):\mbq]$. Since $p$ is prime to $p-1$, we get that $p(p-1)$ is a factor of $[\mbq(\zeta_p,\sqrt[p]{2}):\mbq]$. Finally $[\mbq(\zeta_p,\sqrt[p]{2}):\mbq]=p(p-1)$.

(7) The splitting field of $x^n-1$ in $\pq$ is $\mbq(\zeta_n)$ and $[\mbq(\zeta_n):\mbq]=\phi(n)$ where $\zeta_n=e^{\frac{2\pi i}{n}}$ and $\phi(n)$ is the number of positive integers less than and prime to $n$.~\footnote{The function $\phi$ defined on the set of positive integers is called \textbf{Euler's  phi function}.~\index{Euler's phi function}}

The set of roots of $x^n-1$ is a multiplicative cyclic group $\langle \zeta_n \rangle$ isomorphic to $\mbz_n$ whose generators are $\zeta_n^m$ for all $1\leq m\leq n$ prime to $n$. Hence $\langle \zeta_n \rangle$ has $\phi(n)$ generators.

Define $\Phi_n(x)=\displaystyle\prod_{\substack{1\leq m\leq n \\ \gcd(m,n)=1}}(x-\zeta_n^m)$ which is called \textbf{the $n$-th cyclotomic polynomial}.~\index{cyclotomic polynomial}

We list some of $\Phi_n$'s:
    \begin{itemize}
    \item $\Phi_p(x)=\dfrac{x^p-1}{x-1}=x^{p-1}+x^{p-2}+\cdots+x+1$ when $p$ is prime.
    \item $\Phi_1(x)=x-1$.
    \item $\Phi_4(x)=(x-i)(x+i)=x^2+1$.
    \item $\Phi_6(x)=(x-\zeta_6)(x-\zeta_6^5)=x^2-x+1$.
    \item $\Phi_8(x)=(x-\zeta_8)(x-\zeta_8^3)(x-\zeta_8^5)(x-\zeta_8^7)=x^4+1$.
    \item $\Phi_9(x)=(x-\zeta_9)(x-\zeta_9^2)(x-\zeta_9^4)(x-\zeta_9^5)(x-\zeta_9^7)(x-\zeta_9^8)=x^6+x^3+1$.
    \item $\Phi_{10}(x)=(x-\zeta_{10})(x-\zeta_{10}^3)(x-\zeta_{10}^7)(x-\zeta_{10}^9)=x^4-x^3+x^2-x+1$.
    \item $\Phi_{12}(x)=(x-\zeta_{12})(x-\zeta_{12}^5)(x-\zeta_{12}^7)(x-\zeta_{12}^{11})=x^4-x^2+1$.
    \end{itemize}

    The following theorem will imply that  $\mbq(\zeta_n)\cong\pq/(\Phi_n)$, hence $[\mbq(\zeta_n):\mbq]=\deg (\Phi_n)=\phi(n)$.

\begin{theorem}
\label{thm:Phi}
The following hold for $\Phi_n(x)$:
\begin{enumerate}
  \item[(i)] $\Phi_n(x)$ is in $\pz$ and has constant term $\pm 1$.
  \item[(ii)]$\Phi_n(x)$ is  irreducible in $\pq$.
\end{enumerate}
\end{theorem}

Before proceeding to the proof of the theorem, we present some background of separable polynomials.

\begin{definition}
  A polynomial $f$ in $\pf$ is called \textbf{separable} if it has no multiple roots, otherwise \textbf{inseparable}.~\index{separable polynomial}\index{inseparable polynomial}

\end{definition}

\begin{example}\

  \begin{enumerate}
    \item Any polynomial of degree 1 is separable.
    \item $x^2-3$ in $\pq$ is separable.
      \end{enumerate}
\end{example}
\begin{definition}
  The \textbf{derivative} of a polynomial $f(x)=a_nx^n+a_{n-1}x^{n-1}+\cdots+a_1x+a_0$ in $\pf$ is defined to be the polynomial
  $D(f)=na_nx^{n-1}+(n-1)a_{n-1}x^{n-2}+\cdots+2a_2x+a_1$ in $\pf$. \index{derivative}
\end{definition}
The following properties of derivatives hold.
\begin{proposition}\
For $f,g$ in $\pf$,
  \begin{enumerate}
    \item $D(f+g)=D(f)+D(g)$;
    \item $D(fg)=D(f)g+fD(g)$.
  \end{enumerate}
\end{proposition}

Below is a characterization of separable polynomials via derivatives.

\begin{theorem}
\label{thm:derivative}
  A polynomial $f$ in $\pf$ is separable iff $f$ is prime to $D(f)$.
\end{theorem}
\begin{proof}
  Without loss of generality, we may assume that $f$ splits over $\mbf$.

One can see that $f$ is not prime to $D(f)$ iff $f$ and $D(f)$ have a common root. Then we show that any common root of $f$ and $D(f)$ is a multiple root of $f$, which completes the proof.

Suppose $\alpha$ is a common root of $f$ and $D(f)$. Then $f(x)=(x-\alpha)g(x)$ for $g$ in $\pf$. Then $D(f)=g+(x-\alpha)D(g)$. It follows that $g(\alpha)=0$. Hence $g(x)=(x-\alpha)h$ for $h$ in $\pf$. Therefore $f(x)=(x-\alpha)^2h(x)$ which implies that $\alpha$ is a multiple root of $f$.

Conversely assume that $\alpha$ is a multiple root of $f$. Then $f(x)=(x-\alpha)^2h(x)$ for  $h$ in $\pf$. Then $D(f)=2(x-\alpha)h(x)+(x-\alpha)^2D(h)$ which implies that $\alpha$ is also a root of $D(f)$, hence a common root of $f$ and $D(f)$.

\end{proof}

Theorem~\ref{thm:derivative} has some immediate applications.

\begin{enumerate}
  \item $x^{p^n}-x$ in $\mbf_p[x]$ is separable since $D(x^{p^n}-x)=-1$.
  \item $x^n-1$ is separable over any field $\mbf$ such that $\txch \mbf$ does not divides $n$ since $D(x^n-1)=nx^{n-1}$ and $1=\frac{1}{n}x(nx^{n-1})-(x^n-1)$, which means, $x^n-1$ is prime to $D(x^n-1)$.
  \item If $\txch \mbf$ divides $n$, then $x^n-1$ is not separable since $D(x^n-1)=nx^{n-1}=0$.
\end{enumerate}

Now we are ready to prove Theorem~\ref{thm:Phi}.

 \begin{proof}[Proof of Theorem~\ref{thm:Phi}]\

(i) We do induction on $n$.

$\Phi_1(x)=x-1$ is in $\pz$ and has constant term $-1$.

Assume that for every $m<n$,  $\Phi_m$ is in $\pz$ and has constant term $\pm 1$.

In $\langle \zeta_n \rangle$, every element  has order $d$ for some $d$ dividing $n$ and an element of order $d$ is of the form $\zeta_n^{m\frac{n}{d}}=\zeta_d^m$ for some $1\leq m\leq n$ which is prime to $d$.

Hence $$\displaystyle x^n-1=\prod_{d|n}\Phi_d(x)=\Phi_n(x)\prod_{\substack{1\leq d<n \\ d|n}}\Phi_d(x).$$

It follows from the assumption that $\Phi_n(x)$ is in $\pq$.

Define $\tf_n(x)=\displaystyle\prod_{\substack{1\leq d<n \\ d|n}}\Phi_d(x)$.

Then $x^n-1=\tf_n\Phi_n$. By assumption every $\Phi_d(x)$ is in $\pz$ and has constant term $\pm 1$ for all $1\leq d<n$ and $d|n$, so is $\tf_n$.

Hence the constant term of $\Phi_n(x)$ is $\pm 1$. Inductively every coefficient of $\Phi_n(x)$ is an integer.

(ii)According to Gauss's lemma, it suffices to prove that $\Phi_n(x)$ is irreducible in $\pz$.

Suppose $f$ is an irreducible monic polynomial in $\pz$ which divides $\Phi_n(x)$. We are going to show that every $\zeta_n^m$ is a root of $f$ which implies that $f=\Phi_n$ and $\Phi_n$ is irreducible.

For any root $\zeta$ of $f$, we show that $f(\zeta^p)=0$ for any prime $p\nmid n$, which implies that $f(\zeta_n^m)=0$ for any  $1\leq m\leq n$ prime to n.

Suppose there exists $p\nmid n$ such that $f(\zeta^p)\neq 0$ for some root $\zeta$ of $f$. Since $\zeta^p$ is a root of $\Phi_n$,  there is $g$ in $\pz$ such that $\Phi_n=fg$ and $g(\zeta^p)=0$. That is to say that $\zeta$ is a root of the polynomial $g(x^p)$. Then $g(x^p)=f(x)h(x)$ for $h$ in $\pz$ since $f$ is irreducible and $\zeta$ is a root of $f$.

Consider the ring homomorphism from $\pz$ to $\mbf_p[x]$ mapping $q=a_0+a_1x+\cdots+a_kx^k$ in $\pz$ to $\bar{q}=\bar{a}_0+\bar{a}_1x+\cdots+\bar{a}_kx^k$ in $\mbf_p[x]$. For $g(x)=b_0+b_1x+\cdots+b_lx^l$, one has
$$\bar{f}(x)\bar{h}(x)=\bar{g}(x^p)=\bar{b}_0+\bar{b}_1x^p+\cdots+\bar{b}_lx^{lp}=\bar{b}_0^p+\bar{b}_1^px^p+\cdots+\bar{b}_l^px^{lp}=(\bar{g}(x))^p.$$

Here $\bar{b}_i^p=\bar{b}_i$ for $0\leq i\leq l$ follows from Fermat's little theorem.

  Hence $\bar{f}(x)$ and $\bar{g}(x)$ have a common root. Thus $\overline{\Phi_n}=\bar{f}\bar{g}$ is inseparable in $\mbf_p[x]$. However $\overline{\Phi_n}$ is  separable in $\mbf_p[x]$ since $x^n-1$ is separable in $\mbf_p[x]$. This is a contradiction.

 \end{proof}

 The field $\mbq(\zeta_n)$ is called a \textbf{cyclotomic extension} of $\mbq$.~\index{cyclotomic extension}

\section{Separable extensions}

In this section, we discuss separable extensions which paves the way for later discussions of Galois extensions.

\begin{proposition}
  \label{prop:sep}
  Assume that $\txch \mbf=0$. Then every irreducible polynomial in $\pf$ is separable. Moreover a monic polynomial in $\pf$ is separable iff it is a product of distinct monic irreducible polynomials.
\end{proposition}

\begin{proof}
  If $\txch \mbf=0$, then $D(p)\neq 0$ and $\deg (D(p))<\deg p$ for every irreducible polynomial in $\pf$. Hence $p$ is prime to $D(p)$.

Every $f$ in $\pf$ is a product of irreducibles and a root of $f$ is multiple iff this root is a common root of at least two of these irreducibles. However distinct irreducibles have no common roots. This shows that $f$ is separable iff $f$ is a product of distinct irreducibles.
\end{proof}

\begin{lemma}
\label{lm:p}
  Suppose $\mbf$ is a field with $\txch \mbf=p$. Then $\varphi:\mbf\to\mbf$ given by $\varphi(a)=a^p$ for all $a$ in $\mbf$ is an injective ring homomorphism. Moreover if $\mbf$ is a finite field, then $\varphi$ is an isomorphism.
\end{lemma}
\begin{proof}
 For all $a,b$ in $\mbf$, we have $\varphi(a+b)=(a+b)^p=\sum_{i=0}^{p}\binom{p}{i}a^ib^{p-i}=a^p+b^p$ since $p$ divides $\binom{p}{i}$'s for all $1\leq i\leq p-1$, and $\varphi(ab)=(ab)^p=a^pb^p=\varphi(a)\varphi(b)$. So $\varphi$ is a ring homomorphism. Moreover $\varphi(1)=1$, that is, $\varphi$ is a nonzero ring homomorphism between fields, hence $\varphi$ is injective.

When  $\mbf$ is finite,  the injection $\varphi$ is also a surjection, hence an isomorphism.
\end{proof}

\begin{definition}
A field $\mbk$ is called \textbf{perfect} if $\txch\mbk=0$ or $\mbk=\mbk^p$ when $\txch\mbk=p$, otherwise call it \textbf{imperfect}.~\index{perfect field}\index{imperfect field}
\end{definition}

Except for fields with characteristic 0, by Lemma~\ref{lm:p}, finite fields are also perfect fields.

 The fraction field of $\mbf_2[t]$, $\rm {Frac}(\mbf_2[t])=\{\dfrac{p}{q}\,|\,p,q\in\mbf_2[t], \,q\neq0\}$ is  imperfect since  no $p,q$ in $\mbf_2[t]$ satisfy that $(\frac{p}{q})^2=t$.

\begin{theorem}
Every irreducible polynomial over a perfect field is separable. Moreover a polynomial over a perfect field is separable iff it is a product of distinct irreducible polynomials.
\end{theorem}
\begin{proof}
  For fields of characteristic 0, we already prove the statements in Proposition~\ref{prop:sep}.

  Assume that $\txch\mbk=p$ and $\mbk=\mbk^p$, which means, the ring homomorphism $\varphi: \mbf\to\mbf$ given by $\varphi(a)=a^p$ for all $a$ in $\mbf$ is surjective.

  Suppose $f$ is irreducible over $\mbf$. If $f$ is not separable, then $f$ is not prime to $D(f)$, which means, $f$ divides $D(f)$. This only happens when $D(f)=0$. Hence $f=a_mx^{mp}+a_{m-1}x^{(m-1)p}+\cdots+a_1x^p+a_0$. Since $\varphi$ is surjective, we have $a_i=b_i^p$ for all $0\leq i\leq m$. Therefore
  \begin{align*}
 f&=a_mx^{mp}+a_{m-1}x^{(m-1)p}+\cdots+a_1x^p+a_0 \\
 &=b_m^px^{mp}+b_{m-1}^px^{(m-1)p}+\cdots+b_1^px^p+b_0^p\\
 &=(b_mx^{m}+b_{m-1}x^{m-1}+\cdots+b_1x+b_0)^p,
  \end{align*}
  which contradicts that $f$ is irreducible. Thus every irreducible polynomial over $\mbf$ is separable.

  The second statement follows from the first statement and the argument is similar to that in Proposition~\ref{prop:sep}.
\end{proof}

\begin{definition}
A field extension $\mbk/\mbf$ is called a \textbf{separable extension} if every element of $\mbk$ is a root of a separable polynomial in $\pf$.~\index{separable extension}
\end{definition}

\begin{corollary}
  Every finite extension of a perfect field is separable.
\end{corollary}

Over an imperfect field, there exist irreducible inseparable polynomials.

Let $\mbf=\rm {Frac}(\mbf_2[t])$. The polynomial $x^2-t$ in $\pf$ is irreducible since $x^2-t$ has no root in $\mbf$. Moreover $D(x^2-t)=2x=0$, so $\gcd(x^2-t, D(x^2-t))=x^2-t$, which implies that $x^2-t$ is inseparable.

\section*{Exercises}
\begin{exercise}
  Prove that $\{v_1,v_2,\cdots, v_n\}$ is a basis of a vector space $V$ iff  $\{v_1,v_2,\cdots, v_n\}$ is a maximal linearly independent subset of $V$.
\end{exercise}

\begin{exercise}
  Prove that $\dim_\mbq\mbr=\infty$.
\end{exercise}

\begin{exercise}
Prove Theorem~\ref{thm: dim}.
\end{exercise}

\begin{exercise}
Prove that when the characteristic of a field is nonzero, it is prime.
\end{exercise}

\begin{exercise}
Prove that $[\mbf_{p^n}: \mbf_p]=n$.
\end{exercise}

\begin{exercise}
Suppose $\alpha$ in $\mbe$ is algebraic over $\mbf$ and $p$ is a nonzero monic polynomial of the smallest degree in $\pf$ such that $p(\alpha)=0$.  Prove that $p=m_{\alpha,\mbf}$.
\end{exercise}

\begin{exercise}
  Find the minimal polynomial of $\sqrt[5]{7-\sqrt[3]{2}}$ in $\pq$.
\end{exercise}

\begin{exercise}
  Prove that a finite extension is an algebraic extension, but the converse is not true.
\end{exercise}

\begin{exercise}
  Find the splitting fields of the following polynomials in $\pq$:
  \begin{enumerate}
    \item $(x^2+2x-2)(x^2-3)$;
    \item $x^4-4$;
    \item $x^3+5$.
  \end{enumerate}
 \end{exercise}

\begin{exercise}
Let $p$ be prime.  Prove that $\Phi_p(x)=\dfrac{x^p-1}{x-1}$ is irreducible in $\pq$.
\end{exercise}

\begin{exercise}
  Find the splitting field of $x^2+x+1$ in $\mbf_2[x]$.
\end{exercise}

\begin{exercise}
  Prove that every polynomial of the form $x^n+a$ in $\pq$ is separable.
\end{exercise}

\begin{exercise}
  Let $\mbf=\mbf_2(t)$. Show that $x^2-t$ in $\pf$ has no root in $\mbf$.
\end{exercise}

\begin{exercise}
  Prove that $\mbq(\sqrt{2}+\sqrt{3})=\mbq(\sqrt{2},\sqrt{3})$. Find the minimal polynomial of $\sqrt{2}+\sqrt{3}$ over $\mbq$.
\end{exercise}

\begin{exercise}
Determine $[\mbq(\sqrt[3]{2+\sqrt{2}}):\mbq]$.
\end{exercise}

\begin{exercise}
  Prove that $\mbq(\sqrt[3]{2})$ is not a subfield of $\mbq(\zeta_n)$ for any positive integer $n$.
\end{exercise}

\begin{exercise}
Prove that if $[\mbf(\alpha):\mbf]$ is odd then $\mbf(\alpha)=\mbf(\alpha^2)$. Then prove the following generalization: if $n\nmid [\mbf(\alpha):\mbf]$, then $[\mbf(\alpha):\mbf(\alpha^n)]<n$.
\end{exercise}

\begin{exercise}
 Let $\overline{\mbq}$ be the field of algebraic numbers over $\mbq$. Prove that every polynomial in $\overline{\mbq}[x]$ splits over $\overline{\mbq}$. That is, $\overline{\mbq}$ is algebraically closed.
\end{exercise}

\begin{exercise}
Prove that if a subfield $\mbf$ of $\mbc$ contains $\zeta_n$ for $n$ odd, then $\mbf$ contains $\zeta_{2n}$.
\end{exercise}

\begin{exercise}
Prove that $\mbk$ is a splitting field of $f$ in $\pf$ iff every irreducible polynomial in $\pf$ who has a root in $\mbk$ splits over $\mbk$.
\end{exercise}

\begin{exercise}
Prove that if $\mbk_1$ and $\mbk_2$ are splitting fields over $\mbf$, then $\mbk_1\mbk_2$ and $\mbk_1\cap\mbk_2$ are also splitting fields over $\mbf$.
\end{exercise}

\begin{exercise}
  Prove that $x^p-x+a$ for $a$ nonzero is an irreducible separable polynomial in $\mbf_p[x]$.
\end{exercise}
\chapter{Galois Theory and Its Applications}

\section{Characterizations of Galois extension}
\begin{definition}
  Let $\mbk$ be a field.
  \begin{enumerate}
    \item An isomorphism $\sigma:\mbk\to\mbk$ is called an \textbf{automorphism} of $\mbk$. Denote by $\txau(\mbk)$ the set of automorphisms of $\mbk$.~\index{field automorphism}
    \item An automorphism $\sigma$ is said to fix an element $a$ in $\mbk$ if $\sigma(a)=a$.
    \item Let $\mbk/\mbf$ be a field extension.
    $\txau(\mbk/\mbf)=\{\sigma\,|\,\sigma(a)=a\,{\rm for\,all\,a\,in}\,\mbf\}$.
  \end{enumerate}

\end{definition}

Clearly $\txau(\mbk)$ is a group under composition of maps and $\txau(\mbk/\mbf)$ is a subgroup of $\txau(\mbk)$. For  a subgroup $G$ of $\txau(\mbk/\mbf)$, define $\mbk^G=\{a\in\mbk\,|\,\sigma(a)=a\,{\rm for\,all\,\sigma\,in}\,G\}$. It's clear that $\mbk^G$ is an intermediate field between $\mbf$ and $\mbk$.

The following says that elements of $\txau(\mbk/\mbf)$ are permutations of roots of polynomials in $\pf$.

\begin{proposition}
  Let  $\mbk/\mbf$ be a field extension and $\alpha$ in $\mbk$ be algebraic over $\mbf$. Then for any $\sigma$ in $\txau(\mbk/\mbf)$, $\sigma(\alpha)$ is also algebraic over $\mbf$ with the same minimal polynomial as $\alpha$.
\end{proposition}
\begin{proof}
  Suppose that $p(x)=x^n+a_{n-1}x^{n-1}+\cdots+a_1x+a_0$ in $\pf$ is the minimal polynomial of $\alpha$. Then
  \begin{align*}
  &p(\sigma(\alpha))=\sigma(\alpha)^n+a_{n-1}\sigma(\alpha)^{n-1}+\cdots+a_1\sigma(\alpha)+a_0\\
  &=\sigma(\alpha^n+a_{n-1}\alpha^{n-1}+\cdots+a_1\alpha+a_0)=\sigma(p(\alpha))=0.
  \end{align*}
  So $\sigma(\alpha)$ is algebraic over $\mbf$ and its minimal polynomial is also $p$.
\end{proof}
For Galois theory, one mainly cares about a special type of field extensions called Galois extension, which has many characterizations. Below we give some.
\begin{theorem}~\label{thm:GaloisCharacterizations}
  Suppose $\mbk$ is a finite extension of $\mbf$.

  The following are equivalent:
  \begin{enumerate}
    \item $\mbk$ is a splitting field of a separable polynomial $f(x)$ in $\pf$;
    \item $|\txau(\mbk/\mbf)|=[\mbk:\mbf]$;
    \item $\mbk^{\txau(\mbk/\mbf)}=\mbf$;
    \item $\mbk$ is a normal separable extension of $\mbf$, i.e., every element of $\mbk$ is a root of a separable irreducible polynomial in $\pf$ which splits over $\mbk$.
  \end{enumerate}
\end{theorem}
\begin{proof}

We prove that (1)$\Rightarrow$(2)$\Rightarrow$(3)$\Rightarrow$(4)$\Rightarrow$(1).

(1)$\Rightarrow$(2).

We do induction on $[\mbk: \mbf]$.

Suppose  $[\mbk: \mbf]=1$. Then $\mbk=\mbf$. So $\txau(\mbk/\mbf)=\{e\}$ and $|\txau(\mbk/\mbf)|=1=[\mbk: \mbf]$.

Suppose whenever $[\mbk:\mbe]<n$ and $\mbk$ is a splitting field of a separable polynomial $g(x)$ in $\mbe[x]$, we have $|\txau(\mbk/\mbe)|=[\mbk:\mbe]$.

Now assume that $[\mbk: \mbf]=n$.

Let $\alpha$ in $\mbk$ be a root of $f(x)$ with $\deg(m_{\alpha, \mbf})(x)=m\geq 2$. Note that $\mbk$ is also a splitting field of $f(x)$ in $\mbf(\alpha)[x]$. By induction, one has $|\txau(\mbk/\mbf(\alpha))|=[\mbk:\mbf(\alpha)]$.

For convenience, denote $\txau(\mbk/\mbf)$ by $G$ and $\txau(\mbk/\mbf(\alpha))$ by $G_\alpha$. Clearly $G_\alpha$ is a subgroup of $G$.

Next we construct a bijection $\psi_\alpha$ from $G/G_\alpha$ onto the set $X_\alpha$ of all roots of $m_{\alpha, \mbf}(x)$. Since $m_{\alpha, \mbf}(x)$ is a factor of $f(x)$ and $f(x)$ is separable, we have $|X_\alpha|=\deg(m_{\alpha, \mbf}(x))$.  So this bijection gives that $[G: G_\alpha]=\deg(m_{\alpha, \mbf}(x))=[\mbf(\alpha): \mbf]$ which implies that $|G|=|G_\alpha|[G: G_\alpha]=[\mbk:\mbf(\alpha)][\mbf(\alpha): \mbf]=[\mbk:\mbf]$.

The bijection $\psi_\alpha$ sends $\sigma G_\alpha$ to $\sigma(\alpha)$ for every $\sigma$ in $G$.

Every $\sigma\in G$ sends $\alpha$, a root of $m_{\alpha, \mbf}(x)$, to a root of $m_{\alpha, \mbf}(x)$.

If $\sigma G_\alpha=\tau G_\alpha$ for $\sigma,\tau\in G$, then $\tau^{-1}\sigma\in G_\alpha$ which implies that $\tau^{-1}\sigma(\alpha)=\alpha$. Hence $\sigma(\alpha)=\tau(\alpha)$ and $\psi_\alpha$ is well-defined.

If $\sigma(\alpha)=\tau(\alpha)$, then $\tau^{-1}\sigma(\alpha)=\alpha$ which means that $\tau^{-1}\sigma\in G_\alpha$ and $\sigma G_\alpha=\tau G_\alpha$. So $\psi_\alpha$ is injective.

To prove the surjectivity of $\psi_\alpha$, we need the following lemma.

\begin{lemma}~\label{lem:transitivity}
Suppose $\mbk$ is a splitting field of a separable polynomial $f$ in $\pf$, and $\alpha,\beta$ in $\mbk$ are roots of an irreducible factor $p$ of $f$. Then there exists  $\sigma$ in $\txau(\mbk/\mbf)$  such that $\sigma(\alpha)=\beta$.

\end{lemma}

\begin{proof}




Firstly there exists an isomorphism $\varphi$ between $\mbf(\alpha)$ and $\mbf(\beta)$ sending $\alpha$ to $\beta$ and fixing all elements in $\mbk$. Since $f$ is in  $\mbf(\alpha)[x]$ and $\mbf(\beta)[x]$, and $\mbk$ is a splitting field of $f$, by Lemma~\ref{lem:unispf},  there is  an isomorphism $\sigma: \mbk\to \mbk$ such that $\sigma(a)=\varphi(a)$ for all $a$ in $\mbf(\alpha)$. It follows that $\sigma$ is in $\txau(\mbk/\mbf)$  and $\sigma(\alpha)=\beta$.
\end{proof}

By Lemma~\ref{lem:transitivity}, for every $\beta\in X_\alpha$, one can find $\sigma\in G$ such that $\sigma(\alpha)=\beta$.

Hence $\psi_\alpha$ is surjective.

(2)$\Rightarrow$(3).

The goal is to prove that $[\mbk: \mbk^\txau(\mbk/\mbf)]\geq |\txau(\mbk/\mbf)|$. If this is done, noticing that  $\mbf$ is a subfield of $\mbk^{\txau(\mbk/\mbf)}$,  then
$$[\mbk:\mbf]\geq [\mbk: \mbk^{\txau(\mbk/\mbf)}]\geq |\txau(\mbk/\mbf)|=[\mbk:\mbf],$$
which implies that $\mbk^{\txau(\mbk/\mbf)}=\mbf$.


The inequality $[\mbk: \mbk^{\txau(\mbk/\mbf)}]\geq |\txau(\mbk/\mbf)|$ is a special case of the following theorem~\cite[Theorem 13]{Artin1944}.

\begin{theorem}~\label{thm:GaloisGroupSize}
Let $\{\sigma_1,\sigma_2,\cdots,\sigma_n\}$ be a subset of ${\rm} Aut(\mbk)$ and $\mbe=\{a\in \mbk|\,\sigma_i(a)=a\, {\rm for \,all}\, 1\leq i\leq n\}$ be its fixed field, then $[\mbk:\mbe]\geq n$.
\end{theorem}

We first need a preliminary result concerning group characters, which is in a more general context than we need.

\begin{definition}
A group homomorphism from a group $G$ to $\mbf^{\times}$~(the multiplication group consisting of nonzero elements in a field $\mbf$) is called a \textbf{character}.
\index{character}
Characters $\sigma_1,\cdots, \sigma_n$ are called \textbf{linearly independent} over $\mbf$ if $\lambda_i's$ in $\mbf$ satisfying that $\sum_{i=1}^{n}\lambda_i\sigma_i(x)=0$ for all $x\in G$ are all zeroes, otherwise call $\sigma_1,\cdots, \sigma_n$ \textbf{linearly dependent} over $\mbf$.~\index{linearly independent characters}
\end{definition}

\begin{theorem}~\label{thm: LinearIndependence}
A finite set of (distinct) characters $\{\sigma_1,\sigma_2,\cdots,\sigma_n\}$ is linearly independent.
\end{theorem}
\begin{proof}
It's clear that $\{\sigma_1\}$ is linearly independent since $\sigma_1$ is nonzero.

Suppose any set with less than $n$ distinct characters is linearly independent.

Now assume that $\{\sigma_1,\sigma_2,\cdots,\sigma_n\}$ is linearly dependent. Then there exist nonzero $a_1,a_2,\cdots, a_n$ such that
$\sum_{i=1}^{n}a_i\sigma_i(x)=0$ for all $x\in G$. Note that every $a_i$ must be nonzero otherwise it leads to a contradiction to the assumption. So we have $a_1a_n^{-1}\sigma_1(x)+a_2a_n^{-1}\sigma_2(x)+\cdots+a_{n-1}a_n^{-1}\sigma_{n-1}(x)+\sigma_n(x)=0$ for every $x$ in $G$. Since $\sigma_1\neq\sigma_n$, one has $\sigma_1(y)\neq\sigma_n(y)$ for some $y$ in $G$. Then $a_1a_n^{-1}\sigma_1(yx)+a_2a_n^{-1}\sigma_2(yx)+\cdots+a_{n-1}a_n^{-1}\sigma_{n-1}(yx)+\sigma_n(yx)=0$ for every $x$ in $G$. It follows that $a_1a_n^{-1}\sigma_n(a^{-1})\sigma_1(y)\sigma_1(x)+a_2a_n^{-1}\sigma_n(y^{-1})\sigma_2(y)\sigma_2(x)+\cdots+a_{n-1}a_n^{-1}\sigma_n(y^{-1})\sigma_{n-1}(y)\sigma_{n-1}(x)+\sigma_n(x)=0$ for every $x$ in $G$. Hence
\begin{align*}
&a_1a_n^{-1}(\sigma_1(y)\sigma_n(y^{-1})-1)\sigma_1(x)+(a_2a_n^{-1}\sigma_n(y^{-1})\sigma_2(y)-a_2a_n^{-1})\sigma_2(x) \\
&+\cdots+(a_{n-1}a_n^{-1}\sigma_n(y^{-1})\sigma_{n-1}(y)-a_{n-1}a_n^{-1})\sigma_{n-1}(x)=0
\end{align*}
for every $x$ in $G$. But $a_1a_n^{-1}(\sigma_1(y)\sigma_n(y^{-1})-1)\neq 0$, which is a contradiction to that $\{\sigma_1,\cdots,\sigma_{n-1}\}$ is linearly independent over $\mbf$.
\end{proof}

\begin{proof}~[Proof of Theorem~\ref{thm:GaloisGroupSize}]

 Suppose $r=[\mbk:\mbe]<n$. Let $\{v_1,v_2,\cdots, v_r\}$ be a basis of $\mbk$ over $\mbe$. Consider the homogeneous linear equation system

 \begin{equation*}
  \begin{aligned}
    &x_1\sigma_1(v_1)+x_2\sigma_2(v_1)+\cdots+x_n\sigma_n(v_1)=0 \\
    &x_1\sigma_1(v_2)+x_2\sigma_2(v_2)+\cdots+x_n\sigma_n(v_2)=0 \\
    &\hdots\qquad\qquad \hdots\qquad\qquad\hdots\\
    &x_1\sigma_1(v_r)+x_2\sigma_2(v_r)+\cdots+x_n\sigma_n(v_r)=0.
  \end{aligned}
  \end{equation*}
  Note that this system has more unknowns than equations, hence it has a nonzero solution, say, $x_1, x_2,\cdots, x_n$ in $\mbk$. Every $a$ in $\mbk$ is a linear combination of $v_1,v_2,\cdots, v_n$, so $x_1\sigma_1(a)+x_2\sigma_2(a)+\cdots+x_n\sigma_n(a)=0$. This means that $\{\sigma_1,\sigma_2,\cdots,\sigma_n\}$ is linearly independent  over $\mbk$.  This is a contradiction to Theorem.~\ref{thm: LinearIndependence} 
\end{proof}

(3)$\Rightarrow$(4).

Suppose $\alpha$ in $\mbk$ is a root of  an irreducible polynomial $p(x)$ in $\pf$. Let $G=\txau(\mbk/\mbf)=\{\sigma_1,\sigma_2,\cdots,\sigma_n\}$. Suppose $\{\alpha,\alpha_1,\alpha_2,\cdots,\alpha_r\}$ are all distinct elements of $\{\sigma_1(\alpha),\sigma_2(\alpha),\cdots,\sigma_n(\alpha)\}$.

Consider $q(x)=(x-\alpha)(x-\alpha_1)\cdots(x-\alpha_r)$. Note that $q$ is fixed by $\sigma_i$ for all $1\leq i\leq n$ since each $\sigma_i$ is a permutation of $\{\alpha,\alpha_1,\alpha_2,\cdots,\alpha_r\}$. Hence all coefficients of $q$ are fixed by $G$. By (3), $\mbk^G=\mbf$. So all coefficients of $q$ are in $\mbf$, that is, $q$ is in $\pf$. Note that $q(\alpha)=0$ and $\deg(q)\leq \deg(p)$. So we have $q=p$.

Therefore $p$ is separable and all roots of $p$  are in $\mbk$.

(4)$\Rightarrow$(1).

Let $\{v_1, v_2,\cdots, v_n\}$ be a basis of $\mbk$ over $\mbf$. 

Hence one can find monic separable irreducible polynomials $p_i(x)$ in $\pf$ such that $p_i(v_i)=0$ for every $1\leq i\leq n$. Let $f(x)$ be the product of all distinct $p_i's$. Hence $f$ is separable since each $p_i$ is separable, and $\mbk$ is a splitting field of $f(x)$.

\end{proof}

\begin{remark}
Lemma~\ref{lem:transitivity} says that the action of $\txau(\mbk/\mbf)$ on the roots of an irreducible polynomial $p(x)$ in $\pf$ which splits over $\mbk$ is transitive, that is, the orbit of any root of $p$ under $\txau(\mbk/\mbf)$ exhausts all roots of $p$.
\end{remark}

\begin{definition}
A finite extension $\mbk$ of $\mbf$ is called a \textbf{Galois extension} \index{Galois extension} if one of the four conditions in Theorem~\ref{thm:GaloisCharacterizations} holds. In this case, denote $\txau(\mbk/\mbf)$ by ${\rm Gal}(\mbk/\mbf)$.
\end{definition}

Consider the field extension $\mbq(\sqrt{2})/\mbq$. Every $\sigma$ in $\txau(\mbq(\sqrt{2}))$ satisfies that $\sigma(1)=1$. Hence
$\txau(\mbq(\sqrt{2})/\mbq)=\txau(\mbq(\sqrt{2}))$. The minimal polynomial of $\sqrt{2}$ is $x^2-2$ whose roots are $\pm\sqrt{2}$.
$\txau(\mbq(\sqrt{2})/\mbq)=\{e,\sigma\}$ where $\sigma(a+b\sqrt{2})=a-b\sqrt{2}$ for all $a,b$ in $\mbq$. Then $|\txau(\mbq(\sqrt{2})/\mbq)|=2=[\mbq(\sqrt{2}):\mbq]$. Hence $\mbq(\sqrt{2})/\mbq$ is a Galois extension.

Consider the field extension $\mbq(\sqrt[3]{2})/\mbq$. Also $\txau(\mbq(\sqrt[3]{2})/\mbq)=\txau(\mbq(\sqrt[3]{2}))$. The minimal polynomial of $\sqrt[3]{2}$ is $x^3-2$ whose roots are $\sqrt[3]{2}$, $\sqrt[3]{2}(-\frac{1}{2}+\frac{\sqrt{3}}{2}i)$ and $\sqrt[3]{2}(-\frac{1}{2}-\frac{\sqrt{3}}{2}i)$. For every $\sigma$ in $\txau(\mbq(\sqrt[3]{2})/\mbq)$, $\sigma(\sqrt[3]{2})$ cannot be $\sqrt[3]{2}(-\frac{1}{2}+\frac{\sqrt{3}}{2}i)$ or $\sqrt[3]{2}(-\frac{1}{2}-\frac{\sqrt{3}}{2}i)$ otherwise $\sqrt{3}i$ belongs to $\mbq(\sqrt[3]{2})\subseteq\mbr$ which is a contradiction. Hence $\sigma(\sqrt[3]{2})=\sqrt[3]{2}$.
Therefore  $\txau(\mbq(\sqrt[3]{2})/\mbq)$ contains only the identity map which implies that $|\txau(\mbq(\sqrt[3]{2})/\mbq)|=1<[\mbq(\sqrt[3]{2}):\mbq]=3$. So $\mbq(\sqrt[3]{2})/\mbq$ is not a Galois extension.

\section{The fundamental theorem of Galois theory}

In this section, we prove the fundamental theorem of Galois theory.

Let $\mbk$ be a field and $H$ is a subgroup of $\txau(\mbk)$. Define
$$\mbk^H=\{a\in\mbk\,|\,\sigma(a)=a\, {\rm for\,all}\,\sigma\,{\rm in}\, H\}.$$

\begin{lemma}
$\mbk^H$ is a subfield of $\mbk$.
\end{lemma}
\begin{proof}
  If $a,b$ in $\mbk^H$, then $\sigma(a-b)=\sigma(a)-\sigma(b)=a-b$ for every $\sigma$ in $H$ which means that $a-b$ is in $\mbk^H$. 

  If $a,b$ in $\mbk^H$ and $b\neq 0$, then $\sigma(ab^{-1})=\sigma(a)\sigma(b^{-1})=\sigma(a)\sigma(b)^{-1}=ab^{-1}$ for all $\sigma$ in $H$ which means that $ab^{-1}$ is in $\mbk^H$.

  We conclude from the above that $\mbk^H$ is a subfield of $\mbk$.
\end{proof}

\begin{example}\

  \begin{enumerate}
    \item   Recall that $\txau(\mbq(\sqrt{2})/\mbq)=\{e,\sigma\}$ where $\sigma$ is given by $\sigma(a+b\sqrt{2})=a-b\sqrt{2}$ for all $a,b$ in $\mbq$. If $\sigma(a+b\sqrt{2})=a+b\sqrt{2}$, then $b=0$. Hence $\mbq(\sqrt{2})^{\txau(\mbq(\sqrt{2})/\mbq)}=\mbq$.
    \item $\mbq(\sqrt[3]{2})^{\txau(\mbq(\sqrt[3]{2})/\mbq)}=\mbq(\sqrt[3]{2})^{\{e\}}=\mbq(\sqrt[3]{2})$.
  \end{enumerate}
\end{example}

\begin{theorem}[Fundamental Theorem of Galois Theory]
~\label{thm:FundamentalGalois}

Suppose $\mbk$ is a Galois extension of $\mbf$.

The map $\Phi: \mathcal{IF}=\{{\rm subfields\, of}\, \mbk\, {\rm containing}\, \mbf\}\to \mathcal{SG}=\{{\rm subgroups\, of} \,{\rm Gal}(\mbk/\mbf)\}$ given by $\Phi(\mbe)=\txau(\mbk/\mbe)$ is a bijection whose inverse map $\Psi: \mathcal{SG}\to \mathcal{IF}$ is given by $\Psi(H)=\mbk^H$. Furthermore  the following hold.
\begin{enumerate}
  \item $\mbk$ is always a Galois extension of $\mbe$, and $[\mbe:\mbf]=|{\rm Gal}(\mbk/\mbf):{\rm Gal}(\mbk/\mbe)|$;
  \item $\mbe$ is a Galois extension of $\mbf$ iff ${\rm Gal}(\mbk/\mbe)$ is a normal subgroup of ${\rm Gal}(\mbe/\mbf)$ and ${\rm Gal}(\mbe/\mbf)\cong {\rm Gal}(\mbk/\mbf)/{\rm Gal}(\mbk/\mbe)$;
  \item For  $\mbe_1$ and $\mbe_2$ in $\mathcal{IF}$, we have $\mbe_1\subseteq \mbe_2$ iff ${\rm Gal}(\mbk/\mbe_1)\supseteq {\rm Gal}(\mbk/\mbe_2)$. Moreover ${\rm Gal}(\mbk/\mbe_1\cap \mbe_2)= {\rm Gal}(\mbk/\mbe_1)\vee {\rm Gal}(\mbk/\mbe_2)$~\footnote{For subgroups $H_1$ and $H_2$ of a group $G$, the notation $H_1\vee H_2$ stands for the smallest subgroup of $G$ containing $H_1$ and $H_2$.} and ${\rm Gal}(\mbk/\mbe_1\mbe_2)={\rm Gal}(\mbk/\mbe_1)\cap {\rm Gal}(\mbk/\mbe_2)$.
  \end{enumerate}
\end{theorem}

To prove Theorem~\ref{thm:FundamentalGalois}, we need some preliminaries.

\begin{theorem}~\label{Thm:Intermediate}
Suppose $G$ is a finite subgroup of $\txau(\mbk)$. Then $[\mbk: \mbk^G]=|G|$.
\end{theorem}
\begin{proof}

By Theorem~\ref{thm:GaloisGroupSize}, we have  $[\mbk: \mbk^G]\geq |G|$.

We prove $[\mbk: \mbk^G]\leq |G|$ by contradiction.

Denote the group $G$ by $\{\sigma_1,\sigma_2,\cdots, \sigma_n\}$ and assume that $[\mbk: \mbk^G]>n$.

Let $\{v_1, v_2, \cdots, v_{n+1}\}$ be a linearly independent set in $\mbk$ over $\mbk^G$.

Consider the homogeneous linear equation system
\begin{equation}~\label{eq:1}
  \begin{aligned}
    &x_1\sigma_1(v_1)+x_2\sigma_1(v_2)+\cdots+x_{n+1}\sigma_1(v_{n+1})=0 \\
    &x_1\sigma_2(v_1)+x_2\sigma_2(v_2)+\cdots+x_{n+1}\sigma_2(v_{n+1})=0 \\
    &\hdots\qquad\qquad \hdots\qquad\qquad\hdots\\
    &x_1\sigma_n(v_1)+x_2\sigma_n(v_2)+\cdots+x_{n+1}\sigma_n(v_{n+1})=0.
  \end{aligned}
  \end{equation}
This system has more unknowns than equations, hence has a nonzero solution in $\mbk$. Let $(y_1, y_2,\cdots,y_{n+1})$ be a solution with the smallest number of nonzeros. Without loss of generality, one may assume that $(y_1, y_2, \cdots, y_{m-1},1, 0,\cdots,0)$ is such a solution. clearly $m\geq 2$. Also
\begin{equation}~\label{s1}
  \begin{aligned}
    &y_1\sigma_1(v_1)+y_2\sigma_1(v_2)+\cdots+y_{m-1}\sigma_1(v_{m-1})+\sigma_1(v_m)=0 \\
    &y_1\sigma_2(v_1)+y_2\sigma_2(v_2)+\cdots+y_{m-1}\sigma_2(v_{m-1})+\sigma_2(v_m)=0 \\
    &\hdots\qquad\qquad \hdots\qquad\qquad\hdots\\
    &y_1\sigma_n(v_1)+y_2\sigma_n(v_2)+\cdots+y_{m-1}\sigma_n(v_{m-1})+\sigma_n(v_m)=0.
  \end{aligned}
  \end{equation}

Not every $y_j$ for $1\leq j\leq m$ is in  $\mbk^G$ otherwise a contradiction to linear independence of $\{v_1, v_2,\cdots, v_m\}$. Assume that $y_1\notin \mbk^G$ for convenience. Then there exists $\sigma_{k_0}$ such that $\sigma_{k_0}(y_1)\neq y_1$.

Let $\sigma_{k_0}$ act on both sides of ~\ref{s1}. Then we get
  \begin{equation}~\label{s2}
  \begin{aligned}
    &\sigma_{k_0}(y_1)\sigma_{k_0}\sigma_1(v_1)+\sigma_{k_0}(y_2)\sigma_{k_0}\sigma_1(v_2)+\cdots+\sigma_{k_0}(y_{m-1})\sigma_{k_0}\sigma_1(v_{m-1})+\sigma_{k_0}\sigma_1(v_m)=0 \\
    &\sigma_{k_0}(y_1)\sigma_{k_0}\sigma_2(v_1)+\sigma_{k_0}(y_2)\sigma_{k_0}\sigma_2(v_2)+\cdots+\sigma_{k_0}(y_{m-1})\sigma_{k_0}\sigma_2(v_{m-1})+\sigma_{k_0}\sigma_2(v_m)=0 \\
    &\hdots\qquad\qquad \hdots\qquad\qquad\hdots\\
    &\sigma_{k_0}(y_1)\sigma_{k_0}\sigma_n(v_1)+\sigma_{k_0}(y_2)\sigma_{k_0}\sigma_n(v_2)+\cdots+\sigma_{k_0}(y_{m-1})\sigma_{k_0}\sigma_n(v_{m-1})+\sigma_{k_0}\sigma_n(v_m)=0.
  \end{aligned}
  \end{equation}

Since $G$ is a group, one has $\{\sigma_{k_0}\sigma_1,\sigma_{k_0}\sigma_2,\cdots, \sigma_{k_0}\sigma_n\}=\{\sigma_1,\sigma_2,\cdots, \sigma_n\}$.

Hence $\{\sigma_{k_0}(y_1), \sigma_{k_0}(y_2), \cdots, \sigma_{k_0}(y_{m-1}),1, 0,\cdots,0)$ is also a solution of  the linear equation system~\ref{eq:1}.
It follows from $y_1-\sigma_{k_0}(y_1)\neq 0$ that $\{\sigma_{k_0}(y_1)-y_1, \sigma_{k_0}(y_2)-y_2, \cdots, \sigma_{k_0}(y_{m-1})-y_{m-1},0, 0,\cdots,0)$ is a  nonzero solution of ~\ref{eq:1}, which only has at most $m-1$ many nonzeros.

This is a contradiction to that  a solution of ~\ref{eq:1} has at least $m$ nonzeros.
\end{proof}
\begin{proof}[Proof of Theorem~\ref{thm:FundamentalGalois}]
From Theorem~\ref{thm:GaloisCharacterizations}(1), $\mbk$ is the splitting field of a separable polynomial $f(x)$ in $\pf$. For every subfield $\mbe$ of $\mbk$ containing $\mbf$,  $\mbk$ is also the splitting field of $f(x)$ in $\mbe[x]$. Hence $\mbk$ is a Galois extension of $\mbe$.

For every $\mbe$ in $\mathcal{IF}$, we have that $\Psi\Phi(\mbe)=\Psi(\txau(\mbk/\mbe))=\mbk^{\txau(\mbk/\mbe)}=\mbe$ since $\mbk/\mbe$ is a Galois extension.

For every $H$ in $\mathcal{SG}$, we have that $\Phi\Psi(H)=\Phi(\mbk^H)=\txau(\mbk/\mbk^H)$. Clearly $H$ is a subgroup of  $\txau(\mbk/\mbk^H)$. Moreover since $\mbk/\mbk^H$ is a Galois extension, we obtain that $|\txau(\mbk/\mbk^H)|=[\mbk:\mbk^H]$. It follows from Theorem~\ref{Thm:Intermediate} that $[\mbk:\mbk^H]=|H|$. So $|\txau(\mbk/\mbk^H)|=|H|$ which implies that $H=\txau(\mbk/\mbk^H)$. 

In summarize, we  have proved  that $\Psi\Phi$ is the identity map on $\mathcal{IF}$ and $\Phi\Psi$ is the identity map on $\mathcal{SG}$.


(1) We already show that  $\mbk$ is a Galois extension of $\mbe$. Moreover $[\mbk:\mbf]=[\mbk:\mbe][\mbe:\mbf]$. Note that $[\mbk:\mbf]=|{\rm Gal}(\mbk/\mbf)|$ and $[\mbk:\mbe]=|{\rm Gal}(\mbk/\mbe)|$. Hence $[\mbe:\mbf]=|{\rm Gal}(\mbk/\mbf):{\rm Gal}(\mbk/\mbe)|$.

(2) Suppose ${\rm Gal}(\mbk/\mbe)$ is a normal subgroup of ${\rm Gal}(\mbk/\mbf)$. One can define $\varphi: {\rm Gal}(\mbk/\mbf)/{\rm Gal}(\mbk/\mbe)\to\txau(\mbe/\mbf)$ by $\varphi(\sigma {\rm Gal}(\mbk/\mbe))=\sigma|_\mbe$ for every $\sigma$ in ${\rm Gal}(\mbk/\mbf)$.

To check that $\varphi$ is  well-defined, one need verify that $\sigma(\mbe)=\mbe$ for every $\sigma\in {\rm Gal}(\mbk/\mbf)$.

Note that $\mbk$ is a Galois extension of $\mbe$, so $\mbe=\mbk^{{\rm Gal}(\mbk/\mbe)}$ by Theorem~\ref{thm:GaloisCharacterizations}(3). Also we have
\begin{claim}
$\sigma(\mbe)=\mbk^{\sigma {\rm Gal}(\mbk/\mbe)\sigma^{-1}}$ for every $\sigma\in {\rm Gal}(\mbk/\mbf)$.
\end{claim}
\begin{proof}
  For every $a$ in $\mbe$ and every $\tau\in {\rm Gal}(\mbk/\mbe)$, we have $\sigma\tau\sigma^{-1}(\sigma(a))=\sigma(a)$. Hence $\sigma(\mbe)\subseteq \mbk^{\sigma {\rm Gal}(\mbk/\mbe)\sigma^{-1}}$.

  On the other hand for every $b\in \mbk^{\sigma {\rm Gal}(\mbk/\mbe)\sigma^{-1}}$, one has $\sigma\tau\sigma^{-1}(b)=b$ for every $\tau\in {\rm Gal}(\mbk/\mbe)$. So $\sigma^{-1}(b)$ is in $\mbk^{{\rm Gal}(\mbk/\mbe)}=\mbe$, that is, $b$ is in $\sigma(\mbe)$. Therefore $\sigma(\mbe)\supseteq \mbk^{\sigma {\rm Gal}(\mbk/\mbe)\sigma^{-1}}$.
\end{proof}

Since ${\rm Gal}(\mbk/\mbe)$ is a normal subgroup of ${\rm Gal}(\mbk/\mbf)$, we obtain that $\sigma(\mbe)=\mbe$ for every $\sigma\in {\rm Gal}(\mbk/\mbf)$.

It's routine to check that $\varphi$ is an injective group homomorphism.

For every $\tau\in \txau(\mbe/\mbf)$, one can find an automorphism $\hat{\tau}$ on $\mbk$ such that $\hat{\tau}|_\mbe=\tau$ since $\mbk$ is a splitting field of a separable polynomial in $\pf$. This gives surjectivity of $\varphi$.

So $|\txau(\mbe/\mbf)|=|{\rm Gal}(\mbk/\mbf):{\rm Gal}(\mbk/\mbe)|=\frac{[\mbk:\mbf]}{[\mbk:\mbe]}=[\mbe:\mbf]$. Hence by Theorem~\ref{thm:GaloisCharacterizations}(2), $\mbe$ is a Galois extension of $\mbf$.

Suppose $\mbe$ is a Galois extension of $\mbf$. We are going to prove that $\sigma(\mbe)=\mbe$ for every $\sigma\in {\rm Gal}(\mbk/\mbf)$.

After proving this, we can define a  group homomorphism from ${\rm Gal}(\mbk/\mbf)\to {\rm Gal}(\mbe/\mbf)$ by sending $\sigma$  to $\sigma|_\mbe$. The kernel of this homomorphism is ${\rm Gal}(\mbk/\mbe)$. So ${\rm Gal}(\mbk/\mbe)$ is a normal subgroup of ${\rm Gal}(\mbk/\mbf)$.

By assumption $\mbe$ is a splitting field of a separable polynomial $f$ in $\pf$. Let $\alpha_1,\alpha_2,\cdots, \alpha_n\in \mbe$ be all roots of $f(x)$. Hence $\mbe=\mbf(\alpha_1,\cdots,\alpha_n)$ and $\sigma(\mbe)=\mbf(\sigma(\alpha_1),\cdots,\sigma(\alpha_n))$. Since $\sigma$ is in ${\rm Gal}(\mbk/\mbf)$, $\sigma(\alpha_1),\cdots,\sigma(\alpha_n)$ are all roots of $f(x)$. So $\sigma(\mbe)\subseteq \mbe$. Note that $[\sigma(\mbe):\mbf]=[\sigma(\mbe):\sigma(\mbf)]=[\mbe:\mbf]$. Hence $\sigma(\mbe)=\mbe$.


(3) It's obvious that $\mbe_1\subseteq \mbe_2$ implies ${\rm Gal}(\mbk/\mbe_1)\supseteq {\rm Gal}(\mbk/\mbe_2)$.

The converse follows from $\mbe_1=\mbk^{{\rm Gal}(\mbk/\mbe_1)}$ and $\mbe_2=\mbk^{{\rm Gal}(\mbk/\mbe_2)}$.

It follows that ${\rm Gal}(\mbk/\mbe_1)$ and ${\rm Gal}(\mbk/\mbe_2)$ are subgroups of ${\rm Gal}(\mbk/\mbe_1\cap \mbe_2)$. Hence $ {\rm Gal}(\mbk/\mbe_1)\vee {\rm Gal}(\mbk/\mbe_2)$ is also a subgroup of ${\rm Gal}(\mbk/\mbe_1\cap \mbe_2)$.

On the other hand $\mbk^{ {\rm Gal}(\mbk/\mbe_1)\vee {\rm Gal}(\mbk/\mbe_2)}$ is a subfield of $\mbk^{{\rm Gal}(\mbk/\mbe_1)}=\mbe_1$ and $\mbk^{{\rm Gal}(\mbk/\mbe_2)}=\mbe_2$. Hence $\mbk^{{\rm Gal}(\mbk/\mbe_1)\vee {\rm Gal}(\mbk/\mbe_2)}$ is a subfield of $\mbe_1\cap \mbe_2$. Therefore
${\rm Gal}(\mbk/\mbe_1)\vee {\rm Gal}(\mbk/\mbe_2)={\rm Gal}(\mbk/\mbk^{ {\rm Gal}(\mbk/\mbe_1)\vee {\rm Gal}(\mbk/\mbe_2)})$ is a subgroup of ${\rm Gal}(\mbk/\mbe_1\cap \mbe_2)$.

Since $\mbe_1$ and $\mbe_2$ are subfields of $\mbe_1\mbe_2$,  ${\rm Gal}(\mbk/\mbe_1\mbe_2)\subseteq {\rm Gal}(\mbk/\mbe_1)\cap {\rm Gal}(\mbk/\mbe_2)$. Moreover if $\sigma$ is in ${\rm Gal}(\mbk/\mbe_1)\cap {\rm Gal}(\mbk/\mbe_2)$, then it fixes $\mbe_1$ and $\mbe_2$. Thus $\sigma$ fixes $\mbe_1\mbe_2$. So ${\rm Gal}(\mbk/\mbe_1\mbe_2)\supseteq {\rm Gal}(\mbk/\mbe_1)\cap {\rm Gal}(\mbk/\mbe_2)$.
\end{proof}

\section{Computations of Galois groups}

The \textbf{Galois group of a separable polynomial} $f(x)$ in $\pf$ \index{Galois group of a separable polynomial} means the Galois group ${\rm Gal}(\mbk/\mbf)$ with $\mbk$ being the splitting field of $f$.

(1) Galois group of $\mbq(\sqrt{2},\sqrt{3})/\mbq$:

 $\mbq(\sqrt{2},\sqrt{3})$ is the splitting field of $(x^2-2)(x^2-3)$ in $\pq$, hence $\mbq(\sqrt{2},\sqrt{3})/\mbq$ is a Galois extension. Moreover $G={\rm Gal}(\mbq(\sqrt{2},\sqrt{3})/\mbq)=\{e,\sigma,\tau,\sigma\tau\}$ where $\sigma:\mbq(\sqrt{2},\sqrt{3})\to \mbq(\sqrt{2},\sqrt{3})$ is given by $\sigma(\sqrt{2})=-\sqrt{2}$ and $\sigma(\sqrt{3})=\sqrt{3}$ and $\tau:\mbq(\sqrt{2},\sqrt{3})\to \mbq(\sqrt{2},\sqrt{3})$ is given by $\tau(\sqrt{2})=\sqrt{2}$ and $\sigma(\sqrt{3})=-\sqrt{3}$.

Denote $\mbq(\sqrt{2},\sqrt{3})$ by $\mbk$.

 The correspondence between subgroups of $G$ and intermediate fields between $\mbq$ and $\mbk$ is given by the following chart:
 \begin{center}

 \begin{tabular}{| c | c | }
    \hline
    $H$&$\mbk^H$  \\ \hline
    $\{e\}$&$\mbq(\sqrt{2},\sqrt{3})$\\ \hline
     $\{e,\sigma\}$&$\mbq(\sqrt{3})$\\ \hline
      $\{e,\tau\}$&$\mbq(\sqrt{2})$\\ \hline
       $\{e,\sigma\tau\}$&$\mbq(\sqrt{6})$\\ \hline
        $G$&$\mbq$\\ \hline

  \end{tabular}
\end{center}

(2) Galois group of $\mbq(\sqrt[3]{2},\sqrt{3}i)/\mbq$:
$\mbq(\sqrt[3]{2},\sqrt{3}i)$ is the splitting field of the irreducible polynomial $x^3-2$ in $\pq$, hence $\mbq(\sqrt[3]{2},\sqrt{3}i)/\mbq$ is a Galois extension. The polynomial $x^3-2$ has 3 roots: $\omega_1=\sqrt[3]{2}$, $\omega_2=\sqrt[3]{2}(-\frac{1}{2}+\frac{\sqrt{3}}{2}i)$ and $\omega_3=\sqrt[3]{2}(-\frac{1}{2}-\frac{\sqrt{3}}{2}i)$.

Every element in ${\rm Gal}(\mbq(\sqrt[3]{2},\sqrt{3}i)/\mbq)$ is a permutation of these 3 roots, hence $G={\rm Gal}(\mbq(\sqrt[3]{2},\sqrt{3}i)/\mbq)$ is a subgroup of $S_3$. Since $|{\rm Gal}(\mbq(\sqrt[3]{2},\sqrt{3}i)/\mbq)|=[\mbq(\sqrt[3]{2}: \sqrt{3}i)/\mbq]=6$, we have $G=S_3$.

Let $\mbk=\mbq(\sqrt[3]{2},\sqrt{3}i)$.

 The correspondence between subgroups of $G$ and intermediate fields between $\mbq$ and $\mbk$ is listed below:
 \begin{center}

 \begin{tabular}{| c | c | }
    \hline
    $H$&$\mbk^H$  \\ \hline
    $\{e\}$&$\mbq(\sqrt[3]{2},\sqrt{3}i)$\\ \hline
     $\{e,(1\,2)\}$&$\mbq(\omega_3)$\\ \hline
      $\{e,(1\,3)\}$&$\mbq(\omega_2)$\\ \hline
       $\{e,(2\,3)\}$&$\mbq(\omega_1)$\\ \hline
       $\{e,(1\,2\,3),(1\,3\,2)\}$&$\mbq(\sqrt{3}i)$\\ \hline
        $G$&$\mbq$\\ \hline

  \end{tabular}
\end{center}

(3) Galois group of $x^n-a$ in $\pf$:

\begin{theorem}~\label{thm:Roots}
Suppose $\mbf$ is a field with ${\rm Ch}\mbf\nmid n$ and $\mbk$ is a field over which $x^n-1$ splits. Then  $X_n=\{a\in \mbk\,|\, a^n-1=0\}$ is a cyclic group of order $n$.
\end{theorem}
\begin{proof}
Firstly  $D(x^n-1)=nx^{n-1}\neq 0$  since ${\rm Ch}\mbf\nmid n$. Hence $\gcd(x^n-1,nx^{n-1})=1$. So $x^n-1=0$ is separable.

Therefore $X_n$ is an abelian group of order $n$. Suppose $n=p_1^{n_1}\cdots p_k^{n_k}$ for distinct primes $p_1,\cdots, p_k$ and positive integers $n_1,\cdots, n_k$. By Sylow's Theorem, there exist  subgroups $G_1, G_2, \cdots, G_k$ of $X_n$ whose orders are $p_1^{n_1},\cdots, p_k^{n_k}$ respectively. Since $X_n$ is abelian, each $G_i$ is a normal subgroup of $x_n$. By the second isomorphism theorem of groups, we have that  $X_n=G_1G_2\cdots G_k$. 

To prove that $X_n$ is cyclic, it suffices to prove that every $G_i$ is cyclic.


Let $m=\displaystyle\max_{\alpha\in G_i} \{{\rm order\, of}\, \alpha\}$. If $m=p_i^d$ for some $d<n_i$, then  $\alpha^m=1$ for every  $\alpha$ in $G_i$, that is, every element in $G_i$ is a root of $x^m-1=0$. Note that $x^m-1=0$ has at most $m$ distinct roots in $\mbk$, however $G_i$ has $p_i^{n_i}>m$ many distinct elements. This leads to a contradiction.  So there exists an element in $G_i$ with the order $p_i^{n_i}=|G_i|$, that is, $G_i$ is cyclic. 
\end{proof}

\begin{proposition}
Suppose $\mbf$ is a field with ${\rm Ch}\mbf\nmid n$. Then the Galois group of  $x^n-1$ in $\pf$ is isomorphic to $\mbz_n^{\times}$.
\end{proposition}
\begin{proof}
 Let $\mbk$ be a splitting field of $x^n-1$. Denote the set of roots of $x^n-1=0$ by $X_n$.

 By Theorem~\ref{thm:Roots}, $\mbk=\mbf(\zeta)$ for $\zeta$ being a generator of the cyclic group $X_n$.

  Then every $\sigma$ in ${\rm Gal}(\mbk/\mbf)$ is completely determined by $\sigma(\zeta)$. It's easy to see that $\sigma(\zeta)$ is also a generator of $X_n$, hence we can define a map $\Phi: {\rm Gal}(\mbk/\mbf)\to \mbz_n^{\times}$ by $\sigma(\zeta)=\zeta^{\Phi(\sigma)}$.

One can verify the following.

\begin{enumerate}
  \item $\zeta^{\Phi(\tau\sigma)}=\tau\sigma(\zeta)=\tau(\zeta^{\Phi(\sigma)})=\tau(\zeta)^{\Phi(\sigma)}=\zeta^{\Phi(\tau)\Phi(\sigma)}$ for all $\sigma,\tau$ in ${\rm Gal}(\mbk/\mbf)$, which means $\Phi$ is a group homomorphism.
  \item If $\Phi(\sigma)=\Phi(\tau)$ in $\mbz_n^{\times}$, then $\sigma(\zeta)=\tau(\zeta)$. Hence $\Phi$ is injective.
  \item For every $m$ in  $\mbz_n^{\times}$, we can define an automorphism on $\mbk$ by sending to $p(\zeta)$ to $p(\zeta^m)$ for every $p$ in $\pf$~\footnote{Note that $\mbf(\zeta)=\mbf[\zeta]$.}. So $\Phi$ is surjective.
\end{enumerate}

Hence ${\rm Gal}(\mbk/\mbf)$ is isomorphic to  $\mbz_n^{\times}$.
\end{proof}

 We are going to prove the following.

\begin{theorem}
~\label{thm: cyclicextension}
Suppose $\mbf$ is a field with ${\rm Ch}\mbf\nmid n$ and all roots of $x^n-1$ are in $\mbf$. Then the Galois group of  $x^n-a$ in $\pf$ is a cyclic group of order dividing  $n$. Conversely if ${\rm Gal}(\mbk/\mbf)$ is a cyclic group of order $k$ and $k$ divides $n$, then $\mbk=\mbf(\sqrt[k]{a})$ for some $a\in \mbf$.
\end{theorem}
\begin{proof}
Let $\mbk$ be the splitting field of the polynomial $x^n-a$. We are going to define an injective homomorphism from ${\rm Gal}(\mbk/\mbf)$ into $\mbz_n$.

Let $\sqrt[n]{a}\in \mbk$ be a root of $x^n-a=0$ and $\zeta$ be a primitive root of $x^n-1=0$. Then $\{\sqrt[n]{a}\zeta^j\}_{j=0}^{n-1}$ exhaust all roots of $x^n-a=0$. Every $\sigma\in {\rm Gal}(\mbk/\mbf)$ maps $\sqrt[n]{a}$ to $\sqrt[n]{a}\zeta^j$ for some $0\leq j\leq n-1$ and denote $j$  by $m_\sigma$.  Define a map $\Phi:{\rm Gal}(\mbk/\mbf)\to \mbz_n$ by $\Phi(\sigma)=m_\sigma$.

We verify that $\Phi$ is an injective homomorphism.

Firstly $\sigma\tau(\sqrt[n]{a})=\sqrt[n]{a}\zeta^{m_{\sigma\tau}}$ for all $\sigma,\tau\in {\rm Gal}(\mbk/\mbf)$. On the other hand

$$\sigma\tau(\sqrt[n]{a})=\sigma(\tau(\sqrt[n]{a}))=\sigma(\sqrt[n]{a}\zeta^{m_\tau})=\zeta^{m_\tau} \sigma(\sqrt[n]{a})=\sqrt[n]{a}\zeta^{m_\tau}\zeta^{m_\sigma}=\sqrt[n]{a}\zeta^{m_\tau+m_\sigma}.$$

Hence $\Phi$ is a group homomorphism.

Secondly if $m_\sigma=m_\tau$ then $\sigma=\tau$ since $\mbk=\mbf(\sqrt[n]{a})$. So $\Phi$ is  injective. Therefore ${\rm Gal}(\mbk/\mbf)$ is isomorphic to a subgroup of $\mbz_n$ which is a cyclic group of order dividing $n$.

Now suppose $\mbk$ is a Galois extension of $\mbf$ with ${\rm Gal}(\mbk/\mbf)=\{1,\sigma,\sigma^2,\cdots,\sigma^{n-1}\}$ being cyclic. Let $\zeta\in \mbf$ be a primitive root of $x^n-1=0$.

Since $\{1,\sigma,\sigma^2,\cdots,\sigma^{k-1}\}$ is linearly independent over $\mbf$, there exists $\alpha$ in $\mbk$ such that~\footnote{The element $(\alpha,\zeta)$ is called a Lagrange resolvent.}
$$(\alpha,\zeta):=\alpha+\zeta\sigma(\alpha)+\zeta^2\sigma^2(\alpha)+\cdots+\zeta^{k-1}\sigma^{k-1}(\alpha)\neq 0.$$

One has $\sigma(\alpha,\zeta)=\sigma(\alpha)+\zeta\sigma^2(\alpha)+\zeta^2\sigma^3(\alpha)+\cdots+\zeta^{k-1}\sigma^{k}(\alpha)=\zeta^{-1}(\alpha,\zeta)$.

For every $0<j<k$, we have $\sigma^j(\alpha,\zeta)=\zeta^{-j}(\alpha,\zeta)$. Hence $(\alpha,\zeta)\in \mbk$ cannot be fixed by any element in ${\rm Gal}(\mbk/\mbf)$ except the identity. Hence ${\rm Gal}(\mbk/\mbf((\alpha,\zeta)))=\{1\}$ which implies that $\mbk=\mbf((\alpha,\zeta))$. From $\sigma(\alpha,\zeta)=\zeta^{-1}(\alpha,\zeta)$, we know that $\sigma((\alpha,\zeta)^k)=\zeta^{-k}(\alpha,\zeta)^k=(\alpha,\zeta)^k$. That is, $(\alpha,\zeta)^k$ is in $\mbk^{{\rm Gal}(\mbk/\mbf)}=\mbf$. 

\end{proof}

\begin{definition}
  An extension $\mbk/\mbf$ is called \textbf{cyclic} if it is a Galois extension with the  Galois group being cyclic.~\index{cyclic extension}
\end{definition}

(4) Galois group of $\mbf_{p^n}/\mbf_p$:

We know that $\mbf_{p^n}$ is the splitting field of the separable polynomial $x^{p^n}-x$ in $\mbf_p$, hence is a Galois extension of $\mbf_p$. Also $[\mbf_{p^n}: \mbf_p]=n$.

\begin{theorem}
${\rm Gal}(\mbf_{p^n}/\mbf_p)\cong\mbz_n$.
\end{theorem}
\begin{proof}
We are going to prove that the Frobenius isomorphism $\sigma: \mbf_{p^n}\to \mbf_{p^n}$ defined by $\sigma(a)=a^p$ for every $a$ in $\mbf_{p^n}$ is of order $n$.
Hence  ${\rm Gal}(\mbf_{p^n}/\mbf_p)=\langle\sigma\rangle\cong \mbz_n$.

First $\sigma$ is an automorphism on $\mbf_{p^n}$ fixing $\mbf_p$ since $a^p=a$ for every $a$ in $\mbf_p$. Secondly $\sigma^j\neq 1$ for all $0\leq j<n$ otherwise all elements of $\mbf_{p^n}$ is a root of $x^{p^j}-x=0$. This leads to a contradiction since $x^{p^j}-x=0$ has at most $p^j$ roots and $\mbf_{p^n}$ has $p^n$ elements.

\end{proof}

By the fundamental theorem of Galois theory, we can classify subfields of $\mbf_{p^n}$ via classifying subgroups of $\mbz_n$.
\begin{corollary}
  A subfield of $\mbf_{p^n}$ is $\mbf_{p^d}$ for some $d|n$.
\end{corollary}
\begin{proof}
By  the fundamental theorem of Galois theory, every  subfield $\mbe$ of $\mbf_{p^n}$ is of the form $\mbf_{p^n}^H$ for some subgroup $H$ of $\mbz_n$ with $m=|H|$ dividing $n$. Hence
$$[\mbe:\mbf_p]=[{\rm Gal}(\mbf_{p^n}/\mbf_p):{\rm Gal}(\mbf_{p^n}/\mbe)]=d=\frac{n}{m}.$$
That is, $\mbe\cong \mbf_{p^d}$.
\end{proof}

\section{Galois' great theorem and polynomials solvable by radicals}

Via concrete examples, we explain how to describe solvability of a polynomial by radicals in terms of fields extensions.

Consider a quadratic  $f(x)=x^2+bx+c$ in $\mbc[x]$. Let $\mbf=\mbq(b,c)$. Then the splitting field of $f(x)$ is $\mbf(\sqrt{b^2-4c})$ and $\sqrt{b^2-4c}$ is a root of $x^2-(b^2-4c)$ in $\pf$. In another word, The splitting field of $f(x)$ is embedded into the tower $\mbf\subseteq \mbf(\sqrt{b^2-4c})$ with $\sqrt{b^2-4c}$ being a root of $x^2-(b^2-4c)$ in $\pf$.

Let $f(x)=x^3+qx+r$ be a cubic in $\mbc[x]$.~\footnote{Solving the cubic equation $y^3+ay^2+by+c=0$ amounts to solving the cubic equation  $x^3+qx+r=0$ by letting $y=x-\frac{a}{3}$.} Define $\mbf=\mbq(q,r)$. The roots of $f(x)$ are of the form $y+z$, $\omega y+\omega^2 z$ and $\omega^2 y+\omega z$ where $y^3=\frac{1}{2}(-r+\sqrt{r^2+\frac{4q^3}{27}})$, $z=\frac{-q}{3y}$ and $\omega=e^{\frac{2\pi i}{3}}$. Hence the splitting field of $f(x)$ is embedded into $\mathbb{B}_3$ in a tower
$$\mbf=\mathbb{B}_0\subseteq \mathbb{B}_1=\mathbb{B}_0(\omega)\subseteq \mathbb{B}_2=\mathbb{B}_1(\sqrt{r^2+\frac{4q^3}{27}})\subseteq \mathbb{B}_3=\mathbb{B}_2(\sqrt[3]{\frac{1}{2}(-r+\sqrt{r^2+\frac{4q^3}{27}})}).$$

The above motivate the following definitions.

\begin{definition}
An extension $\mbe$ of $\mbf$ is called a \textbf{radical extension}\index{radical extension} if there exists a tower of fields~(called a \textbf{ radical tower})\index{radical tower}
$\mbf=\mathbb{B}_0\subseteq \mathbb{B}_1\subseteq\cdots\subseteq \mathbb{B}_s=\mbe$ such that each $\mathbb{B}_{i+1}=\mathbb{B}_i(\sqrt[\leftroot{-1}\uproot{2}m_i]{a_i})$ with $a_i$ in $\mathbb{B}_i$ for some positive integer $m_i$. A polynomial $f(x)$ in $\pf$ is \textbf{solvable by radicals}\index{solvable by radicals} if $f(x)$ splits over a radical extension of $\mbf$.
\end{definition}

\begin{proposition}~\label{prop:GaloisRadicalTower}

Let $\mbf$ be a field of characteristic 0.

If $f(x)$ in $\pf$ is  solvable by radicals, then its splitting field is contained in a radical extension $\mbf=\mathbb{B}_0\subseteq \mathbb{B}_1\subseteq\cdots\subseteq \mathbb{B}_s=\mbe$ such that  $\mbe/\mbf$ is a Galois extension and $\mathbb{B}_{i+1}$ is a cyclic extension of $\mathbb{B}_i$ for every $i$.~\footnote{We call such  a tower \textbf{cyclic}.}~\index{cyclic tower}
\end{proposition}
\begin{proof}

Suppose $f(x)$ in $\pf$ is  solvable by radicals, then the splitting field $\mbk$ of $f(x)$ is embedded into a radical tower $\mbf=\mathbb{B}_0\subseteq \mathbb{B}_1\subseteq\cdots\subseteq \mathbb{B}_s$ with $\mathbb{B}_{i+1}=\mathbb{B}_i(\sqrt[\leftroot{-1}\uproot{2}m_i]{a_i})$ with $a_i$ in $\mathbb{B}_i$ for all $1\leq i\leq s-1$. Let $\mbf'$ be  the splitting field of  $x^n-1$ with $n$ being the least common multiple of $m_i's$.


One may assume that all fields above sit in a larger field so that compositions of these fields make sense.

Then we get a radical tower:
$$\mbf\subseteq \mbf\mbf'\subseteq \mathbb{B}_1\mbf'\subseteq \mathbb{B}_2\mbf'\subseteq\cdots\subseteq \mathbb{B}_s\mbf'=\mbe.$$

First of all one can check that $\mbe$ is the splitting field of $\displaystyle\prod_{i=1}^s (x^{m_i}-a_i)$.

Secondly by Theorem~\ref{thm: cyclicextension}, the tower $\mbf\mbf'\subseteq \mathbb{B}_1\mbf'\subseteq \mathbb{B}_2\mbf'\subseteq\cdots\subseteq \mathbb{B}_s\mbf'$ is already what we want.

At last by the fundamental theorem of Galois theory, one can insert intermediate fields between $\mbf$ and $\mbf\mbf'$ to get a cyclic tower. 
\end{proof}
We are now in front of one  central theorem in Galois theory.
\begin{theorem}(Galois' Great Theorem)
~\label{thm: GaloisGreat}

Let $\mbf$ be a field of characteristic 0.

Then a polynomial $f(x)$ in $\pf$ is solvable by radicals iff  the Galois group ${\rm Gal}(\mbk/\mbf)$ of its splitting field $\mbk$ is solvable.
\end{theorem}

\begin{proof}
Suppose $f(x)$ is solvable by radicals.

By Proposition~\ref{prop:GaloisRadicalTower}, there exists a radical tower
$\mbf=\mathbb{B}_0\subseteq \mathbb{B}_1\subseteq\cdots\subseteq \mathbb{B}_s=\mbe$ such that $\mbe$ is a Galois extension of $\mbf$ and every $\mathbb{B}_{i+1}$ is a Galois extension of $\mathbb{B}_i$ with ${\rm Gal}(\mathbb{B}_{i+1}/\mathbb{B}_i)$ cyclic.

By the fundamental theorem of Galois theory, one has ${\rm Gal}(\mbe/\mathbb{B}_0)\geq {\rm Gal}(\mbe/\mathbb{B}_1)\geq \cdots  \geq {\rm Gal}(\mbe/\mbe)=\{1\}$. Also each $\mathbb{B}_{i+1}$ is a Galois extension of $\mathbb{B}_i$, hence ${\rm Gal}(\mbe/\mathbb{B}_{i+1})$ is a normal subgroup of ${\rm Gal}(\mbe/\mathbb{B}_i)$ and ${\rm Gal}(\mbe/\mathbb{B}_i)/{\rm Gal}(\mbe/\mathbb{B}_{i+1})\cong {\rm Gal}(\mathbb{B}_{i+1}/\mathbb{B}_i)$ is cyclic. So ${\rm Gal}(\mbe/\mbf)$ is a solvable group. The Galois group of $f(x)$, ${\rm Gal}(\mbk/\mbf)\cong {\rm Gal}(\mbe/\mbf)/{\rm Gal}(\mbe/\mbk)$ is also solvable.

Now assume that $G={\rm Gal}(\mbk/\mbf)$  is solvable.  By Exercise~\ref{ex: Solvable}, one has $G_0=G\unlhd G_1\unlhd G_2\unlhd\cdots \unlhd G_s=\{1\}$ such that $G_i/G_{i+1}$ is cyclic. By the fundamental theorem of Galois theory, we have
$\mbf=\mbk^{G_0}\subseteq \mbk^{G_1}\subseteq \mbk^{G_2}\subseteq\cdots\subseteq \mbk^{G_s}=\mbk$. Let $\mathbb{B}_i=\mbk^{G_i}$ and note that ${\rm Gal}(\mathbb{B}_s/\mathbb{B}_i)={\rm Gal}(\mbk/\mbk^{G_i})=G_i$. For every $i$, since ${\rm Gal}(\mbk/\mathbb{B}_{i+1})=G_{i+1}$ is a normal subgroup of ${\rm Gal}(\mbk/\mathbb{B}_{i})=G_{i}$, one has that $\mathbb{B}_{i+1}$ is a Galois extension of $\mathbb{B}_i$ with ${\rm Gal}(\mathbb{B}_{i+1}/\mathbb{B}_i)$ being cyclic. By Theorem~\ref{thm: cyclicextension}, $\mbf=\mathbb{B}_0\subseteq \mathbb{B}_1\subseteq \cdots \subseteq \mathbb{B}_s=\mbk$ is a radical tower. So $f(x)$ is solvable by radicals.

\end{proof}

\begin{remark}
In Theorem~\ref{thm: GaloisGreat}, we assume that $\mbf$ is of characteristic 0 to guarantee that the characteristic of $\mbf$ does not divide $m_i$ for each $i$. Hence Theorem~\ref{thm: cyclicextension} applies in the proof.
\end{remark}
\begin{corollary}
Over  a field of characteristic 0,  all polynomials of degree less than 5 are solvable by radicals.
\end{corollary}
\begin{proof}
 Suppose that $f$ is in $\mbf[x]$ and ${\rm deg}f\leq 4$. The Galois group of $f$ is a subgroup of $S_4$, hence solvable.
\end{proof}

\begin{theorem}[Abel-Ruffini]\

There exists a quintic polynomial in $\mbq[x]$ which is not solvable by radicals.
\end{theorem}
\begin{proof}
Consider $f(x)=x^5-80x+5$ in $\mbq[x]$. Let $\mbk$ be the splitting field of $f$ and $G={\rm Gal}(\mbk/\mbq)$.

By Eisenstein's criterion $f(x)$ is irreducible. Let $\alpha$ be a root of $f(x)=0$. Then $[\mbq(\alpha):\mbq]=5$. Hence 5 divides $|G|=[\mbk:\mbq]$.

Since $f'(x)=5x^4-80$, $f(x)$ is decreasing in $[-2,2]$ and  increasing in $(-\infty, -2]$ and $[2,\infty)$. Note that $f(-2)>0$ and $f(2)<0$, so $f(x)$ has exactly 3 real roots. Therefore the complex conjugate of complex number is an element in $G$. This element is a transposition since it interchanges two complex roots and fixes the other roots.

So $G$  contains a 5-cycle and a transposition in $S_5$. Hence $G=S_5$ by Exercise~\ref{ex:S5}. Since $G$ is not solvable, by Galois' great theorem $f(x)$ is not solvable by radicals.
\end{proof}
\section{Ruler and compass problems}

\begin{definition}
We say that a complex number $z=x+iy$ or a point $(x,y)$ in the plane  is \textbf{constructible by straightedge and compass}\index{constructible by straightedge and compass} if it can be obtained by the following straightedge-compass operations:




\begin{itemize}
  \item connecting two given points by a straight line;
  \item finding an intersection point of two straight lines;
  \item drawing a circle with given center and radius;
  \item finding intersection points of a straight line and a circle or the intersection of two circles.
\end{itemize}

\end{definition}

 Via straightedge-compass operations, one can achieve the following:

\begin{itemize}
\item finding a line  through a given point $A$  parallel to a given line $l$;
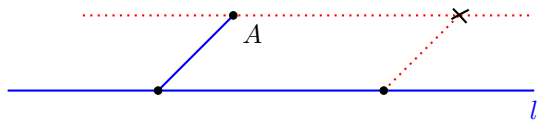
\begin{center}
\begin{tikzpicture}
\draw[thick,blue](-2,0)--(5,0) node[below] {$l$};
\draw [thick, blue](0,0) to (1,1);
\draw [dotted,thick, red](-1,1)--(5,1);
\draw[dotted,thick,red] (3,0)--(4,1);
\draw [thick] (4.1,0.9) arc [radius=5, start angle=42, end angle= 45];
\draw  [thick](3.9,0.95) arc [radius=5, start angle=120, end angle= 117];
\draw[fill] (0,0) circle [radius=0.05];
\draw[fill] (3,0) circle [radius=0.05];
\draw[fill] (1,1)node[below right]{$A$} circle [radius=0.05];
\end{tikzpicture}
\end{center}
\item finding a line  passing through a given point $A$ on a line $l$ and perpendicular to $l$;
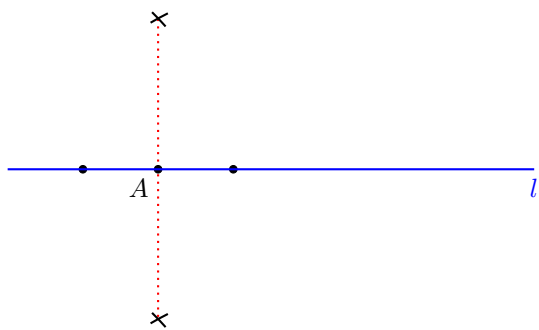
\begin{center}
\begin{tikzpicture}
\draw [dotted,thick, red](0,2)--(0,-2);
\draw [thick] (0.1,1.9) arc [radius=5, start angle=42, end angle= 45];
\draw [thick] (-0.1,1.95) arc [radius=5, start angle=120, end angle= 117];
\draw [thick] (0.1,-2.1) arc [radius=5, start angle=42, end angle= 45];
\draw [thick] (-0.1,-2.05) arc [radius=5, start angle=120, end angle= 117];
\draw[fill] (0,0) node[below left]{$A$}circle [radius=0.05];
\draw[fill] (-1,0) circle [radius=0.05];
\draw[fill] (1,0) circle [radius=0.05];
\draw[thick,blue](-2,0)--(5,0) node[below] {$l$};
\end{tikzpicture}
\end{center}
\item finding the middle point of a line segment $AB$.
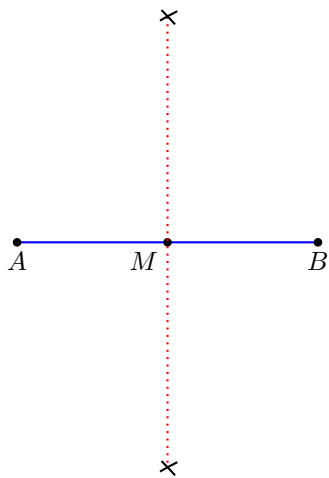
\begin{center}
\begin{tikzpicture}
\draw [thick, blue](0,0)--(4,0);
\draw[fill] (0,0) node[below]{$A$}circle [radius=0.05];
\draw[fill] (4,0) node[below]{$B$}circle [radius=0.05];
\draw[fill] (2,0) node[below left]{$M$}circle [radius=0.05];
\draw[dotted,thick,red](2,3)--(2,-3);
\draw [thick] (2.1,2.9) arc [radius=5, start angle=42, end angle= 45];
\draw [thick] (1.9,2.95) arc [radius=5, start angle=120, end angle= 117];
\draw [thick] (2.1,-3.1) arc [radius=5, start angle=42, end angle= 45];
\draw [thick] (1.9,-3.05) arc [radius=5, start angle=120, end angle= 117];
\end{tikzpicture}
\end{center}
\end{itemize}

Set a fixed radius by 1.

From above, we have the following observations.

\begin{remarks}
\begin{enumerate}
\item By a straightedge and a compass, one can build a rectangular coordinate  system with rational coordinates.
\item The point $(x,y)$ is constructible iff $(x,0)$ and $(0,y)$ is constructible.
\end{enumerate}
\end{remarks}

\begin{definition}
A real number  $x$ is called \textbf{constructible} if $(x,0)$ is constructible, or equivalently, $(0,x)$ is constructible by straightedge and compass.~\index{constructible}
\end{definition}

If $a, b$ in $\mbr$ are constructible, then $a\pm b$ are also constructible.

If positive numbers $a, b$ are constructible, then $ab$ and $\frac{a}{b}$ are constructible as indicated by the following figures:

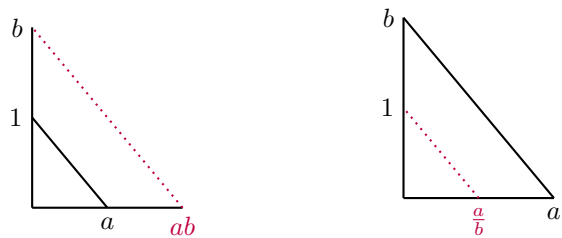
\begin{center}
\begin{tikzpicture}
\draw [thick](0,0)--(2,0);
\draw [thick](0,0)--(0,2.4);
\draw [thick](1,0)node[below]{$a$}--(0,1.2)node[left]{1};
\draw [thick,dotted, purple](2,0)--(0,2.4);
\node[left] at (0,2.4){$b$};
\node[below, purple] at (2,0) {$ab$};
\end{tikzpicture}
\hspace{2cm}
\begin{tikzpicture}
\draw [thick](6,0)--(8,0);
\draw [thick](6,0)--(6,2.4);
\draw [thick,dotted,purple](7,0)node[below]{$\frac{a}{b}$}--(6,1.2);
\node[left] at (6,1.2) {1};
\draw [thick](8,0)node[below]{$a$}--(6,2.4);
\node[left] at (6,2.4){$b$};
\end{tikzpicture}
\end{center}

So all constructible real numbers form a field $\mbe$ containing $\mbq$.

Suppose we have two constructible points  in  $\mbf^2$ for a field $\mbf$ consisting of constructible numbers. Connecting these two points   gives a straight line $ax+by+c=0$ in $\mbr^2$ with $a,b,c\in \mbf$. We call it a constructible line over $\mbf$.
\begin{theorem}
~\label{thm: StraightedgeCompass}

Suppose $\mbf$ is a field of real numbers and  $[\mbf:\mbq]<\infty$. Then $\mbf$  is a subfield of $\mbe$ iff $[\mbf:\mbq]=2^m$ for some nonnegative integer $m$.
\end{theorem}

\begin{proof}

Suppose a finite extension $\mbf$ of $\mbq$ is a subfield of $\mbe$. 

To prove $[\mbf:\mbq]=2^m$, it suffices to prove that  $[\mbq(\alpha):\mbq]$ is a power of 2 for every $\alpha$ in $\mbf$. Note that $\alpha$ is obtained from $\mbq$ after finitely many straightedge and compass operations. So it's enough to show that if a real number $\alpha$ is obtained from a subfield $\mbf'$ of $\mbf$ by one straightedge and compass operation, then $[\mbf'(\alpha):\mbf']=1$ or 2.

If $(\alpha,\beta)$ is the intersection point of two constructible lines over $\mbf'$. Then $\alpha$ is a root of a polynomial of degree 1 in $\mbf'[x]$, which means $\mbf'(\alpha)=\mbf'$.

If $(\alpha,\beta)$ is an intersection point of a circle $(x-x_0)^2+(y-y_0)^2=r^2$ and a straight line $ax+by+c=0$ with $x_0, y_0, r, a,b,c\in \mbf'$. Then $\alpha$ and $\beta$  are roots of an equation in $\mbf'[x]$ of degree at most 2. Hence $[\mbf'(\alpha):\mbf']=1$ or 2.

Similarly if  $(\alpha,\beta)$ is an intersection point of two circles whose center's coordinates and radius are in $\mbf'$. Combining the equations of the two circles, we see that $(\alpha,\beta)$ is an intersection point of a circle $(x-x_0)^2+(y-y_0)^2=r^2$ and a straight line $ax+by+c=0$ with $x_0, y_0, r, a,b,c\in \mbf'$. Again $[\mbf'(\alpha):\mbf']=1$ or 2.


Conversely we prove the following statement:

If $\mbf'$ is a field of constructible numbers and  a  field $\mbf$ of real numbers is a finite extension of $\mbq$ such that $[\mbf:\mbf']=2^m$, then $\mbf$ also consists of constructible numbers. 

We do induction on $m$.

The following figure shows that if  a positive number $a$ is constructible, then $\sqrt{a}$ is constructible. Hence the statement holds when $m=1$.


\begin{center}
\begin{tikzpicture}

\draw[thick, purple] (0,0) arc(180:0:3);
\draw [thick](0,0)--(6,0);
\draw [thick](4,0)--(4,2.835);
\draw [thick](0,0)--(4,2.835);
\draw [thick](6,0)--(4,2.835);
\node[left] at (4,1.418){$\sqrt{a}$};
\node[below] at (2,0){$a$};
\node[below] at (5,0){$1$};
\end{tikzpicture}
\end{center}

Assume that the statement holds when  $m<n$. 

Now suppose that $[\mbf:\mbf']=2^n$.  Take $\alpha$ in $\mbf\setminus\mbf'$. If $[\mbf'(\alpha):\mbf']=2^n$, then $\{1,\alpha,\cdots,\alpha^{2^n-1}\}$ is a basis of $\mbf$ over $\mbf'$. Since $\mbf'(\alpha):\mbf'(\alpha^2)]=2$, we have that  $[\mbf'(\alpha^2):\mbf']=2^{n-1}$. Anyway there exists $\beta$ in $\mbf$ such that $[\mbf'(\beta):\mbf']=2^k<2^n$. By assumption, $\mbf'(\beta)$ consists of constructible numbers. Moreover $[\mbf: \mbf'(\beta)]=2^{n-k}<2^n$. Hence $\mbf$ consists of constructible real numbers.

\end{proof}

\begin{remark}
Theorem~\ref{thm: StraightedgeCompass} shows that a real number $\alpha$ is constructible iff $[\mbq(\alpha):\mbq]=2^m$.
\end{remark}

So the field of constructible numbers contains many numbers other than rationales. But anyway it is a field consisting of algebraic numbers. Moreover by straightedge and compass it's impossible to construct any of the following:
\begin{enumerate}
  \item a square  whose area is $\pi$;
  \item  a cube whose volume is 2;
  \item trisect $\theta=\frac{\pi}{3}$,
\end{enumerate}
since
\begin{enumerate}
  \item The number $\pi$ is \textbf{transcendental}~(not algebraic over $\mbq$).~\index{transcendental} ~\cite{Baker1990}.
  \item The number $\sqrt[3]{2}$ is not constructible since $[\mbq(\sqrt[3]{2}):\mbq]=3$.
  \item An angle $\theta$ can be trisected iff $\cos\frac{\theta}{3}$ is constructible. But the minimal polynomial of $\cos\frac{\pi}{9}$ is $8x^3-6x-1$~\footnote{Recall that $\cos3\theta=4\cos^3\theta-3\cos\theta$ and let $\cos3\theta=\frac{1}{2}$.}, which means, $[\mbq(\cos\frac{\pi}{9}):\mbq]=3$.
\end{enumerate}

The integer $F_k=2^{2^k}+1$ is called the $k$-th \textbf{Fermat prime}. Among Fermat numbers, so far only 5 primes: $F_0,\cdots,F_4$ are found  though it is conjectured that there are infinitely many.~\index{Fermat number}

\begin{theorem}
  A regular $n$-gon is constructible iff $n=2^mp_1\cdots p_k$ where $p_i's$ are Fermat primes.
\end{theorem}
\begin{proof}
 The regular $n$-gon is constructible iff the primitive root $\zeta_n$ of $x^n-1$ is constructible. By Theorem~\ref{thm: StraightedgeCompass}, the number $\zeta_n$ is constructible iff $\varphi(n)=[\mbq(\zeta_n):\mbq]=2^l$ for some nonnegative integer $l$. And $\varphi(n)=2^l$ iff  $n=2^mp_1\cdots p_k$ for $p_i's$ being Fermat primes.
\end{proof}

\section*{Exercises}
\begin{exercise}
Prove that $\txau(\mbk/\mbf)$ is a finite group when $\mbk/\mbf$ is a finite field extension.
\end{exercise}

\begin{exercise}
  Let $\mbk$ be a Galois extension of $\mbf$. Define $${\rm Tr}_{\mbk/\mbf}(\alpha)=\sum_{\sigma\in {\rm Gal}(\mbk/\mbf)}\sigma(\alpha)$$ for every $\alpha$ in $\mbk$.

\begin{enumerate}
  \item Prove that ${\rm Tr}_{\mbk/\mbf}(\alpha)$ is in $\mbf$ for all $\alpha$ in $\mbk$.
  \item Prove that ${\rm Tr}_{\mbk/\mbf}(\alpha+\beta)={\rm Tr}_{\mbk/\mbf}(\alpha)+{\rm Tr}_{\mbk/\mbf}(\beta)$.
   \item Let $\mbk=\mbf(\sqrt{D})$ be a quadratic extension of $\mbf$. Show that  ${\rm Tr}_{\mbk/\mbf}(a+b\sqrt{D})=2a$ for all $a,b\in \mbf$.
   \item Let $m_{\alpha,\mbf}(x)=x^d+a_{d-1}x^{d-1}+\cdots+a_1x+a_0$ in $\pf$ be the minimal polynomial of $\alpha$. Prove that ${\rm Tr}_{\mbk/\mbf}(\alpha)=-\frac{n}{d}a_{d-1}$.
\end{enumerate}
\end{exercise}

\begin{exercise}
  Suppose $\mbk$ be a Galois extension of $\mbf$. Define
  $${\rm N}_{\mbk/\mbf}(\alpha)=\prod_{\sigma\in{\rm Gal}(\mbk/\mbf)}\sigma(\alpha).$$
  \begin{enumerate}
    \item Prove that ${\rm N}_{\mbk/\mbf}(\alpha)$ is in $\mbf$.
    \item  Prove that ${\rm N}_{\mbk/\mbf}(\alpha\beta)={\rm N}_{\mbk/\mbf}(\alpha){\rm N}_{\mbk/\mbf}(\beta)$.
    \item  Let $\mbk=\mbf(\sqrt{D})$ be a quadratic extension of $\mbf$. Show that ${\rm N}_{\mbk/\mbf}(a+b\sqrt{D})=a^2-Db^2$.
    \item Let $m_{\alpha,\mbf}(x)=x^d+a_{d-1}x^{d-1}+\cdots+a_1x+a_0$ in $\pf$ be the minimal polynomial of $\alpha$. Prove that ${\rm N}_{\mbk/\mbf}(\alpha)=(-1)^da_0$.
    \end{enumerate}
\end{exercise}

\begin{exercise}
  Suppose $\mbk$ be a Galois extension of $\mbf$ and $\sigma$ is in ${\rm Gal}(\mbk/\mbf)$.
  \begin{enumerate}
    \item Suppose $\alpha=\frac{\beta}{\sigma(\beta)}$ for some nonzero $\beta$ in $\mbk$. Prove that ${\rm N}_{\mbk/\mbf}(\alpha)=1$.
    \item Suppose $\alpha=\beta-\sigma(\beta)$ for some $\beta$ in $\mbk$. Prove that ${\rm Tr}_{\mbk/\mbf}(\alpha)=0$.
  \end{enumerate}
\end{exercise}

\begin{exercise}
let $\mbk$ be a Galois extension of $\mbf$ such that ${\rm Gal}(\mbk/\mbf)$ is cyclic of order n generated by $\sigma$. Suppose $\alpha$ in $\mbk$ satisfies that ${\rm N}_{\mbk/\mbf}(\alpha)=1$. Prove that $\alpha=\frac{\beta}{\sigma(\beta)}$ for some nonzero $\beta$ in $\mbk$.
\end{exercise}

\begin{exercise}
let $\mbk$ is a Galois extension of $\mbe$  and $\mbe$ is a Galois extension of $\mbf$. Whether or not  $\mbk$ is a Galois extension of $\mbf$?
\end{exercise}

 \begin{exercise}
   Suppose $\mbk/\mbf$ is a Galois extension and $\mbe$ is an intermediate field between $\mbf$ and $\mbk$. Prove that $\mbe/\mbf$ is a Galois extension iff $\sigma(\mbe)=\mbe$ for all $\sigma$ in ${\rm Gal}(\mbk/\mbf)$.
 \end{exercise}

\begin{exercise}
  Let $\zeta$ be a $p^{th}$ root of unity for a prime $p$. Prove that ${\rm Tr}_{\mbq(\zeta_p)/\mbq}(\zeta)$ is $-1$ or $p-1$ depending on whether or not $\zeta$ is a primitive $p^{th}$ root of unity.
\end{exercise}

\begin{exercise}
  Let $p,q,r$ be primes and $q\neq r$. Let $\zeta$ be any root of $x^p-q$ and $\eta$ be any root of $x^p-r$. Prove that $\mbq(\zeta)\neq\mbq(\eta)$.
\end{exercise}

\begin{exercise}
  Compute ${\rm Aut}(\mbq(\sqrt{3+\sqrt{3}})/\mbq)$ and decide whether it is a Galois extension.
\end{exercise}

\begin{exercise}
Compute the Galois group of $x^4+2$, and find all subgroups and corresponding fixed fields.
\end{exercise}

\begin{exercise}
Determine the Galois group of $x^4-14x^2+9$, and find all subgroups and corresponding fixed fields.
\end{exercise}

\begin{exercise}
  Describe  the Galois group of $f(x)=x^4+x+1$ in $\mbf_2[x]$ over $\mbf_2$. Find all intermediate fields of the splitting field of $f$ over $\mbf_2$.
\end{exercise}

\begin{exercise}
  Find a radical tower for the extension $\mbq(\sqrt[6]{1+\sqrt[4]{3}},\sqrt[4]{2})/\mbq$.
\end{exercise}

\begin{exercise}
Prove that $x^5-4x+2$ is unsolvable.
\end{exercise}

\begin{exercise}
Prove that $\cos\theta$ is constructible iff $\sin\theta$ is constructible.
\end{exercise}

\begin{exercise}
Prove that a regular n-gon is constructible iff $\cos\frac{2\pi}{n}$ is constructible.
\end{exercise}

\begin{exercise}
Determine whether or not the following regular $n$-gon's are constructible:
\begin{enumerate}
  \item $n=9$.
  \item $n=5$.
  \item $n=7$.
 \end{enumerate}

\end{exercise}



\backmatter
\printindex

\bibliography{}

\end{document}